\documentclass[a4paper]{amsart}
\usepackage{amssymb}
\usepackage[hidelinks]{hyperref}
\usepackage{tikz}
\usetikzlibrary{arrows,arrows.meta,calc,matrix,shapes,cd,patterns,decorations.pathreplacing}

\newenvironment{pic}[1][]
{\begin{aligned}\begin{tikzpicture}[#1]}
{\end{tikzpicture}\end{aligned}}

\def\strarr{length=2pt, width=3pt}
\tikzset{arrow/.style={decoration={
    markings,
    mark=at position #1 with \arrow{>[\strarr]}},
    postaction=decorate},
    reverse arrow/.style={decoration={
    markings,
    mark=at position #1 with {{\arrow{<[\strarr]}}}},
    postaction=decorate}
}

\tikzstyle{dot}=[circle, draw=black, fill=white, inner sep=.4ex, on layer=foreground]
\tikzstyle{blackdot}=[dot, fill=black!50]
\tikzstyle{whitedot}=[dot, fill=white]

\newif\ifvflip\pgfkeys{/tikz/vflip/.is if=vflip}
\newif\ifhflip\pgfkeys{/tikz/hflip/.is if=hflip}
\newif\ifhvflip\pgfkeys{/tikz/hvflip/.is if=hvflip}
\newlength\morphismheight
\newlength\wedgewidth
\tikzset{width/.initial=1mm}
\makeatletter
\tikzstyle{morphism}=[font=\small,morphismshape]
\pgfdeclareshape{morphismshape}
{
    \savedanchor\centerpoint
    {
        \pgf@x=0pt
        \pgf@y=0pt
    }
    \anchor{center}{\centerpoint}
    \anchorborder{\centerpoint}
    \saveddimen\overallwidth
    {
        \pgfkeysgetvalue{/tikz/width}{\minwidth}
        \pgf@x=\wd\pgfnodeparttextbox
        \ifdim\pgf@x<\minwidth
            \pgf@x=\minwidth
        \fi
    }
    \savedanchor{\upperrightcorner}
    {
        \pgf@y=.5\ht\pgfnodeparttextbox
        \advance\pgf@y by -.5\dp\pgfnodeparttextbox
        \pgf@x=.5\wd\pgfnodeparttextbox
    }
    \anchor{north}
    {
        \pgf@x=0pt
        \pgf@y=0.5\morphismheight
    }
    \anchor{north east}
    {
        \pgf@x=\overallwidth
        \multiply \pgf@x by 2
        \divide \pgf@x by 5
        \pgf@y=0.5\morphismheight
    }
    \anchor{east}
    {
        \pgf@x=\overallwidth
        \divide \pgf@x by 2
        \advance \pgf@x by 5pt
        \pgf@y=0pt
    }
    \anchor{west}
    {
        \pgf@x=-\overallwidth
        \divide \pgf@x by 2
        \advance \pgf@x by -5pt
        \pgf@y=0pt
    }
    \anchor{north west}
    {
        \pgf@x=-\overallwidth
        \multiply \pgf@x by 2
        \divide \pgf@x by 5
        \pgf@y=0.5\morphismheight
    }
    \anchor{south east}
    {
        \pgf@x=\overallwidth
        \multiply \pgf@x by 2
        \divide \pgf@x by 5
        \pgf@y=-0.5\morphismheight
    }
    \anchor{south west}
    {
        \pgf@x=-\overallwidth
        \multiply \pgf@x by 2
        \divide \pgf@x by 5
        \pgf@y=-0.5\morphismheight
    }
    \anchor{south}
    {
        \pgf@x=0pt
        \pgf@y=-0.5\morphismheight
    }
    \anchor{text}
    {
        \upperrightcorner
        \pgf@x=-\pgf@x
        \pgf@y=-\pgf@y
    }
    \backgroundpath
    {
    \begin{scope}
        \pgfkeysgetvalue{/tikz/fill}{\morphismfill}
        \pgfsetstrokecolor{black}
        \pgfsetlinewidth{.7pt}
        \begin{scope}
        \pgfsetstrokecolor{black}
        \pgfsetfillcolor{white}
        \pgfsetlinewidth{.7pt}
                \ifhflip
                    \pgftransformyscale{-1}
                \fi
                \ifvflip
                    \pgftransformxscale{-1}
                \fi
                \ifhvflip
                    \pgftransformxscale{-1}
                    \pgftransformyscale{-1}
                \fi
                \pgfpathmoveto{\pgfpoint
                    {-0.5*\overallwidth-5pt}
                    {0.5*\morphismheight}}
                \pgfpathlineto{\pgfpoint
                    {0.5*\overallwidth+5pt}
                    {0.5*\morphismheight}}
                \pgfpathlineto{\pgfpoint
                    {0.5*\overallwidth + \wedgewidth}
                    {-0.5*\morphismheight}}
                \pgfpathlineto{\pgfpoint
                    {-0.5*\overallwidth-5pt}
                    {-0.5*\morphismheight}}
                \pgfpathclose
                \pgfusepath{fill,stroke}
        \end{scope}
    \end{scope}
    }
}

\pgfkeys{%
  /tikz/on layer/.code={
    \pgfonlayer{#1}\begingroup
    \aftergroup\endpgfonlayer
    \aftergroup\endgroup
  },
  /tikz/node on layer/.code={
    \gdef\node@@on@layer{%
      \setbox\tikz@tempbox=\hbox\bgroup\pgfonlayer{#1}\unhbox\tikz@tempbox\endpgfonlayer\pgfsetlinewidth{\thickness}\egroup}
    \aftergroup\node@on@layer
  },
  /tikz/end node on layer/.code={
    \endpgfonlayer\endgroup\endgroup
  }
}
\def\node@on@layer{\aftergroup\node@@on@layer}
\makeatother

\pgfdeclarelayer{foreground}
\pgfdeclarelayer{background}
\pgfdeclarelayer{morphismlayer}
\pgfdeclarelayer{edgelayer}
\pgfdeclarelayer{nodelayer}
\pgfsetlayers{background,main,morphismlayer,foreground,edgelayer,nodelayer}

\newcommand{\tinycounit}[1][dot]{
\smash{\raisebox{-1pt}{\ensuremath{\hspace{0pt}\begin{pic}[scale=0.33]
        \node (0) at (0,0) {};
        \node (1) at (0,1) {};
        \node[#1, inner sep=1.5pt] (d) at (0,0.55) {};
        \draw (0.center) to (d.south);
    \end{pic}
    \hspace{-1pt}}}}}

\newcommand{\tinyid}{
\smash{\raisebox{-1pt}{\ensuremath{\hspace{0pt}\begin{pic}[scale=0.33]
        \draw (0,0) to (0,1.25);
    \end{pic}
    \hspace{-1pt}}}}}

\newcommand{\tinycomult}[1][dot]{
\smash{\raisebox{-1pt}{\hspace{-2pt}\ensuremath{\begin{pic}[scale=0.33]
    \node (0) at (0,0) {};
    \node[#1, inner sep=1.5pt] (1) at (0,0.55) {};
    \node (2) at (-0.5,1) {};
    \node (3) at (0.5,1) {};
    \draw (0.center) to (1.south);
    \draw (1.west) to [out=left, in=down, out looseness=1.3] (2.center);
    \draw (1.east) to [out=right, in=down, out looseness=1.3] (3.center);
\end{pic}
}\hspace{-3pt}}}}

\tikzset{halo/.style={
         preaction={draw,white,line width=4pt,-},
         preaction={draw,white,ultra thick, shorten >=-2.5\pgflinewidth}}}

\newcommand{\halfbraiding}[2]{
  \draw[-] (#1+.025,#2+.08) to (#1-.15,#2);
  \draw[-] (#1+.15,#2) to (#1-.025,#2-.08);
}

\usepackage{color}
\newcommand{\changed}[1]{#1}

\theoremstyle{plain}
\newtheorem{theorem}{Theorem}[section]
\newtheorem{proposition}[theorem]{Proposition}
\newtheorem{lemma}[theorem]{Lemma}
\newtheorem{corollary}[theorem]{Corollary}

\theoremstyle{definition}
\newtheorem{definition}[theorem]{Definition}
\newtheorem{example}[theorem]{Example}

\newtheorem*{theorem*}{Theorem}
\newtheorem*{corollary*}{Corollary}

\numberwithin{equation}{section}

\DeclareMathOperator{\ZI}{ZI}
\DeclareMathOperator{\ISub}{ISub}
\DeclareMathOperator{\Spec}{Spec}
\DeclareMathOperator{\Idl}{Idl}
\DeclareMathOperator{\Tr}{Tr}
\DeclareMathOperator{\colim}{colim}
\newcommand{\cat}[1]{\ensuremath{\mathbf{#1}}}
\newcommand{\op}{\ensuremath{^{\mathrm{op}}}}
\newcommand{\id}[1][]{\ensuremath{\mathrm{id}_{#1}}}
\newcommand{\restrict}[1]{\ensuremath{\|_{#1}}}
\newcommand{\Restrict}[1]{\ensuremath{\|^{#1}}}
\newcommand{\downset}{\ensuremath{\mathop{\downarrow}\!}}
\newcommand{\tuple}[1]{\left\langle#1\right\rangle}

\begin{document}
\title{Sheaf representation of monoidal categories}
\author[R.S.\ Barbosa]{Rui Soares Barbosa}
\address{INL -- International Iberian Nanotechnology Laboratory, Braga, Portugal}
\email{rui.soaresbarbosa@inl.int}
\author[C.\ Heunen]{Chris Heunen}
\address{University of Edinburgh, United Kingdom}
\email{chris.heunen@ed.ac.uk}
\date{\today}
\begin{abstract}
  Every small monoidal category with universal \changed{finite} joins of central idempotents is monoidally equivalent to the category of global sections of a sheaf of local monoidal categories on a topological space. 
  Every small stiff monoidal category monoidally embeds into such a category of global sections.
  \changed{An infinitary version of these theorems also holds in the spatial case.}
  These representation results are functorial and subsume the Lambek--Moerdijk--Awodey sheaf representation for toposes, the Stone representation of Boolean algebras, and the Takahashi representation of Hilbert modules as continuous fields of Hilbert spaces.
  Many properties of a monoidal category carry over to the stalks of its sheaf, including having a trace, having exponential objects, having dual objects, having limits of some shape, and the central idempotents forming a Boolean algebra.
\end{abstract}
\maketitle

\section{Introduction}

Representation theorems make abstract structures easier to work with by showing that they always have a more concrete form.
For example, the category $\cat{Vect}$ of vector spaces is easy in the sense that it is not a product of other categories, whereas the category $\cat{Vect} \times \cat{Vect}$ is also monoidal, but in a more complicated way.
The goal of this article is to prove that any monoidal category embeds into a product of easier ones. 
In fact we prove something stronger: nice enough monoidal categories are always equivalent to a dependent product of easy ones. 

To explain what we mean by ``easy'' categories, ``dependent'' products, and ``nice enough'' monoidal categories, consider again $\prod_{i \in \{0,1\}} \cat{Vect}$. Its decomposable nature can be detected by its \emph{central idempotents}, which correspond to the open sets of the discrete topological space $\{0,1\}$. We consider a monoidal category ``nice'' when its central idempotents are respected by tensor products, and have the structure of a \changed{distributive lattice} that is respected by tensor products; in the former case we call the category \emph{stiff} and in the latter we say that the category has \emph{universal finite joins} of central idempotents.
We consider a category ``easy'' when its central idempotents are \emph{\changed{$\vee$-local}} in that any finite cover already contains the covered element;
topologically this means there is a single focal point that every net converges to; logically this is the \changed{disjunction property} that if a finite disjunction holds then one of the disjoints holds. 

\changed{
In addition to the above finitary readings of ``nice'' and ``easy'', we also prove an infinitary version.
Here, the original category is ``nicer'' in that central idempotents form a frame whose structure is respected by tensor products, called having \emph{universal joins}, and that frame is spatial.
At the same time the constituent categories are ``easier'' in that its central idempotents are \emph{$\bigvee$-local}\footnote{\changed{What we call $\vee$-local here is sometimes called ``local''~\cite{awodey:sheafrepresentations,lambekmoerdijk:sheafrepresentations,lambekscott:categoricallogic} or ``sublocal''~\cite{awodey:sheafrepresentationsinlogic} in the literature. What we call $\bigvee$-local here remains unnamed in the literature, but see also footnote~\ref{footnote:hyperlocal}.}}, meaning that every infinite cover already contains the covered element.
}

Finally, by a ``dependent'' product we mean that the fibre $\cat{Vect}$ does not have to be constant but can vary continuously with the index $i$, that is, the category consists of global sections of a sheaf. With this terminology, made more precise below, here are our main results.

\begin{theorem*}
  Any small monoidal category with universal finite joins of central idempotents is monoidally equivalent to the category of global sections of a sheaf of $\vee$-local categories.

  \changed{
	  Any small monoidal category with universal joins of central idempotents forming a spatial frame is monoidally equivalent to the category of global sections of a sheaf of $\bigvee$-local categories.
  }
\end{theorem*}

\begin{corollary*}
  Any small stiff monoidal category monoidally embeds into a category of global sections of a sheaf of \changed{$\vee$-}local categories, \changed{and into a product of $\vee$-local monoidal categories.}
\end{corollary*}

This subsumes sheaf representation theorems for toposes~\cite{awodey:sheafrepresentationsinlogic,lambekmoerdijk:sheafrepresentations,lambekscott:categoricallogic,maclanemoerdijk:sheaves}\footnote{\label{footnote:hyperlocal}\changed{Our representation does not directly subsume~\cite{awodey:sheafrepresentations}, where the terminal object in each stalk is additionally projective. This property, sometimes called ``hyperlocal''~\cite{awodey:sheafrepresentationsinlogic} or ``local''~\cite{awodey:sheafrepresentations,lambek:world}, logically corresponds to the existence property. Such an extension is less clear in our monoidal case, and is left to future work.}}, where the lattice of central idempotents corresponds to that of elements of the subobject classifier.
In fact, the simple main insight underlying this work is that in cartesian categories, subterminal objects can be characterised entirely algebraically as central idempotents. This improves on earlier work~\cite{enriquemolinerheunentull:tensortopology}, which focused on the special case of central idempotents called subunits.
From a logical point of view, it extends the sheaf representation of (topos) categorical models of higher-order intuitionistic logic\changed{~\cite{aneljoyal:topologie,lambek:possibleworlds,lurie:highertopos}} to (symmetric monoidal closed) categorical models of multiplicative linear logic.
An additional improvement over results from the literature is that our proof is entirely concrete and avoids stacks.
Furthermore, by virtue of being monoidal, our results are not just analogous to, but directly capture, sheaf representation theorems for frames and modules~\cite{johnstone:stonespaces,daunshofmann:sections}.
The representation theorem may also be regarded as a structure theorem for higher-dimensional algebra~\cite{baez:2hilbertspaces,kapranovvoevodsky:2vectorspaces}, because 2-vector spaces may be seen as sheaves of vector spaces~\cite[Appendix~A]{heunenreyes:frobenius}.
Finally, our results find applications in physics and concurrent (probabilistic) computing~\cite{constantindicaireheunen:localisablemonads}, where a monoidal category models a process theory~\cite{heunenvicary:cqm} and its central idempotents may be interpreted as an underlying causal structure~\cite{enriquemolinerheunentull:space}.

We proceed as follows. Section~\ref{sec:centralidempotents} defines central idempotents and gives examples. Basic properties of central idempotents that make the category ``nice'' are discussed in Section~\ref{sec:universaljoins}. In Section~\ref{sec:basespace} we construct the topological space on which the structure sheaf will be based. Section~\ref{sec:presheaf} details the presheaf structure, and Section~\ref{sec:sheaf} establishes the sheaf condition. Next, Section~\ref{sec:stalks} investigates the stalks, Section~\ref{sec:local} shows that they are local, \changed{and Section~\ref{sec:inflocal} extends to case of infinitary joins}, finishing the proof of the Theorem above.
The sheaf representation is deepened further in Section~\ref{sec:preservation}, by showing that it preserves being Boolean, having limits, being closed, being compact, having a trace, \changed{and satisfying the external axiom of choice}. Section~\ref{sec:examples} works out several examples 
\changed{ to which sheaf representations for toposes do not apply directly, including recovering the Stone representation of  Boolean algebras regarded as posetal categories} and the Takahashi representation of Hilbert modules.
Section~\ref{sec:functoriality} settles functoriality of the main construction, after which Section~\ref{sec:embedding} proves the Corollary above.
Finally, in Section~\ref{sec:conclusion} we discuss several open questions that may be attacked using the representation theorem.
Appendix~\ref{sec:subunits} compares central idempotents with the special case of subunits~\cite{enriquemolinerheunentull:tensortopology}.

\section{Central idempotents}\label{sec:centralidempotents}

This section introduces central idempotents, gives examples, and discusses basic properties. 
If $\cat{C}$ is a monoidal category, \changed{$I$ its tensor unit, and we write $\lambda_U \colon I \otimes U \to U$ and $\rho_U \colon U \otimes I \to U$ for the unitors}, then the slice category $\cat{C}\slash I$ is again monoidal: the (terminal) tensor unit is the identity $I \to I$, the tensor product of objects $u \colon U \to I$ and $v \colon V \to I$ is $\rho_I \circ (u \otimes v) \colon U \otimes V \to I$, and the tensor product of morphisms is as in $\cat{C}$. An object $u \colon U \to I$ in $\cat{C}\slash I$ may be idempotent in the sense that $u \simeq u \otimes u$. We are interested in objects in $\cat{C}\slash I$ that are idempotent in a canonical way.

\begin{definition}
  A morphism $u \colon U \to I$ in a monoidal category is a \emph{left idempotent} when $\lambda_U \circ (u \otimes U) \colon U \otimes U \to U$ is invertible, and a \emph{right idempotent} when $\rho_U \circ (U \otimes u) \colon U \otimes U \to U$ is invertible.
\end{definition}

We will freely use the graphical calculus for monoidal categories~\cite{heunenvicary:cqm}, and draw a left idempotent $u \colon U \to I$ as:
\[\begin{pic}
  \node[dot] (d) {};
  \draw (d) to ++(0,-.5) node[below]{$U$};
\end{pic}\]
For a left idempotent we draw the inverse of $\lambda_U \circ (u \otimes U) \colon U \otimes U \to U$ as:
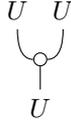
\[\begin{pic}
  \node[dot] (d) at (0,0) {};
  \draw (d) to ++(0,-.4) node[below]{$U$};
  \draw (d) to[out=0,in=-90] ++(.3,.4) node[above]{$U$};
  \draw (d) to[out=180,in=-90] ++(-.3,.4) node[above]{$U$};
\end{pic}\]

Recall that the \emph{centre} of a monoidal category $\cat{C}$ has as objects $U \in \cat{C}$ equipped with a \emph{half-braiding} (see~\cite[XIII.4]{kassel:quantumgroups} and~\cite{joyalstreet:yangbaxter}): a natural transformation $\sigma_A \colon U \otimes A \to A \otimes U$ such that $\sigma_{A \otimes B} = (A \otimes \sigma_B) \circ (\sigma_A \otimes B)$. A morphism $(U,\sigma) \to (V,\tau)$ in the centre is a morphism $f \colon U \to V$ in $\cat{C}$ satisfying $(A \otimes f) \circ \sigma_A = \tau_A \circ (f \otimes A)$. The tensor unit $I$ of a monoidal category always carries a half-braiding $\rho_A^{-1} \circ \lambda_A \colon I \otimes A \to A \otimes I$. We draw $\sigma_A$ as:
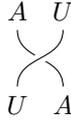
\[\begin{pic}[xscale=.6,yscale=.75]
  \draw (1,0) node[below]{$A$} to[out=90,in=-90] (0,1) node[above]{$A$};
  \draw[halo] (0,0) node[below]{$U$} to[out=90,in=-90] (1,1) node[above]{$U$};
\end{pic}\]

\begin{definition}  
  A \emph{central idempotent} in a monoidal category is \changed{a morphism $u \colon U \to I$ in its centre that is a left idempotent} such that:
  \begin{equation}\label{eq:central}
    \lambda_U \circ (u \otimes U) = \rho_U \circ (U \otimes u) \colon U \otimes U \to U
  \end{equation}
  Explicitly, it is a morphism $u \colon U \to I$ such that~\eqref{eq:central} holds and is invertible, equipped with a half-braiding satisfying:
  \begin{equation}\label{eq:halfbraidingrespectsu}
    \begin{pic}[xscale=.5,yscale=.75]
      \draw (1,0) node[below]{$A$} to[out=90,in=-90] (0,.8) to (0,1) node[above]{$A$};
      \draw[halo] (0,0) node[below]{$U$} to[out=90,in=-90] (1,.8) node[dot]{};
    \end{pic}
    =
    \begin{pic}[xscale=.5,yscale=.75]
      \draw (0,0) node[below]{$U$} to (0,.8) node[dot]{};
      \draw (1,0) node[below]{$A$} to (1,1) node[above]{$A$};
    \end{pic}
  \end{equation}

  We will identify two central idempotents $u \colon U \to I$ and $v \colon V \to I$ when $u=v \circ m$ for an isomorphism $m \colon U \to V$ that respects the half-braidings, and write $\ZI(\cat{C})$ for the set\footnote{\changed{Throughout this article we only consider central idempotents of \emph{small} categories to prevent size subtleties.}} of central idempotents of $\cat{C}$.
\end{definition}

The central equation~\eqref{eq:central} graphically becomes:
\[\begin{pic}
    \draw (0,0) node[below]{$U$} to (0,.5) node[dot]{};
    \draw (.3,0) node[below]{$U$} to (.3,1);
  \end{pic}
  =
  \begin{pic}
    \draw (-.3,0) node[below]{$U$} to (-.3,1);
    \draw (0,0) node[below]{$U$} to (0,.5) node[dot]{};
  \end{pic}
\]

If $u \colon U \to I$ and $v \colon V \to I$ are central idempotents, we draw a morphism $m \colon U \to V$ satisfying $u = v \circ m$ as:
\[\begin{pic}
  \node[dot] (d) {};
  \draw (d) to ++(0,-.5) node[below]{$U$};
  \draw (d) to ++(0,.5)  node[above]{$V$};
\end{pic}\]
All in all, the following graphical identities will be useful:
\[
  \begin{pic}
    \node[dot] (d) at (0,0) {};
    \draw (d) to ++(0,-.4) node[below]{$U$};
    \draw (d) to[out=0,in=-90] ++(.3,.3) to ++(0,.3) node[above]{$U$};
    \draw (d) to[out=180,in=-90] node[left]{$U$} ++(-.3,.3) node[dot]{};
  \end{pic}
  =
  \begin{pic}
    \draw (0,0) node[below]{$U$} to (0,1) node[above]{$U$};
  \end{pic}
  =
  \begin{pic}
    \node[dot] (d) at (0,0) {};
    \draw (d) to ++(0,-.4) node[below]{$U$};
    \draw (d) to[out=180,in=-90] ++(-.3,.3) to ++(0,.3) node[above]{$U$};
    \draw (d) to[out=0,in=-90] node[right]{$U$} ++(.3,.3) node[dot]{};
  \end{pic}
  \qquad\qquad
  \begin{tikzpicture}[baseline=-2mm]
    \node[dot] (b) at (0,0) {};
    \node[dot] (t) at (0,.6) {};
    \draw (b) to ++(0,-.6) node[below]{$U$};
    \draw (b) to node[left=-.5mm]{$V$} (t);
  \end{tikzpicture}
  =
  \begin{tikzpicture}[baseline=-2mm]
    \draw (0,-.6) node[below]{$U$} to ++(0,.6) node[dot]{};
  \end{tikzpicture}
  \qquad\qquad
  \begin{pic}
    \node[dot] (t) at (0,.5) {};
    \node[dot] (b) at (0,0) {};
    \draw (b) to ++(0,-.3) node[below]{$U$};
    \draw (b) to node[right=-.75mm]{$V$} (t);
    \draw (t) to[out=180,in=-90] ++(-.3,.3) node[above]{$V$};
    \draw (t) to[out=0,in=-90] ++(.3,.3) node[above]{$V$};
  \end{pic}
  =
  \begin{pic}
    \node[dot] (b) at (0,0) {};
    \node[dot] (l) at (-.3,.5) {};
    \node[dot] (r) at (.3,.5) {};
    \draw (b) to ++(0,-.3) node[below]{$U$};
    \draw (b) to[out=180,in=-90] node[left=-.4mm]{$U$} (l);
    \draw (b) to[out=0,in=-90] node[right=-.4mm]{$U$} (r);
    \draw (l) to ++(0,.3) node[above]{$V$};
    \draw (r) to ++(0,.3) node[above]{$V$};
  \end{pic}
\]
The following proposition ensures that these graphical notations are well-defined.

\begin{proposition}
  For central idempotents $u \colon U \to I$ and $v \colon V \to I$ in a monoidal category $\cat{C}$, the following are equivalent:
  \begin{itemize}
    \item $u = v \circ m$ for a necessarily unique $m \colon U \to V$ that respects half-braidings;
    \item $U \otimes v \colon U \otimes V \to U \otimes I$ is invertible;
    \item $v \otimes U \colon V \otimes U \to I \otimes U$ is invertible.
  \end{itemize}
  We say $u \leq v$ when these conditions hold. This partially orders $\ZI(\cat{C})$.
\end{proposition}
\begin{proof}
  First we prove that the morphism $m$ is unique. 
  If $u=v \circ m = v \circ n$, then:
  \[
    \begin{pic}
      \node[morphism] (m) at (0,0) {$m$};
      \draw (m.north) to ++(0,.3) node[above]{$V$};
      \draw (m.south) to ++(0,-.75) node[below]{$U$};
    \end{pic}
    \;=\;
    \begin{pic}
      \node[morphism] (m) at (0,0) {$m$};
      \draw (m.north) to ++(0,.3) node[above]{$V$};
      \node[dot] (d) at (.35,-.6) {};
      \draw (m.south) to[out=-90,in=180] (d.west);
      \draw (d.south) to ++(0,-.3) node[below]{$U$};
      \draw (d.east) to[out=0,in=-90] ++(.3,.6) node[dot]{};
    \end{pic}
    \;=\;
    \begin{pic}
      \node[morphism] (m) at (0,0) {$m$};
      \node[morphism] (n) at (.7,0) {$n$};
      \draw (m.north) to ++(0,.3) node[above]{$V$};
      \draw (n.north) to ++(0,.2) node[dot]{} node[above=1mm]{$V$};
      \node[dot] (d) at (.35,-.6) {};
      \draw (m.south) to[out=-90,in=180] (d.west);
      \draw (n.south) to[out=-90,in=0] (d.east);
      \draw (d.south) to ++(0,-.3) node[below]{$U$};
    \end{pic}
    \;=\;
    \begin{pic}
      \node[morphism] (m) at (0,0) {$m$};
      \node[morphism] (n) at (.7,0) {$n$};
      \draw (n.north) to ++(0,.3) node[above]{$V$};
      \draw (m.north) to ++(0,.2) node[dot]{} node[above=1mm]{$V$};
      \node[dot] (d) at (.35,-.6) {};
      \draw (m.south) to[out=-90,in=180] (d.west);
      \draw (n.south) to[out=-90,in=0] (d.east);
      \draw (d.south) to ++(0,-.3) node[below]{$U$};
    \end{pic}
    \;=\;
    \begin{pic}
      \node[morphism] (n) at (0,0) {$n$};
      \draw (n.north) to ++(0,.3) node[above]{$V$};
      \node[dot] (d) at (-.35,-.6) {};
      \draw (n.south) to[out=-90,in=0] (d.east);
      \draw (d.south) to ++(0,-.3) node[below]{$U$};
      \draw (d.west) to[out=180,in=-90] ++(-.3,.6) node[dot]{};
    \end{pic}
    =
    \begin{pic}
      \node[morphism] (n) at (0,0) {$n$};
      \draw (n.north) to ++(0,.3) node[above]{$V$};
      \draw (n.south) to ++(0,-.75) node[below]{$U$};
    \end{pic}
  \]
  To see that the first point implies the second:
  \[
    \begin{pic}
      \node[dot] (d) at (0,0) {};
      \draw (d.south) to ++(0,-.3) node[below]{$U$};
      \draw (d.west) to[out=180,in=-90,looseness=.8] ++(-.25,.75) node[above]{$U$};
      \draw (d.east) to[out=0,in=-90] node[below right=-1mm]{$U$} ++(.25,.3) node[dot]{} to ++(0,.6) node[dot]{} node[right]{$V$};
      \draw[dashed,gray] (-.5,-.3) rectangle (.7,.45);
    \end{pic}
    =
    \begin{pic}
      \node[dot] (d) at (0,0) {};
      \draw (d.south) to ++(0,-.3) node[below]{$U$};
      \draw (d.west) to[out=180,in=-90,looseness=.8] ++(-.25,.75) node[above]{$U$};
      \draw (d.east) to[out=0,in=-90] ++(.25,.3) node[dot]{};
    \end{pic}
    =
    \begin{pic}
      \draw (0,0) node[below]{$U$} to ++(0,1.1) node[above]{$U$};
    \end{pic}
    \qquad
    \begin{pic}
      \node[dot] (d) at (0,0) {};
      \draw (d.south) to ++(0,-.5) node[below]{$U$};
      \draw (.4,-.59) node[below]{$V$} to ++(0,.25) node[dot]{};
      \draw (d.west) to[out=180,in=-90] ++(-.25,.55) node[above]{$U$};
      \draw (d.east) to[out=0,in=-90] ++(.25,.3) node[dot]{} to ++(0,.30) node[above]{$V$};
      \draw[dashed,gray] (-.5,-.15) rectangle (.6,.45);
    \end{pic}
    =\!
    \begin{pic}
      \node[dot] (d) at (0,0) {};
      \draw (d.south) to ++(0,-.3) node[below]{$U$};
      \draw (.5,-.39) node[below]{$V$} to ++(0,.9) node[dot]{};
      \draw (d.west) to[out=180,in=-90,looseness=.8] ++(-.25,.75) node[above]{$U$};
      \draw (d.east) to[out=0,in=-90] ++(.2,.3) node[dot]{} to ++(0,.50) node[above]{$V$};
    \end{pic}
    =\!
    \begin{pic}
      \node[dot] (d) at (0,0) {};
      \draw (d.south) to ++(0,-.3) node[below]{$U$};
      \draw (.5,-.39) node[below]{$V$} to ++(0,1.15) node[above]{$V$};
      \draw (d.west) to[out=180,in=-90,looseness=.7] ++(-.25,.75) node[above]{$U$};
      \draw (d.east) to[out=0,in=-90] ++(.2,.3) node[dot]{} to ++(0,.3) node[dot]{} node[left]{$V$};
    \end{pic}
    =\!
    \begin{pic}
      \node[dot] (d) at (0,0) {};
      \draw (d.south) to ++(0,-.3) node[below]{$U$};
      \draw (.5,-.39) node[below]{$V$} to ++(0,1.15) node[above]{$V$};
      \draw (d.west) to[out=180,in=-90,looseness=.8] ++(-.25,.75) node[above]{$U$};
      \draw (d.east) to[out=0,in=-90] ++(.2,.3) node[dot]{};
    \end{pic}
    =
    \begin{pic}
      \draw (0,0) node[below]{$U$} to ++(0,1.15) node[above]{$U$};
      \draw (.3,0) node[below]{$V$} to ++(0,1.15) node[above]{$V$};
    \end{pic}  
  \]
  Hence the dashed morphism inverts $U \otimes v$.

  To see that the second point implies the first, suppose $f$ inverts $U \otimes v$. Then:
  \[
    \begin{pic}
      \node[dot] (d) at (0,0) {};
      \draw (d) node[left]{\vphantom{$U$}} to ++(0,-1.4) node[below]{$U$};
    \end{pic}
    =
    \begin{pic}
      \node[morphism, width=6mm] (f) at (0,0) {$f$};
      \draw (f.south) to ++(0,-.3) node[below]{$U$};
      \draw (f.north west) to ++(0,.7) node[dot]{} node[left]{$U$};
      \draw (f.north east) to ++(0,.3) node[dot]{} node[right]{$V$};
    \end{pic}
    = 
    \begin{pic}
      \node[morphism, width=6mm] (f) at (0,0) {$f$};
      \draw (f.south) to ++(0,-.3) node[below]{$U$};
      \draw (f.north west) to ++(0,.3) node[dot]{} node[left]{$U$};
      \draw (f.north east) to ++(0,.7) node[dot]{} node[right]{$V$};
      \draw[dashed,gray] (-.7,-.4) rectangle (.7,.7);
    \end{pic}
  \]
  Hence the dashed morphism $m$ satisfies $u = v \circ m$.

  The first and third points are similarly equivalent.
  These conditions clearly satisfy transitivity, reflexivity, and anti-symmetry.
\end{proof}

The previous proposition justifies drawing the mediating morphism $m \colon U \to V$ as an unlabeled dot. It also follows from the previous proposition that a central idempotent is completely determined by its domain: if $u,u' \colon U \to I$ both represent central idempotents, then $u=u' \circ m$ for a unique isomorphism $m \colon U \to U$. This justifies drawing $u$ as an unlabeled dot on a wire labelled $U$. 

\begin{lemma}
  If $\cat{C}$ is any monoidal category, $\ZI(\cat{C})$ is a (meet-)semilattice, with
  \[
    u \wedge v 
    \quad = \quad
    \begin{tikzpicture}[baseline=-4mm]
      \node[dot] (l) at (0,0) {};
      \node[dot] (r) at (.3,0) {};
      \draw (l.south) to ++(0,-.3) node[below]{$U$};
      \draw (r.south) to ++(0,-.3) node[below]{$V$};
    \end{tikzpicture}
  \]
  and largest element $1 \colon I \to I$.
\end{lemma}
\begin{proof}
  Clearly $u \geq u \wedge v \leq v$.
  If $u \geq w \leq v$, then:
  \[
    \begin{pic}
      \node[dot] (l) at (-.2,.3) {};
      \node[dot] (r) at (.2,.3) {};
      \node[dot] (d) at (0,0) {};
      \draw (d.south) to ++(0,-.3) node[below]{$W$};
      \draw (d.west) to[out=180,in=-90] (l.south);
      \draw (d.east) to[out=0,in=-90] (r.south);
      \draw (l.north) to ++(0,.2) node[dot]{} node[left]{$U$};
      \draw (r.north) to ++(0,.2) node[dot]{} node[right]{$V$};
    \end{pic}
    \;=\;
    \begin{tikzpicture}[baseline=-1.5mm]
      \node[dot] (l) at (-.2,.3) {};
      \node[dot] (r) at (.2,.3) {};
      \node[dot] (d) at (0,0) {};
      \draw (d.south) to ++(0,-.3) node[below]{$W$};
      \draw (d.west) to[out=180,in=-90] (l.south);
      \draw (d.east) to[out=0,in=-90] (r.south);
    \end{tikzpicture}
    \;=\;\;
    \begin{tikzpicture}[baseline=-1.5mm]
      \node[dot] (d) at (0,0) {};
      \draw (d.south) to ++(0,-.3) node[below]{$W$};
    \end{tikzpicture}
  \]
  So $w \leq u \wedge v$, and $u \wedge v$ is the greatest lower bound of $u$ and $v$.
  That $1$ is the greatest element is clear.
\end{proof}

We have taken some pains to define central idempotents for monoidal categories. The next two lemmas show that for braided monoidal categories, the half-braiding is superfluous, and hence the definition of central idempotents simplifies.

\begin{lemma}\label{lem:braiding}
  A half-braiding $\sigma$ makes a left idempotent $u$ in a monoidal category into a central idempotent if and only if $\sigma_{U,U}=U \otimes U$:
  \[
    \begin{pic}
      \draw (0,0) node[below]{$U$} to (0,.5) node[dot]{};
      \draw (.4,0) node[below]{$U$} to (.4,.75) node[above]{$U$};
    \end{pic}
    =
    \begin{pic}
      \draw (-.4,0) node[below]{$U$} to (-.4,.75) node[above]{$U$};
      \draw (0,0) node[below]{$U$} to (0,.5) node[dot]{};
    \end{pic}
    \qquad\iff\qquad
    \begin{pic}[xscale=.6,yscale=.75]
      \draw (1,0) node[below]{$U$} to[out=90,in=-90] (0,1) node[above]{$U$};
      \draw[halo] (0,0) node[below]{$U$} to[out=90,in=-90] (1,1) node[above]{$U$};
    \end{pic}
    =
    \begin{pic}[xscale=.5,yscale=.75]
      \draw (0,0) node[below]{$U$} to (0,1) node[above]{$U$};
      \draw (1,0) node[below]{$U$} to (1,1) node[above]{$U$};
    \end{pic}
  \]
\end{lemma}
\begin{proof}
  If $u$ is a left idempotent satisfying~\eqref{eq:central}, then:
  \[
    \begin{pic}[xscale=.6,yscale=.75]
      \draw (1,0) node[below]{$U$} to[out=90,in=-90] (0,1) to ++(0,.3) node[above]{$U$};
      \draw[halo] (0,0) node[below]{$U$} to[out=90,in=-90] (1,1) to ++(0,.3) node[above]{$U$};
    \end{pic}
    =
    \begin{pic}[xscale=.5,yscale=.75]
      \node[dot] (l) at (0,.8) {};
      \node[dot] (r) at (1,.8) {};
      \draw (1,0) node[below]{$U$} to[out=90,in=-90] (l.south);
      \draw[halo] (0,0) node[below]{$U$} to[out=90,in=-90] (r.south);
      \node[dot] (r) at (1,.8) {};
      \draw (r.west) to[out=180,in=-90] ++(-.3,.5) node[above]{$U$};
      \draw (r.east) to[out=0,in=-90] ++(.3,.5) node[above]{$U$};
    \end{pic}
    =
    \begin{pic}[xscale=.5,yscale=.75]
      \node[dot] (l) at (0,.6) {};
      \node[dot] (r) at (1,1) {};
      \draw (0,0) node[below]{$U$} to (l.south);
      \draw (1,0) node[below]{$U$} to (r.south);
      \draw (l.west) to[out=180,in=-90] ++(-.3,.7) node[above]{$U$};
      \draw (l.east) to[out=0,in=-90] ++(.3,.7) node[above]{$U$};
    \end{pic}
    =
    \begin{pic}[xscale=.5,yscale=.75]
      \node[dot] (l) at (0,.6) {};
      \draw (0,0) node[below]{$U$} to (l.south);
      \draw (1,0) node[below]{$U$} to ++(0,1.3) node[above]{$U$};
      \draw (l.west) to[out=180,in=-90] ++(-.3,.7) node[above]{$U$};
      \draw (l.east) to[out=0,in=-90] ++(.3,.4) node[dot]{};
    \end{pic}
    =
    \begin{pic}[xscale=.5,yscale=.75]
      \draw (0,0) node[below]{$U$} to ++(0,1.3) node[above]{$U$};
      \draw (1,0) node[below]{$U$} to ++(0,1.3) node[above]{$U$};
    \end{pic}
  \]
  Conversely, if $\sigma_{U,U}=U \otimes U$, then $\lambda_U \circ (u \otimes U) = \lambda_U \circ (u \otimes U) \circ \sigma_{U,U} = \rho_U \circ (U \otimes u)$.
\end{proof}

\begin{lemma}\label{lem:halfbraidingisbraiding}
  If $u$ is a central idempotent in a braided monoidal category, then the half-braiding $\sigma_u$ equals \changed{the braiding $U \otimes B \to B \otimes U$}.
\end{lemma}
\begin{proof}
  For this proof only, to distinguish them we will draw the half-braiding of $U$ as
  $\begin{pic}[scale=.5]
    \draw (1,0) to[out=90,in=-90] (0,1);
    \draw[preaction={draw,white,line width=2pt,-},
          preaction={draw,white,ultra thick, shorten >=-2.5\pgflinewidth}] (0,0) to[out=90,in=-90] (1,1);
    \halfbraiding{.5}{.5};
  \end{pic}$,
  the braiding as 
  $\begin{pic}[scale=.5]
    \draw (1,0) to[out=90,in=-90] (0,1);
    \draw[preaction={draw,white,line width=2pt,-},
          preaction={draw,white,ultra thick, shorten >=-2.5\pgflinewidth}] (0,0) to[out=90,in=-90] (1,1);
  \end{pic}$,
  and the inverse braiding as
  $\begin{pic}[scale=.5]
    \draw (0,0) to[out=90,in=-90] (1,1);
    \draw[preaction={draw,white,line width=2pt,-},
          preaction={draw,white,ultra thick, shorten >=-2.5\pgflinewidth}] (1,0) to[out=90,in=-90] (0,1);
  \end{pic}$.
  Then, graphically:
  \begin{align*}
    \begin{pic}[scale=.75]
      \draw (1,-3) node[below]{$B$} to (1,0) to[out=90,in=-90] (0,1) node[above]{$B$};
      \draw[halo] (0,-3) node[below]{$U$} to (0,0) to[out=90,in=-90] (1,1) node[above]{$U$};
      \halfbraiding{.5}{.5};
    \end{pic}
    & = 
    \begin{pic}[scale=.75]
      \node[dot] (d) at (-.5,0) {};
      \draw (d) to (-.5,-.5) node[below]{$U$};
      \draw (d) to[out=180,in=-90] (-1,.5) to (-1,2.5) node[dot]{};
      \draw (1,-.5) node[below]{$B$} to (1,2.5) to[out=90,in=-90] (0,3.5) node[above]{$B$};
      \draw[halo] (d) to[out=0,in=-90] (0,.5) to (0,2.5) to[out=90,in=-90] (1,3.5) node[above]{$U$};
      \halfbraiding{.5}{3};
    \end{pic}
    =
    \begin{pic}[scale=.75]
      \node[dot] (d) at (-.5,0) {};
      \draw (d) to (-.5,-.5) node[below]{$U$};
      \draw (d) to[out=180,in=-90] (-1,.5) to (-1,2.5) node[dot]{};
      \draw (d) to[out=0,in=-90] (0,.5);
      \draw (1,-.5) node[below]{$B$} to (1,.5) to[out=90,in=-90] (0,1.5) to[out=90,in=-90] (1,2.5) to[out=90,in=-90] (0,3.5) node[above]{$B$};
      \draw[halo] (0,.5) to[out=90,in=-90] (1,1.5) to[out=90,in=-90] (0,2.5) to[out=90,in=-90] (1,3.5) node[above]{$U$};
      \halfbraiding{.5}{3};
    \end{pic}
    = 
    \begin{pic}[scale=.75]
      \node[dot] (d) at (-.5,0) {};
      \draw (d) to (-.5,-.5) node[below]{$U$};
      \draw (1,-.5) node[below]{$B$} to (1,.5) to[out=90,in=-90] (0,1.5) to[out=90,in=-90] (1,2.5) to[out=90,in=-90] (0,3.5) node[above]{$B$};
      \draw[halo] (d) to[out=180,in=-90] (-1,.5) to[out=90,in=-150] (.5,3) to[out=30,in=-90] (1,3.5) node[above]{$U$};  
      \draw[halo] (d) to[out=0,in=-90] (0,.5) to[out=90,in=-90] (1,1.5) to[out=90,in=-80] (-1,2.5) node[dot]{};
      \halfbraiding{.5}{3};
    \end{pic}
    \\
    & =
    \begin{pic}[scale=.75]
      \node[dot] (d) at (-.5,0) {};
      \draw (d) to (-.5,-.5) node[below]{$U$};
      \draw (1,-.5) node[below]{$B$} to (1,.5) to[out=90,in=-90] (-1,1.75) to[out=90,in=-90] (0,3) node[above]{$B$};
      \draw[halo] (d) to[out=180,in=-90] (-1,.5) to[out=90,in=-90] (1,2.5) to (1,3) node[above]{$U$};
      \draw[halo] (d) to[out=0,in=-90] (0,.5) to[out=90,in=-90] (1,1.5) to[out=90,in=-80] (-1,2.5) node[dot]{};
      \halfbraiding{-.525}{1.225};
    \end{pic}
    =
    \begin{pic}[scale=.75]
      \node[dot] (d) at (-.5,0) {};
      \draw (d) to (-.5,-.5) node[below]{$U$};
      \draw[halo] (1,-.5) node[below]{$B$} to (1,0) to[out=90,in=-90] (-1,1.5) to[out=90,in=-90] (0,3) node[above]{$B$};
      \draw[halo] (d) to[out=180,in=-90] (-1,.5) to[out=90,in=-90] (0,1.5) to[out=90,in=-90] (-1,2.5) node[dot]{};
      \draw[halo] (d) to[out=0,in=-90] (1,3) node[above]{$U$};
      \halfbraiding{-.6}{.95};
    \end{pic}  
    = 
    \begin{pic}[scale=.75]
      \node[dot] (d) at (-.5,0) {};
      \draw (d) to (-.5,-.5) node[below]{$U$};
      \draw (.5,-.5) node[below]{$B$} to (.5,.5) to[out=90,in=-90] (-1,2) to (-1,3) node[above]{$B$};
      \draw[halo] (d) to[out=180,in=-90] (-1,.5) to (-1,1) to[out=90,in=-100] (-.25,1.75) to (-.25,2.5) node[dot]{};
      \draw[halo] (d) to[out=0,in=-90,looseness=.8] (.5,3) node[above]{$U$};
      \halfbraiding{-.6}{1.4};
    \end{pic}
    =
    \begin{pic}[scale=.75]
      \node[dot] (d) at (-.5,.5) {};
      \draw (d) to (-.5,0) node[below]{$U$};
      \draw (d) to[out=180,in=-90] (-1,1) to (-1,2.5) node[dot]{};
      \draw (1,0) node[below]{$B$} to (1,2.5) to[out=90,in=-90] (0,3.5) node[above]{$B$};
      \draw[halo] (d) to[out=0,in=-90] (0,1) to (0,2.5) to[out=90,in=-90] (1,3.5) node[above]{$U$};
    \end{pic}
    = 
    \begin{pic}[scale=.75]
      \draw (1,-2.5) node[below]{$B$} to (1,0) to[out=90,in=-90] (0,1) node[above]{$B$};
      \draw[halo] (0,-2.5) node[below]{$U$} to (0,0) to[out=90,in=-90] (1,1) node[above]{$U$};
    \end{pic}
  \end{align*}
  These equalities used, respectively: 
  invertibility of $u \otimes U$,
  invertibility of the braiding,
  Lemma~\ref{lem:braiding},
  naturality of the braiding in the first argument,
  Lemma~\ref{lem:braiding},
  naturality of the braiding in the second argument,
  equation~\eqref{eq:halfbraidingrespectsu},
  and finally invertibility of $u \otimes U$.
  Notice that (the proof of) Lemma~\ref{lem:braiding} indeed holds for the braiding rather than the half-braiding, and really only depends on invertibility of $u \otimes U$ and the central equation~\ref{eq:central}.
\end{proof}

The rest of this section elaborates some examples.

\begin{lemma}\label{lem:zi:cartesian}
  In a cartesian category, an object is (the domain of) a central idempotent if and only if it is subterminal.
\end{lemma}
\begin{proof}
  Suppose the unique morphism $u \colon U \to 1$ is a central idempotent.
  Let $f,g \colon A \to U$.
  Write $a$ for the unique morphism $A \to 1$.
  Then $A \times a \colon A \otimes A \to A$ has an inverse $\Delta \colon A \to A \otimes A$, and:
  \[
    \begin{pic}
      \node[morphism] (f) at (0,0) {$f$};
      \draw (f.north) to ++(0,.5) node[above]{$U$};
      \draw (f.south) to ++(0,-1.1) node[below]{$A$};
    \end{pic}
    =
    \begin{pic}
      \node[morphism,width=8mm] (d) at (0,0) {$\Delta$};
      \node[morphism] (f) at ([yshift=6mm]d.north west) {$f$};
      \node[morphism] (a) at ([yshift=6mm]d.north east) {$a$};
      \draw (d.north east) to node[right]{$A$} (a.south);
      \draw (f.north) to ++(0,.5) node[above]{$U$};
      \draw (d.north west) to node[left]{$A$} (f.south);
      \draw (d.south) to ++(0,-.3) node[below]{$A$};
    \end{pic}
    =
    \begin{pic}
      \node[morphism,width=8mm] (d) at (0,0) {$\Delta$};
      \node[morphism] (f) at ([yshift=6mm]d.north west) {$f$};
      \node[morphism] (g) at ([yshift=6mm]d.north east) {$g$};
      \draw (d.north east) to node[right]{$A$} (g.south);
      \draw (f.north) to ++(0,.5) node[above]{$U$};
      \draw (g.north) to ++(0,.3) node[dot]{} node[right]{$U$};
      \draw (d.north west) to node[left]{$A$} (f.south);
      \draw (d.south) to ++(0,-.3) node[below]{$A$};
    \end{pic}
    =
    \begin{pic}
      \node[morphism,width=8mm] (d) at (0,0) {$\Delta$};
      \node[morphism] (f) at ([yshift=6mm]d.north west) {$f$};
      \node[morphism] (g) at ([yshift=6mm]d.north east) {$g$};
      \draw (d.north east) to node[right]{$A$} (g.south);
      \draw (g.north) to ++(0,.5) node[above]{$U$};
      \draw (f.north) to ++(0,.3) node[dot]{} node[left]{$U$};
      \draw (d.north west) to node[left]{$A$} (f.south);
      \draw (d.south) to ++(0,-.3) node[below]{$A$};
    \end{pic}
    =
    \begin{pic}
      \node[morphism,width=8mm] (d) at (0,0) {$\Delta$};
      \node[morphism] (a) at ([yshift=6mm]d.north west) {$a$};
      \node[morphism] (g) at ([yshift=6mm]d.north east) {$g$};
      \draw (d.north west) to node[left]{$A$} (a.south);
      \draw (g.north) to ++(0,.5) node[above]{$U$};
      \draw (d.north east) to node[right]{$A$} (g.south);
      \draw (d.south) to ++(0,-.3) node[below]{$A$};
    \end{pic}
    =
    \begin{pic}
      \node[morphism] (g) at (0,0) {$g$};
      \draw (g.north) to ++(0,.5) node[above]{$U$};
      \draw (g.south) to ++(0,-1.1) node[below]{$A$};
    \end{pic}
  \]
  So $u$ is monic.
\end{proof}

In particular, if $X$ is a topological space, then central idempotents in its category of sheaves $\mathrm{Sh}(X)$ correspond to open subsets of $X$~\cite[Corollary~2.2.16]{borceux:3}.

\begin{example}\label{ex:semilattices}
  Any \emph{(meet-)semilattice} may be regarded as a (strict) monoidal category: objects are elements of the semilattice, there is a unique morphism $u \to v$ if and only if $u \leq v$, tensor products are given by greatest lower bounds, and the tensor unit is the greatest element. 
  In this special case of Lemma~\ref{lem:zi:cartesian}, any object is a central idempotent.
	Write $\cat{SLat}_{\leq}$ for the category of semilattices and monotone functions that preserve the greatest element.
	Write $\cat{MonCat}$ for the category of monoidal categories and functors $F$ that are (lax) monoidal and whose coherence morphism $I \to F(I)$ is invertible (see also Definition~\ref{def:moncat} below).
  The former is a coreflective subcategory of the latter:
  \[\begin{pic}
    \node (l) at (0,0) {$\cat{SLat}_{\leq}$};
    \node (r) at (4,0) {$\cat{MonCat}$};
    \draw[->] ([yshift=2mm]l.east) to ([yshift=2mm]r.west);
    \draw[<-] ([yshift=-2mm]l.east) to node[below]{$\ZI$} ([yshift=-2mm]r.west);
    \draw[draw=none] (l) to node{$\perp$} (r);
  \end{pic}\]  
\end{example}

\begin{example}\label{ex:quantale}
  A \emph{quantale}~\cite{rosenthal:quantales} is a complete lattice $Q$ with an element $e \in Q$ and an associative multiplication $Q \times Q \to Q$ such that:
  \[
    u (\bigvee v_i) = \bigvee uv_i
    \qquad
    (\bigvee u_i) v = \bigvee u_i v
    \qquad
    eu = u = ue
  \]
  Any quantale may be regarded as a monoidal category: objects are elements of the quantale, (composition of) morphisms is induced by the partial order, and the tensor product is induced by the multiplication.
  The central idempotents of this category are the central elements $q^2 = q \leq e$.
  Taking as morphisms between quantales functions that preserve $\bigvee$, $\cdot$, and $e$, this gives a functor $\cat{Quantale} \to \cat{MonCat}$.

  For an important special case, recall that a \emph{frame} is a complete lattice in which finite joins distribute over suprema~\cite{johnstone:stonespaces}: a frame is a commutative quantale in which the multiplication is idempotent and whose unit is the largest element.
  Frames form a coreflective subcategory of commutative quantales~\cite[Proposition~3.5]{enriquemolinerheunentull:tensortopology}, where a morphism of frames is a function that preserves $\bigvee$, $\wedge$, and $1$:
  \[\begin{pic}
    \node (l) at (0,0) {$\cat{Frame}$};
    \node (r) at (4,0) {$\cat{cQuantale}$};
    \draw[->] ([yshift=2mm]l.east) to ([yshift=2mm]r.west);
    \draw[<-] ([yshift=-2mm]l.east) to node[below]{$\ZI$} ([yshift=-2mm]r.west);
    \draw[draw=none] (l) to node{$\perp$} (r);
  \end{pic}\]
\end{example}

\begin{example}\label{ex:modules}
  Consider the monoidal category $\cat{Mod}_R$ of \emph{modules} over a commutative ring $R$.
  Centrality of $a \colon A \to R$ means that:
  \begin{equation}\label{eq:mod:central}
    a(x) \cdot y = a(y) \cdot x
  \end{equation}
  for every $x,y \in A$. 
  Idempotency means that:
  \begin{align}
    \forall x\: \exists x_i, y_i \colon x = \sum_{i=1}^n a(x_i) \cdot y_i \label{eq:mod:idempotent:surjective} \\
    a(x) \cdot y = 0 \iff x \otimes y = 0 \label{eq:mod:idempotent:injective}
  \end{align}
  If $e \in R$ is a (central) idempotent element of $R$, then the (inclusion $a$ into $R$ of the) ideal $A=eR$ is a central idempotent in $\cat{Mod}_R$.
  Conversely, \eqref{eq:mod:idempotent:surjective} implies that $a(A) \subseteq R$ is an idempotent ideal.
  %
  In general we cannot say much more, but if $A$ is faithfully flat~\cite[Section~4I]{lam:modules}, then $A \otimes a$ being injective implies that $a\colon A \to R$ is injective, and hence central idempotents correspond to idempotent ideals.
  This includes free modules $A$, and hence vector spaces $A$.
  %
  %
\end{example}


\begin{example}\label{ex:hilbert}
  Consider the category $\cat{Hilb}_{C_0(X)}$ of \emph{Hilbert C*-modules}~\cite{lance1995hilbert} for a locally compact Hausdorff space $X$. 
  This is not an abelian category~\cite[Appendix~A]{heunen:embedding}, nor a topos.
  It is monoidally equivalent to the category of fields of Hilbert spaces over $X$~\cite[Corollary~4.9]{heunenreyes:frobenius}.
  Under this equivalence the monoidal structure is fibrewise, so a central idempotent $a \colon A \to C_0(X)$ in $\cat{Hilb}_{C_0(X)}$ becomes a continuous functions that is fibrewise a central idempotent $a_x \colon A_x \to \mathbb{C}$ in the category of Hilbert spaces and bounded linear maps.
  It follows from Example~\ref{ex:modules} that each $a_x$ is injective. 
  Finally, because composition of maps between fields of Hilbert spaces is fibrewise too, we conclude that the central idempotent $a$ in $\cat{Hilb}_{C_0(X)}$ is monic. 
  It follows~\cite[Proposition~3.16]{enriquemolinerheunentull:tensortopology} that $\ZI(\cat{Hilb}_{C_0(X)}) \simeq \{ U \subseteq X \text{ open} \}$.
\end{example}

\begin{example}\label{ex:idempotentmonad}
  If $\cat{C}$ is any category, the functor category $[\cat{C},\cat{C}]$ is monoidal with composition as tensor product. 
  If $\varepsilon \colon T \Rightarrow \cat{C}$ is a left idempotent satisfying~\eqref{eq:central}, then the inverse $\varepsilon T = T\varepsilon \colon T^2 \Rightarrow T$ is a comultiplication that makes $T$ into an \emph{idempotent comonad} on $\cat{C}$.
  Conversely, if $T$ is an idempotent comonad on $\cat{C}$, then its counit is a left idempotent in $[\cat{C},\cat{C}]$ satisfying~\eqref{eq:central} by~\cite[Proposition~4.2.3]{borceux:2}.
  (In fact, the counit is central if and only if the comonad is idempotent~\cite[Proposition~2.11, e$\Leftrightarrow$f]{clarkwisbauer:idempotent}.) 
  Thus idempotents in $[\cat{C},\cat{C}]$ satisfying~\eqref{eq:central} correspond (up to natural isomorphism) to idempotent comonads.
  Notice that the counit of an idempotent comonad need not be monic~\cite[Lemma~3.4]{lambekrattray:localization}. 
  See also~\cite{kelly:transfinite}.

  An idempotent comonad $\varepsilon \colon T \Rightarrow \cat{C}$ is in the centre of the monoidal category $[\cat{C},\cat{C}]$ when $FT \simeq TF$ for all endofunctors $F \colon \cat{C} \to \cat{C}$, and $\varepsilon$ respects these natural isomorphisms.
  As in Lemma~\ref{lem:braiding}, it follows that:
  \[
    \begin{pic}[xscale=.5,yscale=.8]
      \draw (1,0) node[below]{$F$} to[out=90,in=-90] (0,1) to ++(0,.5);
      \draw[halo] (0,0) node[below]{$T$} to[out=90,in=-90] (1,1) to ++(0,.5);
      \draw (2,0) node[below]{$T$} to ++(0,1) node[dot]{};
    \end{pic}
    =
    \begin{pic}[xscale=.5,yscale=.8]
      \draw (0,0) node[below]{$T$} to ++(0,1) node[dot]{};
      \draw (1,0) node[below]{$F$} to ++(0,1.5);
      \draw (2,0) node[below]{$T$} to ++(0,1.5);
    \end{pic}
  \]
  Hence $\varepsilon_{FT(A)} = FT(\varepsilon_A) \circ \sigma^F_{T(A)}$ is a composition of invertible maps and so an isomorphism for every $A \in \cat{C}$ and $F \colon \cat{C} \to \cat{C}$. Taking $F$ to be the functor that maps every object to $A$ and every morphism to the identity on $A$ shows that $\varepsilon$ is a natural isomorphism.
  Hence $\ZI([\cat{C},\cat{C}]) = 1$.
\end{example}




For more examples, see~\cite[3.7--3.8, 3.12--3.14]{enriquemolinerheunentull:tensortopology}.

\section{Universal joins}\label{sec:universaljoins}

This section defines properties of central idempotents that make them more well-behaved, starting with the weakest one.

\begin{definition}\label{def:universaljoins}
  A monoidal category is \emph{stiff} if
  \[
    \begin{pic}[xscale=3,yscale=1.5]
      \node (tl) at (0,1) {$A \otimes U \otimes V$};
      \node (tr) at (1,1) {$A \otimes V$};
      \node (bl) at (0,0) {$A \otimes U$};
      \node (br) at (1,0) {$A$};
      \draw[->] (tl) to (tr);
      \draw[->] (tl) to (bl);
      \draw[->] (tr) to (br);
      \draw[->] (bl) to (br);
      \draw (.1,.7) to (.15,.7) to (.15,.825);
    \end{pic}
    \qquad\qquad
    \begin{pic}
      \node[dot] (s) at (0,.2) {};
      \draw (s) to (0,-.5) node[below]{$U$};
      \node[dot] (t) at (.5,-.2) {};
      \draw (t) to (.5,-.5) node[below]{$V$};
      \draw (-.5,-.5) node[below]{$A$} to (-.5,.5) node[above]{$A$};
    \end{pic}
    =
    \begin{pic}
  	  \node[dot] (s) at (0,-.2) {};
  	  \draw (s) to (0,-.5) node[below]{$U$};
  	  \node[dot] (t) at (.5,.2) {};
  	  \draw (t) to (.5,-.5) node[below]{$V$};
  	  \draw (-.5,-.5) node[below]{$A$} to (-.5,.5) node[above]{$A$};
    \end{pic}
  \]
  is a pullback for all objects $A$ and central idempotents $u$ and $v$.

  A monoidal category $\cat{C}$ has \emph{finite universal joins}\changed{, or universal $\vee$-joins,} of central idempotents when it has an initial object $0$ satisfying $A \otimes 0 \simeq 0$ for all objects $A$, and $\ZI(\cat{C})$ has binary joins such that
    \begin{equation} \label{eq:pullback-pushout}
      \begin{pic}[xscale=3.3,yscale=1.5]
	    \node (tl) at (0,1) {$A \otimes U \otimes V$};
	    \node (tr) at (1,1) {$A \otimes V$};
	    \node (bl) at (0,0) {$A \otimes U$};
	    \node (br) at (1,0) {$A \otimes (U \vee V)$};
	    \draw[>->] (tl) to node[above]{} (tr);
	    \draw[>->] (tl) to node[left]{} (bl);
	    \draw[>->] (tr) to node[right]{} (br);
	    \draw[>->] (bl) to node[below]{} (br);
	    \draw (.1,.7) to (.15,.7) to (.15,.825);
	    \draw (.95,.3) to (.90,.3) to (.90,.175);
	  \end{pic}
	  \qquad\qquad
	  \begin{pic}[xscale=2]
  		\node[dot] (d) at (.3,0)	{};
  		\draw (d) to (.3,-.5) node[below]{$V$};
  		\node[dot] (D) at (0,0) {};
  		\draw (0,-.5) node[below]{$U$} to (D);
  		\draw (D) to (0,.5) node[above]{$\quad\quad U \vee V$};
  		\draw (-.3,-.5) node[below]{$A$} to (-.3,.5) node[above]{$A$};
	  \end{pic}
    =
    \begin{pic}[xscale=2.3]
      \node[dot] (d) at (0,0) {};
      \draw (d) to (0,-.5) node[below]{$U$};
      \node[dot] (D) at (.3,0) {};
      \draw (.3,-.5) node[below]{$V$} to (D);
      \draw (D) to (.3,.5) node[above]{$U \vee V$};
      \draw (-.3,-.5) node[below]{$A$} to (-.3,.5) node[above]{$A$};
    \end{pic}
   \end{equation}
   is both a pullback and pushout for all objects $A$ and central idempotents $u$ and $v$.

  A monoidal category has \emph{universal joins}\changed{, or universal $\bigvee$-joins,} of central idempotents when $\ZI(\cat{C})$ has all joins such that the cocone
  \[\begin{pic}[xscale=1.2]
  	\node[dot] (d) at (0,0) {};
  	\draw (d) to (0,-.5) node[below]{$U_i$};
  	\draw (d) to (0,.5) node[above]{\smash{$\bigvee\!U_i$}};
  	\draw (-.5,-.5) node[below]{$A$} to (-.5,.5) node[above]{$A$};
  \end{pic}\]
  is a colimit of the diagram with morphisms
  \[\begin{pic}[xscale=1.2]
  	\node[dot] (d) at (0,0) {};
  	\draw (d) to (0,-.5) node[below]{$U_i$};
  	\draw (d) to (0,.5) node[above]{$\smash{U_j}$};
  	\draw (-.5,-.5) node[below]{$A$} to (-.5,.5) node[above]{$A$};
  \end{pic}\]
  if $u_i \leq u_j$ for any set $\{u_i\}$ of central idempotents satisfying $\{u_i\} = \{u_i \wedge u_j\}$.
\end{definition}

If $\cat{C}$ has universal finite joins of central idempotents, then $\ZI(\cat{C})$ is a distributive lattice with least element $0$.
If $\cat{C}$ has universal joins of central idempotents, then $\ZI(\cat{C})$ is even a frame. 

Any coherent category has universal finite joins of central idempotents~\cite[Proposition~7.4]{enriquemolinerheunentull:tensortopology}.
Any cocomplete Heyting category has universal joins of central idempotents~\cite[Proposition~7.12]{enriquemolinerheunentull:tensortopology}; this includes all cocomplete toposes, such as Grothendieck toposes.
Any quantale, regarded as a monoidal category as in Example~\ref{ex:quantale}, has universal joins of central idempotents.
The category of Hilbert C*-modules of Example~\ref{ex:hilbert} has universal joins of central idempotents.

Next, we show that the properties of being equal, being monic, being epic, and being invertible, may be verified `central idempotent-wise'.

\begin{lemma}\label{lem:epimonoisolocal}
  If (finitely many) central idempotents $u_i$ in a monoidal category with universal (finite) joins of central idempotents satisfy $\bigvee u_i = 1$, then:
  \begin{itemize}
    \item morphisms $f,g \colon A \to B$ are equal if and only if $f \otimes u_i = g \otimes u_i$ for each $i$;
    \item a morphism $f$ is epic if $f \otimes u_i$ or $f \otimes U_i$ is epic for each $i$;
    \item a morphism $f$ is monic if $f \otimes u_i$ is monic for each $i$;
    \item a morphism $f$ is invertible if and only if $f \otimes U_i$ is invertible for each $i$.
  \end{itemize}
\end{lemma}
\begin{proof}
  The point about equality of morphisms follows directly from Definition~\ref{def:universaljoins}.

  Next, let $f \colon A \to B$ and suppose that $g \circ f = h \circ f$ for $g,h \colon B \to C$. Then also $(g \otimes u_i) \circ (f \otimes U_i) = (h \otimes u_i) \circ (f \otimes U_i)$. Therefore $g \otimes u_i = h \otimes u_i$ because $f \otimes U_i$ is epic. But then $g=h$ as before because $\bigvee u_i = 1$. 
  If instead each $f \otimes u_i$ is epic, then immediately $g=h$.
  In either case $f$ is epic.

  We turn to monomorphisms. Let $f \colon B \to C$. Suppose that $f \circ g = f \circ h$ for $g,h \colon A \to B$. Because $f \otimes u_i$ is epic, then:
  \[
    \begin{pic}
      \node[morphism] (f) at (0,0) {$f$};
      \node[morphism] (g) at (0,-.8) {$g$};
      \draw (g.north) to node[left=-1mm]{$B$} (f.south);
      \draw (f.north) to  ++(0,.3) node[left=-1mm]{$C$};
      \draw (g.south) to  ++(0,-.3) node[left=-1mm]{$A$};
      \draw (.5,0) node[dot]{} to ++(0,-1.3) node[right=-1mm]{$U\smash{_i}$};
    \end{pic}
    \;=
    \begin{pic}
      \node[morphism] (f) at (0,0) {$f$};
      \node[morphism] (g) at (0,-.8) {$h$};
      \draw (g.north) to node[left=-1mm]{$B$} (f.south);
      \draw (f.north) to  ++(0,.3) node[left=-1mm]{$C$};
      \draw (g.south) to  ++(0,-.3) node[left=-1mm]{$A$};
      \draw (.5,0) node[dot]{} to ++(0,-1.3) node[right=-1mm]{$U\smash{_i}$};
    \end{pic}
    \;\;\implies\;\;
    \begin{pic}
      \node[morphism] (g) at (0,0) {$g$};
      \draw (g.north) to  ++(0,.3) node[above]{$B$};
      \draw (g.south) to  ++(0,-.3) node[below]{$A$};
      \draw (.5,-.5) node[below]{$U_i$} to ++(0,1) node[above]{$U\smash{_i}$};
    \end{pic}
    \!\!=\;
    \begin{pic}
      \node[morphism] (g) at (0,0) {$h$};
      \draw (g.north) to  ++(0,.3) node[above]{$B$};
      \draw (g.south) to  ++(0,-.3) node[below]{$A$};
      \draw (.5,-.5) node[below]{$U_i$} to ++(0,1) node[above]{$U\smash{_i}$};
    \end{pic}
    \;\implies\;\;
    \begin{pic}
      \node[morphism] (g) at (0,0) {$g$};
      \draw (g.north) to  ++(0,.3) node[above]{$B$};
      \draw (g.south) to  ++(0,-.3) node[below]{$A$};
      \draw (.5,-.5) node[below]{$U_i$} to ++(0,.5) node[dot]{};
    \end{pic}
    =
    \begin{pic}
      \node[morphism] (g) at (0,0) {$h$};
      \draw (g.north) to  ++(0,.3) node[above]{$B$};
      \draw (g.south) to  ++(0,-.3) node[below]{$A$};
      \draw (.5,-.5) node[below]{$U_i$} to ++(0,.5) node[dot]{};
    \end{pic}
  \]
  So $g=h$ because $\bigvee u_i=1$. Thus $f$ is monic.

  Finally, suppose $f \colon A \to B$ makes each $f \otimes U_i \colon A \otimes U_i \to B \otimes U_i$ invertible. Write $g_i \colon B \otimes U_i \to A \otimes U_i$ for the inverses. 
  Now the morphisms $g_i \otimes u_i \colon B \otimes U_i \to A$ form a cocone for the diagram $B \otimes u_i \colon B \otimes U_i \to B$, giving a mediating morphism $g \colon B \to A$. It follows from uniqueness that $g \circ f = B$ and $f \circ g = A$, making $f$ invertible.
  Conversely, if $f$ has inverse $g$, then $f \otimes U_i$ has inverse $g \otimes U_i$.
\end{proof}

In Section~\ref{sec:embedding} below we will see that any stiff monoidal category can be freely completed with universal (finite) joins of central idempotents. We end this section by considering which functors between monoidal categories preserve central idempotents.

\begin{definition}
  A lax monoidal functor $F \colon \cat{C} \to \cat{D}$ with coherence morphisms $\theta_{A,B} \colon F(A) \otimes F(B) \to F(A \otimes B)$ and $\theta_I \colon I \to F(I)$ \emph{preserves a half-braiding} $\sigma^U_A \colon U \otimes A \to A \otimes U$ in $\cat{C}$ when there is a chosen half-braiding $F(\sigma^U)_D \colon F(U) \otimes D \to D \otimes F(U)$ in $\cat{D}$ satisfying $\theta_{A,U} \circ F(\sigma^U)_{F(A)} = F(\sigma_A^U) \circ \theta_{U,A}$.
\end{definition}

Any braided lax monoidal functor preserves all half-braidings by Lemma~\ref{lem:halfbraidingisbraiding}.
Conversely, if \changed{a} lax monoidal functor $F \colon \cat{C} \to \cat{D}$ preserves all half-braidings, then it induces a functor from the centre of $\cat{C}$ to the centre of $\cat{D}$ (but this functor need not be lax monoidal without further assumptions).

\begin{lemma}\label{lem:functorspreservingZI}
  Let $F \colon \cat{C} \to \cat{D}$ be an lax monoidal functor with coherence morphisms $\theta_{A,B} \colon F(A) \otimes F(B) \to F(A \otimes B)$ and $\theta_I \colon I \to F(I)$.
  Suppose that $F$ preserves half-braidings of central idempotents, and that $\theta_I$ and $\theta_{A,U}$ are invertible for all objects $A$ and $u \in \ZI(\cat{C})$.
  If $u \colon U \to I$ is a central idempotent in $\cat{C}$, then
  \[
    F(U) 
    \stackrel{F(u)}{\longrightarrow}
    F(I)
    \stackrel{\theta_I^{-1}}{\longrightarrow}
    I
  \]
  is a central idempotent in $\cat{D}$.
  This induces a semilattice morphism $\ZI(\cat{C}) \to \ZI(\cat{D})$ if additionally:
  \begin{align*}
    F(\sigma^I)_D & = (D \otimes \theta_I) \circ \rho_D^{-1} \circ \lambda_D \circ (\theta_I^{-1} \otimes D) \\
    F(\sigma^{U \otimes V})_D & = (D \otimes \theta_{U,V}) \circ \big( F(\sigma_U)_D \otimes F(V) \big) \circ \big( F(U) \otimes F(\sigma^V)_D \big) \circ (\theta_{U,V}^{-1} \otimes D)
  \end{align*}
\end{lemma}
\begin{proof}
  The central equation~\eqref{eq:central} holds:
  \[\begin{tikzcd}[column sep=2cm]
    {F(U) \otimes F(U)} & {F(I) \otimes F(U)} & {I \otimes F(U)} \\
    {F(U \otimes U)} & {F(I \otimes U)} & {F(U)} \\
    {F(U \otimes U)} & {F(U \otimes I)} & {F(U)} \\
    {F(U) \otimes F(U)} & {F(U) \otimes F(I)} & {F(U) \otimes I}
    \arrow["{F(u) \otimes F(U)}", from=1-1, to=1-2]
    \arrow["{\theta_I^{-1} \otimes F(U)}", from=1-2, to=1-3]
    \arrow["{\lambda_{F(U)}}", from=1-3, to=2-3]
    \arrow["{F(U) \otimes F(u)}"', from=4-1, to=4-2]
    \arrow["{F(U) \otimes \theta_I^{-1}}"', from=4-2, to=4-3]
    \arrow["{\rho_{F(U)}}"', from=4-3, to=3-3]
    \arrow[Rightarrow, no head, from=3-3, to=2-3]
    \arrow["{\theta_{U,U}}", from=4-1, to=3-1]
    \arrow["{\theta_{U,U}}"', from=1-1, to=2-1]
    \arrow["{F(\sigma_U)}"', from=2-1, to=3-1]
    \arrow["{F(\sigma_I)}", from=2-2, to=3-2]
    \arrow["{F(U \otimes u)}"', from=3-1, to=3-2]
    \arrow["{F(u \otimes U)}"', from=2-1, to=2-2]
    \arrow["{\theta_{I,U}}", from=1-2, to=2-2]
    \arrow["{\theta_{U,I}}"', from=4-2, to=3-2]
    \arrow["{F(\rho)}"', from=3-2, to=3-3]
    \arrow["{F(\lambda)}", from=2-2, to=2-3]
  \end{tikzcd}\]
  To see that~\eqref{eq:central} is invertible:
  \[\begin{tikzcd}
    & {F(U \otimes U)} & {F(U) \otimes F(U)} \\
    {F(U)} & {F(U \otimes I)} & {F(U) \otimes F(I)} \\
    & {F(U)} & {F(U) \otimes I}
    \arrow["{F(\rho_U)}", from=2-2, to=3-2]
    \arrow["{\rho_{F(U)}}", from=3-3, to=3-2]
    \arrow["{F(U) \otimes \theta_I^{-1}}", from=2-3, to=3-3]
    \arrow["{\theta_{U,I}^{-1}}", from=2-2, to=2-3]
    \arrow["{\theta_{U,U}^{-1}}", from=1-2, to=1-3]
    \arrow["{F(U) \otimes F(u)}", from=1-3, to=2-3]
    \arrow["{F(U \otimes u)}", from=1-2, to=2-2]
    \arrow["{F\big((U \otimes u)^{-1}\big)}", from=2-1, to=1-2]
    \arrow[Rightarrow, no head, from=2-1, to=3-2]
  \end{tikzcd}\]
  By assumption $F(\sigma)$ extends to a half-braiding.

  That $F(u \wedge v) = F(u) \wedge F(v)$ follows because $\theta_{U,V}$ is invertible and from the condition on $F(\sigma^{U \otimes V})$. 
  Finally, top element of the semilattice is preserved by the condition on $F(\sigma^I)$.
\end{proof}

We say that a lax monoidal functor \emph{preserves central idempotents} when it satisfies the conditions of the previous lemma. Braided strong monoidal functors between braided monoidal categories do so automatically.

\begin{definition}\label{def:moncat}
  Write $\cat{MonCat}$ for the category of (small) monoidal categories and lax monoidal functors $F$ whose coherence morphism $I \to F(I)$ is invertible.
  Write $\cat{MonCat_s}$ for the subcategory of stiff monoidal categories and lax monoidal functors $F$ that preserve central idempotents.
  Write $\cat{MonCat}_{\cat{fj}}$ for the further subcategory of monoidal categories with universal finite joins of central idempotents and   functors that preserve finite joins of central idempotents, and $\cat{MonCat}_{\cat{j}}$ for the subcategory of monoidal categories \changed{whose central idempotents have universal joins and form a spatial frame} and functors that preserve joins of central idempotents.
\end{definition}

Example~\ref{ex:semilattices} above also gives a coreflection between $\cat{MonCat}_s$ and $\cat{SLat}$, the category of semilattices and semilattice morphisms, that is, functions that preserve meets and the greatest element.

\section{Base space}\label{sec:basespace}

We will build the representation for the case of universal finite joins first, starting with the base space in this section, followed by the presheaf, sheaf, and stalks in following sections. In Section~\ref{sec:local} below this will \changed{be} upgraded to universal joins. Therefore we first consider distributive lattices. As base space, we will take the prime spectrum of the distributive lattice of central idempotents~\cite{dickmannschwartztressl:spectralspaces,johnstone:stonespaces}.
Recall that a \emph{prime filter} of a lattice $L$ is a nonempty upward-closed and downward-directed \changed{proper} subset $P$ where $u \vee v \in P$ implies $u \in P$ or $v \in P$; equivalently, $P$ is the inverse image of 1 under a lattice homomorphism $L \to \{0,1\}$.

\begin{definition}\label{def:primespectrum}
  Let $L$ be a distributive lattice with a least element.	
  Its \emph{prime spectrum} is a topological space $\mathrm{Spec}(L)$, whose points are prime filters $P \subseteq L$, and whose topology is generated by a basis consisting of the sets
  \[
    B_u = \{P \in \Spec(L) \mid u \in P \}
  \]
  where $u$ ranges over $L$.
\end{definition}

\begin{lemma}\label{lem:basiscompact}
  The basic opens $B_u$ of the prime spectrum of a lattice $L$ are compact.
\end{lemma}
\begin{proof}
  The frame $\mathcal{O}(\Spec(L))$ of opens of $\Spec(L)$ is in bijection with the frame $\Idl(L)$ of ideals of $L$.
  The map that sends an element $v \in L$ to its principal ideal $\downset v = \{u \in L \mid u \leq v\}$ embeds $L$ in $\Idl(L)$ as a lattice. On the topological side, it sends $u \in L$ to $B_u$, establishing an isomorphism between $L$ and the basis $\{ B_u \mid u \in L \} \subseteq \mathcal{O}(\Spec(L))$. Because the principal ideals are compact in $\Idl(L)$, so are $B_u$ in $\mathcal{O}(\Spec(L))$.
\end{proof}

The following lemma holds more generally for sheaves valued in any category $\cat{V}$, but we spell it out for $\cat{MonCat}$-valued sheaves.

\begin{lemma}\label{lem:equaliser}
	To specify a sheaf $F \colon \mathcal{O}(\Spec(L))\op \to \cat{MonCat}$ for a distributive lattice $L$, it suffices to give a presheaf $F \colon \{ B_u \mid u \in L\}\op \to \cat{MonCat}$, such that $F(B_0)$ is terminal in $\cat{MonCat}$ and the following is an equaliser in $\cat{MonCat}$:
  \[\begin{pic}[xscale=5]
    \node (l) at (-1.3,0) {$F(B_{u \vee v})$};
    \node (m) at (0,0) {$F(B_u) \times F(B_v)$};
    \node (r) at (.95,0) {$F(B_{u \wedge v})$};
    \draw[->] (l) to node[above]{$\langle F(B_u \!\subseteq\! B_{u \vee v}), F(B_v \!\subseteq\! B_{u \vee v}) \rangle$} (m);
    \draw[->] ([yshift=-1.5mm]m.east) to node[below]{$F(B_{u \wedge v} \!\subseteq\! B_v) \circ \pi_2$} ([yshift=-1.5mm]r.west);
    \draw[->] ([yshift=1.5mm]m.east) to node[above]{$F(B_{u \wedge v} \!\subseteq\! B_u) \circ \pi_1$} ([yshift=1.5mm]r.west);
  \end{pic}\]  
\end{lemma}
\begin{proof}
  A sheaf on a basis uniquely determines a sheaf on the whole topological space~\cite[Theorem~II.1.3]{maclanemoerdijk:sheaves}.
  The presheaf $F \colon \{B_u \mid u \in L\}\op \to \cat{MonCat}$ provides this data.
  It furthermore needs to satisfy the sheaf condition on the covers of basic opens by basic opens, $B_v = \bigcup_{u \in J} B_u$ with $J \subseteq L$.
  Since the basic open $B_v$ is compact by Lemma~\ref{lem:basiscompact}, it suffices to consider finitary and hence binary covers, $B_{u \vee v} = B_u \cup B_v$ for $u, v \in L$, and nullary covers, $B_0 = \bigvee \emptyset$.
  The sheaf condition for covers of this form says exactly that the functor in the statement is an equaliser, and $F(B_0)$ is isomorphic to the terminal monoidal category $\cat{1}$.
\end{proof}

\section{Structure presheaf}\label{sec:presheaf}

In this section we define the structure presheaf of a monoidal category.
For a mere monoidal category \changed{$\cat{C}$} this is a presheaf on the semilattice \changed{$\ZI(\cat{C})$} of central idempotents, which in the presence of universal finite joins extends to a presheaf on the topological space \changed{$\Spec(\ZI(\cat{C}))$} of the previous section.
So we define a presheaf of monoidal categories on the central idempotents by describing its behaviour over basic open sets. We will generalise the filter-quotient construction from topos theory to monoidal categories~\cite[V.9]{maclanemoerdijk:sheaves}.

Any central idempotent $U$ in a monoidal category $\cat{C}$ is a (commutative) comonoid with comultiplication $\tinycomult$ and counit $\tinycounit$. Hence $- \otimes U$ is a comonad on $\cat{C}$. We are interested in its co-Kleisli category.

\begin{definition}
  Let $\cat{C}$ be a monoidal category, and $u \colon U \to I$ a chosen representative of a central idempotent. Define a category $\cat{C}\restrict{u}$ as follows:
  \begin{itemize}
    \item objects are those of $\cat{C}$;
    \item morphisms $A\to B$ in $\cat{C}\restrict{u}$ are morphisms $A \otimes U \to B$ in $\cat{C}$;
    \item composition of $f \colon A \to B$ and $g \colon B \to C$ is:
      \[\begin{pic}
		\node[morphism,width=8mm] (f) at (0,0) {$f$};
		\node[morphism,width=8mm,anchor=south west] (g) at ([yshift=5mm]f.north east) {$g$};
		\draw (f.north east) to node[left]{$B$} (g.south west);
		\node[dot] (d) at ([xshift=3mm,yshift=-3mm]f.south east) {};
		\draw (d) to ++(0,-.4) node[below]{$U$};
		\draw (d) to[out=180,in=-90] (f.south east);
		\draw (d) to[out=0,in=-90,looseness=.7] (g.south east);
		\draw (f.south west) to ++(0,-.7) node[below]{$A$};
		\draw (g.north) to ++(0,.5) node[above]{$C$};
      \end{pic}\]
	\item identity on $A$ is $\tinyid\!\raisebox{-1.25pt}{\tinycounit}$.
  \end{itemize}
\end{definition}

\begin{lemma}\label{lem:Cumonoidal}
  If $\cat{C}$ is a monoidal category, then so is $\cat{C}\restrict{u}$.
\end{lemma}
\begin{proof}
  The monoidal structure of $\cat{C}\restrict{u}$ is defined as follows. The tensor product of morphisms $f \colon A \to B$ and $g \colon C \to D$ is:
  \[\begin{pic}
  \node[morphism,width=7mm] (f) at (0,0) {$f$};
  \node[morphism,width=7mm] (g) at (1.5,0) {$g$};
  \node[dot] (d) at (1.5,-.7) {};
  \draw (f.north) to ++(0,.4) node[above]{$B$};
  \draw (g.north) to ++(0,.4) node[above]{$D$};
  \draw (g.south east) to[out=-90,in=0] (d);
  \draw (d) to (1.5,-1) node[below]{$U$};
  \draw (f.south west) to ++(0,-.8) node[below]{$A$};
  \draw (g.south west) to[out=-90,in=90] ++(-.6,-.8) node[below]{$C$};
  \draw[halo] (f.south east) to[out=-90,in=180] (d);
  \node[morphism,width=7mm] (f) at (0,0) {$f$};
  \node[dot] (d) at (1.5,-.7) {};
  \end{pic}\]
  The associator and left and right unitors are given by:
  \[\begin{pic}
      \draw (0,0) node[below]{$A$} to ++(0,1);
      \draw (.5,0) node[below]{$B$} to ++(0,1);
      \draw (1,0) node[below]{$C$} to ++(0,1);
      \draw (1.5,0) node[below]{$U$} to ++(0,.5) node[dot]{};
    \end{pic}
    \qquad\qquad
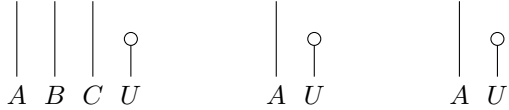
    \begin{pic}
      \draw (0,0) node[below]{$A$} to ++(0,1);
      \draw (.5,0) node[below]{$U$} to ++(0,.5) node[dot]{};
    \end{pic}
    \qquad\qquad
    \begin{pic}
      \draw (0,0) node[below]{$A$} to ++(0,1);
      \draw (.5,0) node[below]{$U$} to ++(0,.5) node[dot]{};
    \end{pic}
  \]
  Easy graphical manipulation establishes naturality, the interchange law, as well as the triangle and pentagon equations.
\end{proof}

\begin{lemma}\label{lem:slice}
  If $\cat{C}$ is cartesian, there is a monoidal equivalence $\cat{C}\restrict{u} \simeq \cat{C}/U$.
\end{lemma}
\begin{proof}
  First, notice that $A \otimes U \colon A \to A \otimes U$ is an isomorphism $A \simeq A \otimes U$ in $\cat{C}\restrict{u}$ with inverse $A \otimes u \otimes u \colon A \otimes U \to A$. Therefore, $\cat{C}\restrict{u}$ is equivalent to its full subcategory $\changed{\cat{C}|_u}$ of objects of the form $A \otimes U$. Moreover, this equivalence preserves tensor products on the nose.

  There is a functor $\changed{\cat{C}|_u} \to \cat{C}/U$ that sends an object $A \otimes U$ to $\pi_2 \colon A \otimes U \to U$, and that sends a morphism $f \colon A \otimes U \otimes U \to B \otimes U$ to $f \circ (A \otimes U \otimes u)^{-1}$. It is monoidal.
  There is also a functor $\cat{C}/U \to \cat{C}\restrict{u}$ that sends an object $! \colon A \to U$ to $A \otimes U$, and that sends a morphism $f \colon A \to B$ to $f \otimes U \otimes u$. This functor is monoidal too.
  These two functors form an equivalence, because $A \otimes U \otimes U \simeq A \otimes U$ in $\cat{C}\restrict{u}$.
\end{proof}

\begin{lemma}\label{lem:restrictionfunctor}
  If $u \leq v$ are central idempotents in a monoidal category $\cat{C}$, 
  there is a strict monoidal functor
  $\cat{C}\restrict{u \leq v} \colon \cat{C}\restrict{v} \to \cat{C}\restrict{u}$ that acts as identity on objects and on morphisms as:
  \[\begin{pic}
  	\node[morphism,width=8mm] (f) at (0,0) {$f$};
  	\draw (f.north) to ++(0,+.2) node[above]{$B$};
  	\draw (f.south west) to ++(0,-.7) node[below]{$A$};
  	\draw (f.south east) to ++(0,-.7) node[below]{$V$};
    \end{pic}
    \qquad\longmapsto\qquad
    \begin{pic}
  	\node[morphism,width=8mm] (f) at (0,0) {$f$};
  	\draw (f.north) to ++(0,+.2) node[above]{$B$};
  	\draw (f.south west) to ++(0,-.7) node[below]{$A$};
  	\node[dot] (d) at ([yshift=-4.5mm]f.south east) {};
  	\draw (f.south east) to node[right=-.5mm]{$V$} (d);
  	\draw (d) to ++(0,-.25) node[below]{$U$};
    \end{pic}
  \]
\end{lemma}
\begin{proof}
We first show functoriality.
To verify that $\cat{C}\restrict{u \leq v}$ preserves identity morphisms:
\[
  \begin{pic}
	\node[morphism,width=6mm] (f) at (0,0) {$\id[A]$};
	\node[dot] (d) at ([yshift=-5mm]f.south east) {};
	\draw (f.north) to ++(0,.3) node[above]{$A$};
	\draw (f.south east) to node[right=-.7mm]{$V$} (d);
	\draw (d) to ++(0,-.3) node[below]{$U$};
	\draw (f.south west) to ++(0,-.8) node[below]{$A$};
  \end{pic}
  =
  \begin{pic}
	\node[dot] (t) at (.5,.6) {};
	\node[dot] (b) at (.5,0) {};
	\draw (b) to node[right=-.7mm]{$V$} (t);
	\draw (b) to ++(0,-.4) node[below]{$U$};
	\draw (0,-.4) node[below]{$A$} to ++(0,1.5) node[above]{$A$};
  \end{pic}
  =
  \begin{pic}
	\draw (0,-.4) node[below]{$A$} to ++(0,1.5) node[above]{$A$};
	\draw (.5,-.4) node[below]{$U$} to ++(0,.7) node[dot]{};
  \end{pic}
\]
The fact that
\[\cat{C}\restrict{u\leq v}(g \circ_v f)
 =
 \cat{C}\restrict{u \leq v}(g) \circ_u \cat{C}\restrict{u \leq v}(f)\]
is proven by:
\[
  \begin{pic}
	\node[morphism,width=6mm] (gf) {$g \circ_v f$};
	\node[dot] (d) at ([yshift=-7mm]gf.south east) {};
	\draw (gf.south east) to node[right=-1mm]{$V$} (d);
	\draw (d) to ++(0,-.45) node[below]{$U$};
	\draw (gf.south west) to ++(0,-1.15) node[below]{$A$};
	\draw (gf.north) to ++(0,1.15) node[above]{$C$};
  \end{pic}
  =
  \begin{pic}
	\node[morphism,width=10mm] (g) at (0,1) {$g$};
	\node[morphism,width=6mm] (f) at ([yshift=-6mm]g.south west) {$f$};
	\node[dot] (t) at ([xshift=3mm,yshift=-4mm]f.south east) {};
	\node[dot] (d) at ([yshift=-5mm]t) {};
	\draw (g.north) to ++(0,.3) node[above]{$C$};
	\draw ([xshift=-1mm]g.south west) to node[left]{$B$} ([xshift=-1mm]f.north);
	\draw (f.south east) to[out=-90,in=180] (t);
	\draw ([xshift=1mm]g.south east) to ++(0,-.8) to[out=-90,in=0] (t);
	\draw (t) to node[right=-.5mm]{$V$} (d);
	\draw (d) to ++(0,-.3) node[below]{$U$};
	\draw (f.south west) to ++(0,-1.2) node[below]{$A$};
  \end{pic}
  =
  \begin{pic}
	\node[morphism,width=10mm] (g) at (0,1) {$g$};
	\node[morphism,width=6mm] (f) at ([yshift=-6mm]g.south west) {$f$};
	\node[dot] (l) at ([yshift=-5mm]f.south east) {};
	\node[dot] (r) at ([xshift=1mm,yshift=-13mm]g.south east) {};
	\node[dot] (d) at ([xshift=3.3mm,yshift=-4mm]l) {};
	\draw (g.north) to ++(0,.3) node[above]{$C$};
	\draw ([xshift=-1mm]g.south west) to node[left]{$B$} ([xshift=-1mm]f.north);
	\draw (f.south east) to node[right=-1mm]{$V$} (l);
	\draw ([xshift=1mm]g.south east) to node[right=-1mm]{$V$} (r);
	\draw (d) to ++(0,-.3) node[below]{$U$};
	\draw (l) to[out=-90,in=180] (d);
	\draw (r) to[out=-90,in=0] (d);
	\draw (f.south west) to ++(0,-1.2) node[below]{$A$};
  \end{pic}
  =
  \begin{pic}
	\node[morphism,width=10mm] (g) at (0,1) {$\cat{C}\restrict{u \leq v}(g)$};
	\node[morphism,width=6mm] (f) at ([yshift=-6mm]g.south west) {$\cat{C}\restrict{u \leq v}(f)$};
	\node[dot] (d) at ([xshift=4.2mm,yshift=-2mm]l) {};
	\draw (g.north) to ++(0,.3) node[above]{$C$};
	\draw ([xshift=-1mm]g.south west) to node[left]{$B$} ([xshift=-1mm]f.north);
	\draw ([xshift=-.7mm]f.south east) to[out=-90,in=180,looseness=.5] (d);
	\draw ([xshift=1mm]g.south east) to[out=-90,in=0,looseness=.5] (d);
	\draw (d) to ++(0,-.5) node[below]{$U$};
	\draw (f.south west) to ++(0,-1.2) node[below]{$A$};
  \end{pic}
\]

  As for monoidality,
  for $f \colon A \otimes V \to B$ and $g \colon C \otimes V \to D$ in $\cat{C}$:
  \[
  \cat{C}\restrict{u \leq v}\big(f \otimes g\big)
  =
  \begin{pic}
	\node[morphism,width=9mm] (fg) at (0,0) {$f \otimes g$};
	\draw ([xshift=-2mm]fg.north) to ++(0,.2) node[above]{$B$};
	\draw ([xshift=2mm]fg.north) to ++(0,.2) node[above]{$D$};
	\draw (fg.south west) to ++(0,-1.25) node[below]{$A$};
	\draw (fg.south) to ++(0,-1.25) node[below]{$C$};
    \node[dot] (m) at ([yshift=-7mm]fg.south east) {};
	\draw (fg.south east) to node[right=-.5mm]{$V$} (m);
	\draw (m) to ++(0,-.55) node[below]{$U$};	
  \end{pic}
  =
  \begin{pic}
	\node[morphism,width=5mm] (f) at (0,0) {$f$};
	\node[morphism,width=5mm] (g) at (1,0) {$g$};
    \node[dot] (d) at ([xshift=1mm,yshift=-4mm]g.south west) {};
    \node[dot] (m) at ([xshift=1mm,yshift=-9mm]g.south west) {};
	\draw (f.north) to ++(0,.2) node[above]{$B$};
	\draw (g.north) to ++(0,.2) node[above]{$D$};
	\draw (g.south east) to[out=-90,in=0]   (d);
	\draw (f.south west) to ++(0,-1.2) node[below]{$A$};
	\draw (g.south west) to[out=-90,in=90] ++(-.5,-1.2) node[below]{$C$};
	\draw (d) to node[right=-.5mm]{$V$} (m);
	\draw (m) to ++(0,-.3) node[below]{$U$};	
  \draw[halo] (f.south east) to[out=-90,in=180] (d);
  \node[morphism,width=5mm] (f) at (0,0) {$f$};
    \node[dot] (d) at ([xshift=1mm,yshift=-4mm]g.south west) {};
  \end{pic}
  =
  \begin{pic}
	\node[morphism,width=5mm] (f) at (0,0) {$f$};
	\node[morphism,width=5mm] (g) at (1,0) {$g$};
    \node[dot] (m1) at ([yshift=-5mm]f.south east) {};
    \node[dot] (m2) at ([yshift=-5mm]g.south east) {};
    \node[dot] (d) at ([yshift=-9mm]g.south west) {};
	\draw (f.north) to ++(0,.2) node[above]{$B$};
	\draw (g.north) to ++(0,.2) node[above]{$D$};
	\draw (f.south east) to node[right=-.5mm]{$V$} (m1);
	\draw (g.south east) to node[right=-.5mm]{$V$} (m2);
	\draw (m1) to[out=-90,in=180] (d);
	\draw (m2) to[out=-90,in=0]   (d);
	\draw (f.south west) to ++(0,-1.2) node[below]{$A$};
	\draw[halo] (g.south west) to[out=-90,in=90] ++(-.5,-1.2) node[below]{$C$};
  \node[morphism,width=5mm] (g) at (1,0) {$g$};
	\draw (d) to ++(0,-.3) node[below]{$U$};	
  \end{pic}
  =
  \cat{C}\restrict{u \leq v}(f) \otimes \cat{C}\restrict{u \leq v}(g)
  \]
  This functor preserves the tensor unit, and hence is strict monoidal.
\end{proof}

The special case $v=1$ gives a functor $\cat{C} \to \cat{C}\restrict{u}$ that acts as the identity on objects and acts on morphisms as:
\[\begin{pic}
	\node[morphism,width=6mm] (f) at (0,0) {$f$};
	\draw (f.north) to ++(0,+.2) node[above]{$B$};
	\draw (f.south) to ++(0,-.2) node[below]{$A$};
  \end{pic}
  \qquad\longmapsto\qquad
  \begin{pic}
	\node[morphism,width=6mm] (f) at (0,0) {$f$};
	\draw (f.north) to ++(0,+.2) node[above]{$B$};
	\draw (f.south) to ++(0,-.2) node[below]{$A$};
	\draw (.8,0) node[dot]{} to ++(0,-.4) node[below]{$U$};
  \end{pic}
\]

\begin{lemma}\label{lem:lari}
  If $u \leq v$ are central idempotents in a monoidal category $\cat{C}$, there is an oplax monoidal functor $\cat{C}\Restrict{u \leq v} \colon \cat{C}\restrict{u} \to \cat{C}\restrict{v}$, that acts on objects as $A \mapsto A \otimes U$, and on morphisms as:
  \[\begin{pic}
    \node[morphism,width=8mm] (f) at (0,0) {$f$};
    \draw (f.north) to ++(0,+.2) node[above]{$B$};
    \draw (f.south west) to ++(0,-.7) node[below]{$A$};
    \draw (f.south east) to ++(0,-.7) node[below]{$U$};
    \end{pic}
    \qquad\longmapsto\qquad
    \begin{pic}
    \node[morphism,width=8mm] (f) at (0,0) {$f$};
    \node[dot] (d) at (.65,-.5) {};
    \draw (f.north) to ++(0,+.2) node[above]{$B$};
    \draw (f.south west) to ++(0,-.7) node[below]{$A$};
    \draw (d) to[out=180,in=-90] (f.south east);
    \draw (d) to ++(0,-.4) node[below]{$U$};
    \draw (d) to[out=0,in=-90] ++(.4,.5) to ++(0,.4) node[above]{$U$};
    \draw (1.3,-.5) node[dot]{} to ++(0,-.4) node[below]{$V$};
    \end{pic}
  \]
  This functor is left adjoint to $\cat{C}\restrict{u \leq v}$, and the unit of this adjunction is invertible.
\end{lemma}
\begin{proof}
  It is straightforward to verify that $\cat{C}\Restrict{u \leq v}$ indeed preserves identities and composition. 
  It is oplax monoidal with $\theta_I = I \otimes u \otimes v \colon \cat{C}\restrict{u \leq v}(I) \to I$ and:
  \[\theta_{A,B} = 
    \begin{pic}
    \node[dot] (d) at (0,0) {};
    \draw (.5,-.5) node[below]{$V$} to ++(0,.5) node[dot]{};
    \draw (-.5,-.5) node[below]{$B$} to[out=90,in=-90] ++(.45,1) node[above]{$B$};
    \draw (-1,-.5) node[below]{$A$} to ++(0,1) node[above]{$A$};
    \draw (d) to ++(0,-.5) node[below]{$U$};
    \draw (d) to[out=0,in=-90] ++(.4,.5) node[above]{$U$};
    \draw[halo] (d) to[out=180,in=-90] ++(-.5,.5) node[above]{$U$};
    \end{pic}
    \colon \cat{C}\Restrict{u \leq v}(A \otimes B) \to \cat{C}\Restrict{u \leq v}(A) \otimes \cat{C}\Restrict{u \leq v}(B)
  \]
  It is routine to verify that $\theta_{A,B}$ and $\theta_I$ are natural and respect the coherence isomorphisms of $\cat{C}\restrict{u}$ and $\cat{C}\restrict{v}$. 
  Notice that in fact $\theta_{A,B}$ is invertible, with inverse $\theta_{A,B}^{-1} = A \otimes u \otimes B \otimes U \otimes v \colon \cat{C}\Restrict{u \leq v}(A) \otimes \cat{C}\Restrict{u \leq v}(B) \to \cat{C}\Restrict{u \leq v}(A \otimes B)$, but that $\theta_I$ is not invertible.

  The fact that $\cat{C}\Restrict{u \leq v}$ is left adjoint to $\cat{C}\restrict{u \leq v}$ simply means:
  \[
    \cat{C}\restrict{v}(A \otimes U,B)
    = \cat{C}(A \otimes U \otimes V,B)
    \simeq \cat{C}(A \otimes U,B)
    = \cat{C}\restrict{v}(A \otimes U,B)
  \]
  This holds because $u \wedge v = u$ as $u \leq v$, and is easily seen to be natural in $A$ and $B$.
  The unit of the adjunction is $\eta_A = A \otimes U \colon A \to A \otimes U$ in $\cat{C}\restrict{u}$. It is inverted by $\eta_A^{-1}=A \otimes u \otimes u \colon A \otimes U \to A$ in $\cat{C}\restrict{u}$.
\end{proof}

To finish this section on the structure presheaf, we determine the central idempotents of the categories $\cat{C}\restrict{u}$.

\begin{lemma}\label{lem:zi:Cu}
  If $u$ is a central idempotent in a monoidal category $\cat{C}$, then there is an isomorphism $\ZI(\cat{C}\restrict{u}) \simeq \ZI(\cat{C}) \cap \downset{u}$ of semilattices.
\end{lemma}
\begin{proof}
  Let $q \colon Q \to I$ be a morphism in $\cat{C}\restrict{u}$. 
  That is, $q$ is a morphism $Q \otimes U \to I$ in $\cat{C}$.
  By definition $q$ satisfies~\eqref{eq:central} in $\cat{C}\restrict{u}$ if and only if the following holds in $\cat{C}$:
  \[
    \begin{pic}
      \node[morphism,width=5mm] (q) at (0,.5) {$q$};
      \draw (q.south west) to ++(0,-.5) node[below]{$Q$};
      \draw (.2,-.2) node[below]{$Q$} to[out=90,in=-90] ++(.5,.5) to ++(0,.6) node[above]{$Q$};
      \draw[halo] (q.south east) to[out=-90,in=90,looseness=.8] ++(.5,-.5) node[below]{$U$};
      \node[morphism,width=5mm] (q) at (0,.5) {$q$};
    \end{pic}
    =
    \begin{pic}
      \node[morphism,width=5mm] (q) at (0,.5) {$q$};
      \draw (q.south west) to ++(0,-.5) node[below]{$Q$};
      \draw (q.south east) to ++(0,-.5) node[below]{$U$};
      \draw (-.6,-.2) node[below]{$Q$} to ++(0,1.1) node[above]{$Q$};
    \end{pic}
  \]
  But by Lemma~\ref{lem:braiding}, this means precisely that $q \colon Q \otimes U \to I$ satisfies~\eqref{eq:central} in $\cat{C}$. 

  Now, $q$ is idempotent in $\cat{C}\restrict{u}$ if and only if there is a morphism $f \colon Q \otimes U \to Q \otimes Q$ in $\cat{C}$ satisfying:
  \[
    \begin{pic}
      \node[morphism,width=4mm] (q) at (0,0) {$q$};
      \node[morphism,width=10mm] (f) at (0,.6) {$f$};
      \node[dot] (d) at (.35,-.5) {};
      \draw (q.south west) to ++(0,-.6) node[below]{$Q$};
      \draw (q.south east) to[out=-90,in=180,looseness=.7] (d);
      \draw ([xshift=2mm]f.south east) to[out=-90,in=0,looseness=.5] (d);
      \draw (d.south) to ++(0,-.2) node[below]{$U$};
      \draw ([xshift=-2mm]f.south west) to ++(0,-1.2) node[below]{$Q$};
      \draw (f.north west) to ++(0,.3) node[above]{$Q$};
      \draw (f.north east) to ++(0,.3) node[above]{$Q$};
    \end{pic}
    \;=
    \begin{pic}
      \draw (0,0) node[below]{$Q$} to ++(0,2) node[above]{$Q$};
      \draw (.4,0) node[below]{$Q$} to ++(0,2) node[above]{$Q$};
      \draw (.8,0) node[below]{$U$} to ++(0,1) node[dot]{};
    \end{pic}
    \qquad\qquad\qquad
    \begin{pic}
      \node[morphism,width=6mm] (f) at (0,0) {$f$};
      \node[morphism,width=4mm,anchor=south west] (q) at ([yshift=3mm]f.north east) {$q$};
      \draw (f.north east) to (q.south west);
      \node[dot] (d) at (.4,-.5) {};
      \draw (f.south east) to[out=-90,in=180,looseness=.5] (d);
      \draw ([xshift=1mm]q.south east) to[out=-90,in=0,looseness=.5] (d);
      \draw (d) to ++(0,-.3) node[below]{$U$};
      \draw (f.south west) to ++(0,-.6) node[below]{$Q$};
      \draw (f.north west) to ++(0,1) node[above]{$Q$};
    \end{pic}
    \;=
    \begin{pic}
      \draw (.4,0) node[below]{$Q$} to ++(0,2) node[above]{$Q$};
      \draw (.8,0) node[below]{$U$} to ++(0,1) node[dot]{};
    \end{pic}
  \]
  Similarly, $q$ is idempotent in $\cat{C}$ if and only if there is $g \colon Q \otimes U \to Q \otimes U \otimes Q \otimes U$ in $\cat{C}$ satisfying:
  \[
    \begin{pic}
      \node[morphism,width=12mm] (g) at (0,0) {$g$};
      \node[morphism,width=4mm] (q) at (.4,.7) {$q$};
      \draw (q.south west) to ++(0,-.3);
      \draw (q.south east) to ++(0,-.3);
      \draw ([xshift=-2mm]g.south) to ++(0,-.3) node[below]{$Q$};
      \draw ([xshift=2mm]g.south) to ++(0,-.3) node[below]{$U$};
      \draw ([xshift=-2mm]g.north west) to ++(0,1) node[above]{$\smash{Q}$};
      \draw ([xshift=2mm]g.north west) to ++(0,1) node[above]{$U$};
    \end{pic}
    \;=
    \begin{pic}
      \draw (.4,0) node[below]{$Q$} to ++(0,1.8) node[above]{$\smash{Q}$};
      \draw (.8,0) node[below]{$U$} to ++(0,1.8) node[above]{$U$};
    \end{pic}
    \qquad\qquad\qquad
    \begin{pic}
      \node[morphism,width=12mm] (g) at (0,0) {$g$};
      \node[morphism,width=4mm] (q) at (.4,-.7) {$q$};
      \draw ([xshift=-.5mm]q.south west) to ++(0,-.3) node[below]{$Q$};
      \draw ([xshift=.5mm]q.south east) to ++(0,-.3) node[below]{$U$};
      \draw ([xshift=-6.75mm]g.north) to ++(0,.3) node[above]{$\smash{Q}$};
      \draw ([xshift=-2.25mm]g.north) to ++(0,.3) node[above]{$U$};
      \draw ([xshift=2.25mm]g.north) to ++(0,.3) node[above]{$\smash{Q}$};
      \draw ([xshift=6.75mm]g.north) to ++(0,.3) node[above]{$U$};
      \draw ([xshift=-2mm]g.south west) to ++(0,-1) node[below]{$Q$};
      \draw ([xshift=2.5mm]g.south west) to ++(0,-1) node[below]{$U$};
    \end{pic}
    \;=
    \begin{pic}
      \draw (.4,0) node[below]{$Q$} to ++(0,1.8) node[above]{$\smash{Q}$};
      \draw (.8,0) node[below]{$U$} to ++(0,1.8) node[above]{$U$};
      \draw (1.2,0) node[below]{$Q$} to ++(0,1.8) node[above]{$\smash{Q}$};
      \draw (1.6,0) node[below]{$U$} to ++(0,1.8) node[above]{$U$};
    \end{pic}
  \]
  But these two properties are equivalent via:
  \[
    \begin{pic}
      \node[morphism,width=5mm] (f) {$f$};
      \draw (f.north west) to ++(0,.7) node[above]{$Q$};
      \draw (f.north east) to ++(0,.7) node[above]{$Q$};
      \draw (f.south west) to ++(0,-.5) node[below]{$Q$};
      \draw (f.south east) to ++(0,-.5) node[below]{$U$};
    \end{pic}
    \;=
    \begin{pic}
      \node[morphism,width=10mm] (g) {$g$};
      \draw ([xshift=-6mm]g.north) to ++(0,.7) node[above]{$Q$};
      \draw ([xshift=-2mm]g.north) to ++(0,.3) node[dot]{};
      \draw ([xshift=2mm]g.north) to ++(0,.7) node[above]{$Q$};
      \draw ([xshift=6mm]g.north) to ++(0,.3) node[dot]{};
      \draw (g.south west) to ++(0,-.5) node[below]{$Q$};
      \draw (g.south east) to ++(0,-.5) node[below]{$U$};
    \end{pic}
    \qquad\qquad\qquad
    \begin{pic}
      \node[morphism,width=10mm] (g) {$g$};
      \draw ([xshift=-6mm]g.north) to ++(0,.5) node[above]{$\smash{Q}$};
      \draw ([xshift=-2mm]g.north) to ++(0,.5) node[above]{$U$};
      \draw ([xshift=2mm]g.north) to ++(0,.5) node[above]{$\smash{Q}$};
      \draw ([xshift=6mm]g.north) to ++(0,.5) node[above]{$U$};
      \draw (g.south west) to ++(0,-.7) node[below]{$Q$};
      \draw (g.south east) to ++(0,-.7) node[below]{$U$};
    \end{pic}
    =    
    \begin{pic}
      \node[morphism,width=5mm] (f) at (0,0) {$f$};
      \node[dot] (l) at (.5,-.5) {};
      \node[dot] (r) at (.8,-.8) {};
      \draw (f.south east) to[out=-90,in=180] (l);
      \draw (l) to[out=-90,in=180] (r);
      \draw (r) to ++(0,-.3) node[below]{$U$};
      \draw (f.north west) to ++(0,.3) node[above]{$\smash{Q}$};
      \draw (f.south west) to ++(0,-.9) node[below]{$Q$};
      \draw (r) to[out=0,in=-90] ++(.3,.3) to ++(0,1) node[above]{$U$};
      \draw (f.north east) to[out=90,in=-90] ++(.5,.3) node[above]{$\smash{Q}$};
      \draw[halo] (l) to[out=0,in=-90] ++(.3,.3) to ++(0,.3) to[out=90,in=-90] ++(-.5,.4) node[above]{$U$};
    \end{pic}
  \]

  Finally, if $\tau_A \colon Q \otimes A \to A \otimes Q$ is a half-braiding in $\cat{C}\restrict{u}$, and $\sigma_A \colon U \otimes A \to A \otimes U$ is the given half-braiding in $\cat{C}$, then 
  \[
    Q \otimes U \otimes A
    \stackrel{Q \otimes \sigma_A}{\longrightarrow}
    Q \otimes A \otimes U
    \simeq
    Q \otimes A \otimes U \otimes U
    \stackrel{\tau_A \otimes U}{\longrightarrow}
    A \otimes Q \otimes U
  \]
  is a half-braiding in $\cat{C}\restrict{u}$.
  Thus any central idempotent $V$ in $\cat{C}$ induces a central idempotent $V$ in $\cat{C}\restrict{u}$, and any central idempotent $q$ in $\cat{C}\restrict{u}$ is represented by a central idempotent coming from $\cat{C}$ in this way.
\end{proof}

\begin{corollary}
  Let $\cat{C}$ be a monoidal monoidal category and $u \in \ZI(\cat{C})$.
  \begin{itemize}
    \item If $\cat{C}$ is stiff, then so is $\cat{C}\restrict{u}$.
    \item If $\cat{C}$ has finite joins of central idempotents, then so does $\cat{C}\restrict{u}$.
    \item If $\cat{C}$ has joins of central idempotents, then so does $\cat{C}\restrict{u}$.
  \end{itemize}
  The functor $\cat{C} \to \cat{C}\restrict{u}$ preserves joins of central idempotents.
\end{corollary}
\begin{proof}
  Follows from Definition~\ref{def:universaljoins} and Lemma~\ref{lem:zi:Cu}.
\end{proof}

\section{Structure sheaf}\label{sec:sheaf}

This section establishes that the structure presheaf is in fact a sheaf.
We start with checking the sheaf condition for binary joins.

\begin{proposition}\label{prop:stalkfunctor}
  If $\cat{C}$ is a monoidal category with universal finite joins of central idempotents, then the following is an equaliser in $\cat{MonCat}$: 
  \[\begin{pic}[xscale=5]
    \node (l) at (-1,0) {$\cat{C}\restrict{u \vee v}$};
    \node (m) at (0,0) {$\cat{C}\restrict{u} \times \cat{C}\restrict{v}$};
    \node (r) at (1,0) {$\cat{C}\restrict{u \wedge v}$};
    \draw[->] (l) to node[above]{$\langle \cat{C}\restrict{u \leq u \vee v}, \cat{C}\restrict{v \leq u \vee v} \rangle$} (m);
    \draw[->] ([yshift=-1mm]m.east) to node[below]{$\cat{C}\restrict{u \wedge v \leq v} \circ \pi_2$} ([yshift=-1mm]r.west);
    \draw[->] ([yshift=1mm]m.east) to node[above]{$\cat{C}\restrict{u \wedge v \leq u} \circ \pi_1$} ([yshift=1mm]r.west);
  \end{pic}\]
\end{proposition}
\begin{proof}
  We will prove that the following diagram is a pullback in $\cat{MonCat}$:
  \[\begin{pic}[xscale=4,yscale=1.5]
  	\node (tl) at (0,1) {$\cat{C}\restrict{u \vee v}$};
  	\node (tr) at (1,1) {$\cat{C}\restrict{u}$};
  	\node (bl) at (0,0) {$\cat{C}\restrict{v}$};
  	\node (br) at (1,0) {$\cat{C}\restrict{u \wedge v}$};
  	\draw[->] (tl) to node[above]{$\cat{C}\restrict{u \leq u \vee v}$} (tr);
  	\draw[->] (tl) to node[left]{$\cat{C}\restrict{v \leq u \wedge v}$} (bl);
  	\draw[->] (tr) to node[right]{$\cat{C}\restrict{u \wedge v \leq u}$} (br);
  	\draw[->] (bl) to node[below]{$\cat{C}\restrict{u \wedge v \leq v}$} (br);
  \end{pic}\]
  Let $\cat{A}$ be a (monoidal) category and $F \colon \cat{A} \to \cat{C}\restrict{u}$ and $G \colon \cat{A} \to \cat{C}\restrict{v}$ (lax monoidal) functors satisfying:
  \begin{equation}\label{eq:pb_outersq}
    \cat{C}\restrict{u \wedge v \leq u} \circ F 
    = 
    \cat{C}\restrict{u \wedge v \leq v} \circ G 
  \end{equation}
  We will show that there is a unique functor
  $H \colon \cat{A} \to \cat{C}\restrict{u \vee v}$ satisfying:
  \begin{equation}\label{eq:pb_tri}
	\cat{C}\restrict{u \leq u \vee v} \circ H = F
  \qquad\qquad
	\cat{C}\restrict{v \leq u \vee v} \circ H = G 
  \end{equation}
  by first showing that~\eqref{eq:pb_tri} forces a unique choice for how $H$ must act on objects and morphisms, and then verifying that this indeed defines a (lax monoidal) functor.

  For an object $A$ of $\cat{A}$, condition~\eqref{eq:pb_outersq} implies that $F(A) = G(A)$ since the restriction functors act as the identity on objects, and~\eqref{eq:pb_tri} forces $H(A)$ to be the same object.
  For a morphism $m \colon A \to B$ in $\cat{A}$, equation~\eqref{eq:pb_outersq} says that the maps $F(m) \colon F (A) \otimes U \to F (B)$ and $G (m) \colon G (A) \otimes V \to G (B)$ in $\cat{C}$ satisfy:
	\[\begin{pic}
      \node[morphism,width=8mm] (f) at (0,0) {$F(m)$};
		  \draw (f.north) to ++(0,+.3) node[above]{$H(B)$};
		  \draw ([xshift=-1mm]f.south west) to ++(0,-.5) node[below]{$H(A)$};
		  \draw ([xshift=1mm]f.south east) to ++(0,-.5) node[below]{$U$};
      \node[dot] (d) at (1,0) {};
		  \draw (d) to ++(0,-.7) node[below]{$V$};
    \end{pic}
	  =
	  \begin{pic}
		  \node[morphism,width=12mm] (f) at (0,0) {$G(m)$};
		  \draw (f.north) to ++(0,+.3) node[above]{$H(B)$};
		  \draw ([xshift=-2mm]f.south west) to ++(0,-.5) node[below]{$H(A)$};
		  \draw ([xshift=2mm]f.south east) to ++(0,-.5) node[below]{$V$};
      \node[dot] (d) at (0,-.4) {};
		  \draw (d) to ++(0,-.3) node[below]{$U$};
    \end{pic}
	\]
  whereas \eqref{eq:pb_tri} says that $H(m) \colon H(A) \otimes (U \vee V) \to H(B)$ must satisfy:
	\begin{equation}\label{eq:pb_tri_morphisms}
	\begin{pic}
		\node[morphism,width=10mm] (f) at (0,0) {$H(m)$};
		\draw (f.north) to ++(0,+.3) node[above]{$H(B)$};
		\draw (f.south west) to ++(0,-.8) node[below]{$H(A)$};
		\node[dot] (d) at ([yshift=-5mm]f.south east) {};
		\draw (f.south east) to node[right=-1mm]{$U \vee V$} (d);
		\draw (d) to ++(0,-.3) node[below]{$V$};
  \end{pic}
	=
	\begin{pic}
		\node[morphism,width=10mm] (f) at (0,0) {$F(m)$};
		\draw (f.north) to ++(0,+.3) node[above]{$H(B)$};
		\draw (f.south west) to ++(0,-.8) node[below]{$H(A)$};
		\draw (f.south east) to ++(0,-.8) node[below]{$U$};
  \end{pic}
  \qquad\qquad
  \begin{pic}
		\node[morphism,width=10mm] (f) at (0,0) {$H(m)$};
		\draw (f.north) to ++(0,+.3) node[above]{$H(B)$};
		\draw (f.south west) to ++(0,-.8) node[below]{$H(A)$};
		\node[dot] (d) at ([yshift=-5mm]f.south east) {};
		\draw (f.south east) to node[right=-1mm]{$U \vee V$} (d);
		\draw (d) to ++(0,-.3) node[below]{$V$};
  \end{pic}
  =
  \begin{pic}
		\node[morphism,width=10mm] (f) at (0,0) {$G(m)$};
		\draw (f.north) to ++(0,+.3) node[above]{$H(B)$};
		\draw (f.south west) to ++(0,-.8) node[below]{$H(A)$};
		\draw (f.south east) to ++(0,-.8) node[below]{$V$};
  \end{pic}
  \end{equation}
  Both conditions are summarised by commutativity of the following diagram in $\cat{C}$:
  \[\begin{pic}[xscale=4,yscale=1.5]
  	\node (tl) at (0,1) {$H (A) \otimes U \otimes V$};
  	\node (tr) at (1,1) {$H (A) \otimes U$};
  	\node (bl) at (0,0) {$H (A) \otimes V$};
  	\node (br) at (1,0) {$H (A) \otimes (U \vee V)$};
  	\node (c)  at (1.75,-.6) {$H(B)$};
  	\draw[->] (tl) to node[above]{} (tr);
  	\draw[->] (tl) to node[left]{} (bl);
  	\draw[->] (tr) to node[right]{} (br);
  	\draw[->] (bl) to node[below]{} (br);
	\draw[->] (tr) to[out=-30,in=90] node[above=1mm]{$F(m)$} (c);
	\draw[->] (bl) to[out=-30,in=180] node[below]{$G(m)$} (c);
	\draw[->,dashed] (br) to node[above]{$H(m)$} (c);
  \end{pic}\]
  Universal finite joins of central idempotents make this square a pushout.
  Thus there is only one possible choice for the morphism $H(m)$.

  It remains to show that this indeed defines a (lax monoidal) functor.
  Let $m_1 \colon A \to B$ and $m_2 \colon B \to C$ in $A$. 
  We will show that $H(m_2) \circ_{u \vee v} H(m_1)$ satisfies the defining conditions of $H(m_2 \circ m_1)$. Observe that
  \begin{align*}
    \cat{C}\restrict{u \leq u \vee v}(H(m_2) \circ_{u \vee v} H(m_1))
    &= \cat{C}\restrict{u \leq u \vee v}(H(m_2)) \circ_{u} \cat{C}\restrict{u \leq u \vee v}(H(m_1)) \\
  	&= F(m_2) \circ_u F(m_1)\\
	  &= F(m_2 \circ m_1)
  \end{align*}
  by functoriality of restriction, \eqref{eq:pb_tri_morphisms} instantiated for both $m_1$ and $m_2$, and functoriality of $F$. The condition holds analogously for $G$.
  Since $H(m_2 \circ m_1)$ is the unique map in $\cat{C}$ satisfying these conditions, $H(m_2) \circ_{u \vee v} H(m_1) = H(m_2 \circ m_1)$.
  
  Identities are preserved similarly: if $A$ is an object of $\cat{A}$, then
  \[
    \cat{C}\restrict{u \leq u \vee v}(\id[H(A)]) 
    = \id[\cat{C}\restrict{u \leq u \vee v}(H(A))] 
    = \id[F(A)] 
    = F(\id[A])
  \] 
  and analogously $\cat{C}\restrict{v \leq u \vee v}(\id[H(A)]) = G(\id[A])$, so $\id[H(A)] = H(\id[A])$.

  Finally we show that $H$ is lax monoidal.
  For objects $A$ and $B$ of $\cat{A}$, write $\theta^F_{A,B}$ for the structure morphism
  $F(A) \otimes F(B) \to F(A \otimes_\cat{A} B)$ in $\cat{C}\restrict{u}$ witnessing that $F$ is lax monoidal. This is a morphism $\theta^F_{A,B} \colon F(A) \otimes F(B) \otimes U \to  F(A \otimes B)$ in $\cat{C}$.
  Similarly, write $\theta^G_{A,B} \colon G(A) \otimes G(B) \otimes V \to  G(A \otimes B)$ in $\cat{C}$.
  It follows from~\eqref{eq:pb_outersq} that:
	\[\begin{pic}
 		  \node[morphism,width=20mm] (f) at (0,0) {$\theta^F_{A,B}$};
		  \draw (f.north) to ++(0,+.3) node[above]{$H(A \otimes B)$};
		  \draw ([xshift=-1mm]f.south west) to ++(0,-.5) node[below]{$HA$};
		  \draw ([xshift=0mm]f.south) to ++(0,-.5) node[below]{$HB$};
		  \draw ([xshift=1mm]f.south east) to ++(0,-.5) node[below]{$U$};
      \node[dot] (d) at (1.5,0) {};
		  \draw (d) to ++(0,-.7) node[below]{$V$};
    \end{pic}
	  =
	  \begin{pic}
		  \node[morphism,width=20mm] (f) at (0,0) {$\theta^G_{A,B}$};
		  \draw (f.north) to ++(0,+.3) node[above]{$H(A \otimes B)$};
		  \draw ([xshift=-2mm]f.south west) to ++(0,-.5) node[below]{$HA$};
		  \draw ([xshift=-2mm]f.south) to ++(0,-.5) node[below]{$HB$};
		  \draw ([xshift=2mm]f.south east) to ++(0,-.5) node[below]{$V$};
      \node[dot] (d) at (.4,-.4) {};
		  \draw (d) to ++(0,-.3) node[below]{$U$};
    \end{pic}
	\]
  That is, the outer square in the following diagram commutes:
  \[\begin{pic}[xscale=4,yscale=1.5]
  	\node (tl) at (0,1) {$H (A) \otimes H(B) \otimes U \otimes V$};
  	\node (tr) at (1,1) {$H (A) \otimes H(B) \otimes U$};
  	\node (bl) at (0,0) {$H (A) \otimes H(B) \otimes V$};
  	\node (br) at (1,0) {$H (A) \otimes H(B) \otimes (U \vee V)$};
  	\node (c)  at (1.75,-.8) {$H(A \otimes B)$};
  	\draw[->] (tl) to node[above]{} (tr);
  	\draw[->] (tl) to node[left]{} (bl);
  	\draw[->] (tr) to node[right]{} (br);
  	\draw[->] (bl) to node[below]{} (br);
	  \draw[->] (tr) to[out=-30,in=90] node[above=1mm]{$\theta^F_{A,B}$} (c);
  	\draw[->] (bl) to[out=-60,in=180] node[below]{$\theta^G_{A,B}$} (c);
	  \draw[->,dashed] (br) to node[right=2mm]{$\theta^H_{A,B}$} (c);
  \end{pic}\]
  This uniquely determines the dashed morphism $\theta^H_{A,B}$ satisfying:
	\[\begin{pic}
		  \node[morphism,width=16mm] (f) at (0,0) {$\theta^F_{A,B}$};
  		\draw (f.north) to ++(0,+.3) node[above]{$H(A \otimes B)$};
  		\draw ([xshift=-1mm]f.south west) to ++(0,-.8) node[below]{$HA$};
  		\draw (f.south) to ++(0,-.8) node[below]{$HB$};
  		\draw ([xshift=1mm]f.south east) to ++(0,-.8) node[below]{$U$};
    \end{pic}
    =
	  \begin{pic}
  		\node[morphism,width=16mm] (f) at (0,0) {$\theta^H_{A,B}$};
	  	\draw (f.north) to ++(0,+.3) node[above]{$H(A \otimes B)$};
  		\draw ([xshift=-1mm]f.south west) to ++(0,-.8) node[below]{$HA$};
  		\draw (f.south) to ++(0,-.8) node[below]{$HB$};
  		\node[dot] (d) at ([xshift=1mm,yshift=-5mm]f.south east) {};
  		\draw ([xshift=1mm]f.south east) to node[right=-1mm]{$U \vee V$} (d);
  		\draw (d) to ++(0,-.3) node[below]{$S$};
    \end{pic}
	  \qquad
	  \begin{pic}
	  	\node[morphism,width=16mm] (f) at (0,0) {$\theta^F_{A,B}$};
		  \draw (f.north) to ++(0,+.3) node[above]{$H(A \otimes B)$};
  		\draw ([xshift=-1mm]f.south west) to ++(0,-.8) node[below]{$HA$};
  		\draw (f.south) to ++(0,-.8) node[below]{$HB$};
  		\draw ([xshift=1mm]f.south east) to ++(0,-.8) node[below]{$V$};
    \end{pic}
    =
    \begin{pic}
		  \node[morphism,width=16mm] (f) at (0,0) {$\theta^H_{A,B}$};
  		\draw (f.north) to ++(0,+.3) node[above]{$H(A \otimes B)$};
  		\draw ([xshift=-1mm]f.south west) to ++(0,-.8) node[below]{$HA$};
  		\draw (f.south) to ++(0,-.8) node[below]{$HB$};
  		\node[dot] (d) at ([xshift=1mm,yshift=-5mm]f.south east) {};
  		\draw ([xshift=1mm]f.south east) to node[right=-1mm]{$U \vee V$} (d);
  		\draw (d) to ++(0,-.3) node[below]{$V$};
    \end{pic}
	\]
  The structure morphism $I \to H(I)$ is defined similarly. Naturality of $\theta^H$ and the laws of unitality and associativity follow from those for the structure morphisms of $F$ and $G$. 
\end{proof}

We have to pay careful attention to the nullary case of the sheaf condition.

\begin{lemma}\label{lem:C0}
  If $\cat{C}$ be a monoidal category with universal finite joins of central idempotents, then $\cat{C}\restrict{0}$ is monoidally equivalent to the terminal category $\cat{1}$. 
\end{lemma}
\begin{proof}
  It suffices to show that every object $A$ is isomorphic to $0$ in $\cat{C}\restrict{0}$. 
  Because $A \otimes 0 \simeq 0$ by universal finite joins of central idempotents, there is a unique morphism $f \colon A \otimes 0 \to 0$ in $\cat{C}$.
  Similarly, there is a unique morphism $g \colon 0 \otimes 0 \to A$.
  The composition $g \circ f$ in $\cat{C}\restrict{0}$ is a morphism $A \otimes 0 \to A$ in $\cat{C}$, and therefore unique, so has to equal the identity in $\cat{C}\restrict{0}$.
  Similarly, the composition $f \circ g$ in $\cat{C}\restrict{0}$ is a morphism $0 \otimes 0 \to 0$ in $\cat{C}$ and so equals the identity in $\cat{C}\restrict{0}$ by uniqueness.
\end{proof}

We can now define the desired sheaf of categories. For clarity we have avoided stacks and worked everything out concretely so far, but at this point we have to make a small change because of the previous lemma, as $\cat{C}\restrict{0}$ is not isomorphic to the terminal category but only equivalent to it. In logical terms, $\cat{C}\restrict{0}$ models the theory in which $0=1$.

\begin{proposition}\label{prop:sheaf}
  Let $\cat{C}$ be a monoidal category with universal finite joins of central idempotents. The functor $\mathcal{O}(\Spec(\ZI(\cat{C})))\op\to\cat{MonCat}$ given by $s \mapsto \cat{C}\restrict{s}$ is naturally monoidally equivalent to \changed{its sheafification, which is the following sheaf of monoidal categories:}
  \begin{align*}
    F \colon \mathcal{O}(\Spec(\ZI(\cat{C})))\op & \to \cat{MonCat} \\
    0 & \mapsto \cat{1} \\
    0 \neq u & \mapsto \cat{C}\restrict{u}
  \end{align*}
\end{proposition}
\begin{proof}
  Because $\cat{1}$ is terminal, $F$ is indeed functorial. There is a natural transformation from $u \mapsto \cat{C}\restrict{u}$ to $F$, whose component at every $u \neq 0$ is the identity functor, and whose component at $0$ is the unique functor to $\cat{1}$. This natural transformation is a (monoidal) equivalence by Lemma~\ref{lem:C0}.  Combining Lemma~\ref{lem:equaliser} with Proposition~\ref{prop:stalkfunctor} shows that $F$ defines a sheaf of (monoidal) categories.
  \changed{The fact that $F$ is the sheafification follows from the defining universal property, because any sheaf of categories has to assign $0 \mapsto \cat{1}$.}
\end{proof}

\section{Stalks}\label{sec:stalks}

In this section, we study the stalks of the sheaf of Proposition~\ref{prop:sheaf}. We start by generalising the idea of germs to the monoidal setting.

\begin{definition}\label{def:stalk}
  Let $\cat{C}$ be a monoidal category, and $x \subseteq \ZI(\cat{C})$ a filter. Define a category $\cat{C}\restrict{x}$ as follows:
  \begin{itemize}
	\item objects are those of $\cat{C}$;
	\item morphisms $A \to B$ are equivalence classes of pairs of $v \in x$ and $f \colon A \otimes V \to B$, where we identify $(v,f)$ and $(v',f')$ when $u \leq v \wedge v'$ for some $u \in x$:
	\[
	  \begin{pic}
		\node[morphism,width=5mm]	(f) at (0,0) {$f$};
		\node[dot] (d) at ([yshift=-5mm]f.south east) {};
		\draw (f.south east) to node[right=-1mm]{$V$} (d);
		\draw (d) to ++(0,-.3) node[below]{$U$};
		\draw ([xshift=-.5mm]f.south west) to ++(0,-.8) node[below]{$A$};
		\draw (f.north) to ++(0,.3) node[above]{$B$};
	  \end{pic}
	  \; = \;
	  \begin{pic}
		\node[morphism,width=5mm]	(f) at (0,0) {$f'$};
		\node[dot] (d) at ([yshift=-5mm]f.south east) {};
		\draw (f.south east) to node[right=-1mm]{$V'$} (d);
		\draw (d) to ++(0,-.3) node[below]{$U$};
		\draw ([xshift=-.5mm]f.south west) to ++(0,-.8) node[below]{$A$};
		\draw (f.north) to ++(0,.3) node[above]{$B$};
	  \end{pic}
	\]
	\item composition of $(u,f) \colon A \to B$ and $(v,g) \colon B \to C$ is:
	\[
	  \Big[
	  u \wedge v,\;
    \begin{pic}[font=\tiny]
  		\node[morphism,width=6mm] (f) at (0,0) {$f$};
  		\node[morphism,width=6mm,anchor=south west] (g) at ([yshift=2mm]f.north east) {$g$};
 		  \draw (f.north east) to (g.south west);
		  \draw (f.south west) to ++(0,-.3) node[right=-1mm]{$A$};
		  \draw (f.south east) to ++(0,-.3) node[right=-1mm]{$U$};
		  \draw (g.north) to ++(0,.3) node[right=-1mm]{$C$};
		  \draw (g.south east) to ++(0,-.9) node[right=-1mm]{$V$};
    \end{pic}\;
	  \Big]
	\]
	\item identity on $A$ is $[1, \text{\raisebox{-1mm}{$\begin{pic}[font=\tiny] \draw (0,0) node[right=-1mm]{$A$} to (0,.3); \end{pic}$}}]$.
  \end{itemize}
\end{definition}

Notice that $[u,f]=[u \wedge v, f \otimes v]$ for any central idempotent $v$.

\begin{lemma}
  If $\cat{C}$ is a monoidal category, and $x \subseteq \ZI(\cat{C})$ a filter, then $\cat{C}\restrict{x}$ is monoidal.
\end{lemma}
\begin{proof}
  The tensor product of objects is as in $\cat{C}$, the tensor product of morphisms $[u,f] \colon A \to B$ and $[v,g] \colon C \to D$ is $[u,f] \otimes [v,g] = [u \wedge v, (f \otimes g) \circ (A \otimes \sigma_{U,V} \otimes D)] \colon A \otimes C \to B \otimes D$.
  The coherence isomorphisms are $[1,\alpha]$, $[1,\lambda]$, and $[1,\rho]$.
\end{proof}

\begin{lemma}\label{lem:stalkcolimit}
  If $\cat{C}$ is a monoidal category, and $x \subseteq \ZI(\cat{C})$ a filter, then
  \[
    \cat{C}\restrict{x} = \colim_{u \in x} \cat{C}\restrict{u}
  \]
  in $\cat{MonCat}$,
  where the colimit ranges over the diagram induced by the functors $\cat{C}\restrict{u \leq v} \colon \cat{C}\restrict{v} \to \cat{C}\restrict{u}$ from Lemma~\ref{lem:restrictionfunctor}.
\end{lemma}
\begin{proof}
  The functors $F_u \colon \cat{C}\restrict{u} \to \cat{C}\restrict{x}$	 that send $f \colon A \otimes U \to B$ to $[u,f]$ form a cocone.
  If $G_u \colon \cat{C}\restrict{u} \to \cat{D}$ is another cocone, then there is a unique mediating functor $M \colon \cat{C}\restrict{x} \to \cat{D}$ given by $M[u,f] = G_u(f)$.
  If $\cat{C}$ and $G_u$ are monoidal, then $M$ is lax monoidal.
\end{proof}



The next lemma characterises the central idempotents in $\cat{C}\restrict{x}$.
The equivalence relation of having the same germ specialises to a semilattice congruence of central idempotents as follows for a  filter $x$ of central idempotents:
\[
  v \sim_x w 
  \iff
  \exists u \in x \colon u \wedge v = u \wedge w
\]

\begin{lemma}\label{lem:zi:stalks}
  If $\cat{C}$ is a braided monoidal category and $x \subseteq \ZI(\cat{C})$ a filter, then there is an isomorphism $\ZI(\cat{C}\restrict{x}) \simeq \ZI(\cat{C}) \slash \smash{\sim}_x$ of semilattices.
\end{lemma}
\begin{proof}
  Let $[v,q] \colon Q \to I$ be a morphism in $\cat{C}\restrict{x}$. 
  That is, choose a representing morphism $q \colon Q \otimes V \to I$ in $\cat{C}$. 
  Then by definition $[v,q]$ satisfies~\eqref{eq:central} in $\cat{C}\restrict{x}$ if and only if:
  \[
    \begin{pic}
      \node[morphism,width=5mm] (q) at (0,0) {$q$};
      \draw (q.south east) to node[right=-1mm]{$V$} ++(0,-.4) node[dot]{} to ++(0,-.3) node[below]{$U$};
      \draw (q.south west) to ++(0,-.7) node[below]{$Q$};
      \draw (-.6,-.9) node[below]{$Q$} to (-.6,.5) node[above]{$\smash{Q}$};
    \end{pic}    
    =
    \begin{pic}
      \node[morphism,width=5mm] (q) at (0,0) {$q$};
      \draw (q.south west) to ++(0,-.7) node[below]{$Q$};
      \draw (.2,-.9) node[below]{$Q$} to (.2,-.7) to[out=90,in=-90] (.6,-.1) to (.6,.5) node[above]{$\smash{Q}$};
      \draw[halo] (q.south east) to[out=-90,in=90] ++(.4,-.4) node[dot]{} to ++(0,-.3) node[below]{$U$};
    \end{pic}    
  \]
  in $\cat{C}$ for some $u \in x \cap \downset v$.
  If $m \colon U \to V$ satisfies $u = v \circ m$, then $q \circ (Q \otimes m) \colon Q \otimes U \to I$ is central in $\cat{C}$ if and only if:
  \[
    \begin{pic}
      \node[morphism,width=5mm] (q) at (0,0) {$q$};
      \draw (q.south east) to node[right=-1mm]{$V$} ++(0,-.4) node[dot]{} to ++(0,-.3) node[below]{$U$};
      \draw (q.south west) to ++(0,-.7) node[below]{$Q$};
      \draw (.6,-.9) node[below]{$Q$} to (.6,.5) node[above]{$\smash{Q}$};
      \draw (1,-.9) node[below]{$U$} to (1,.5) node[above]{$U$};
    \end{pic}
    =
    \begin{pic}
      \node[morphism,width=5mm] (q) at (0,0) {$q$};
      \draw (q.south east) to node[right=-1mm]{$V$} ++(0,-.4) node[dot]{} to ++(0,-.3) node[below]{$U$};
      \draw (q.south west) to ++(0,-.7) node[below]{$Q$};
      \draw (-.6,-.9) node[below]{$U$} to (-.6,.5) node[above]{$U$};
      \draw (-1,-.9) node[below]{$Q$} to (-1,.5) node[above]{$\smash{Q}$};
    \end{pic}
  \]
  But these two equations are equivalent by Lemma~\ref{lem:braiding}.

  Similarly, $[v,q]$ is idempotent in $\cat{C}\restrict{x}$ if and only if there exist $u \in x \cap \downset v$, and a morphism $p \colon Q \otimes U \to Q \otimes Q$ in $\cat{C}$ such that:
  \[
    \begin{pic}
      \node[morphism,width=5mm] (p) at (0,0) {$p$};
      \node[morphism,width=5mm, anchor=south west] (q) at ([yshift=3mm]p.north east) {$q$};
      \node[dot] (d) at (.4,-.5) {};
      \draw (p.north east) to (q.south west);
      \draw (p.north west) to ++(0,1) node[above]{$Q$};
      \draw (p.south west) to ++(0,-.6) node[below]{$Q$};
      \draw (q.south east) to ++(0,-.5) node[dot]{} to[out=-90,in=0,looseness=.7] (d);
      \draw (p.south east) to[out=-90,in=180,looseness=.7] (d);
      \draw (d) to ++(0,-.3) node[below]{$U$};
    \end{pic}
    =
    \begin{pic}
      \draw (0,0) node[below]{$Q$} to ++(0,2) node[above]{$Q$};
      \draw (.4,0) node[below]{$U$} to ++(0,.8) node[dot]{};
    \end{pic}
    \qquad\qquad
    \begin{pic}
      \node[morphism,width=10mm] (p) at (0,0) {$p$};
      \node[morphism,width=4mm] (q) at (0,-.6) {$q$};
      \node[dot] (d) at (.4,-1) {};
      \draw ([xshift=-1mm]p.south west) to ++(0,-1.1) node[below]{$Q$};
      \draw (q.south east) to[out=-90,in=180] (d);
      \draw ([xshift=2mm]p.south east) to ++(0,-.4) node[dot]{} to[out=-90,in=0,looseness=.7] (d);
      \draw (d) to ++(0,-.3) node[below]{$U$};
      \draw ([xshift=1mm]q.south west) to ++(0,-.5) node[below]{$Q$};
      \draw ([xshift=-1mm]p.north west) to ++(0,.5) node[above]{$Q$};
      \draw ([xshift=2mm]p.north east) to ++(0,.5) node[above]{$Q$};
    \end{pic}
    =
    \begin{pic}
      \draw (-.4,0) node[below]{$Q$} to ++(0,2) node[above]{$Q$};
      \draw (0,0) node[below]{$Q$} to ++(0,2) node[above]{$Q$};
      \draw (.4,0) node[below]{$U$} to ++(0,.8) node[dot]{};
    \end{pic}
  \]
  If $u = v \circ m$, then $q \circ (Q \otimes m) \colon Q \otimes U \to I$ is idempotent in $\cat{C}$ if and only if there exists a morphism $f \colon Q \otimes U \to Q \otimes U \otimes Q \otimes U$ in $\cat{C}$ satisfying:
  \[
    \begin{pic}
      \node[morphism,width=10mm] (f) at (0,0) {$f$};
      \node[morphism,width=4mm] (q) at (.3,-.6) {$q$};
      \draw (q.south west) to ++(0,-.7) node[below]{$Q$};
      \draw (q.south east) to ++(0,-.3) node[dot]{} to ++(0,-.7) node[below]{$U$};
      \draw ([xshift=-1mm]f.south west) to ++(0,-1.3) node[below]{$Q$};
      \draw ([xshift=2mm]f.south west) to ++(0,-1.3) node[below]{$U$};
      \draw ([xshift=-1mm]f.north west) to ++(0,.3) node[above]{$\smash{Q}$};
      \draw ([xshift=2.5mm]f.north west) to ++(0,.3) node[above]{$U$};
      \draw ([xshift=-2.5mm]f.north east) to ++(0,.3) node[above]{$\smash{Q}$};
      \draw ([xshift=1mm]f.north east) to ++(0,.3) node[above]{$U$};
    \end{pic}
    \;=\!
    \begin{pic}
      \draw (-.4,0) node[below]{$Q$} to ++(0,2) node[above]{$\smash{Q}$};
      \draw (0,0) node[below]{$U$} to ++(0,2) node[above]{$U$};
      \draw (.4,0) node[below]{$Q$} to ++(0,2) node[above]{$\smash{Q}$};
      \draw (.8,0) node[below]{$U$} to ++(0,2) node[above]{$U$};
    \end{pic}
    \qquad\qquad\qquad
    \begin{tikzpicture}[baseline=3mm]
      \node[morphism,width=10mm] (f) at (0,0) {$f$};
      \node[morphism,width=4mm] (q) at (.3,1) {$q$};
      \draw (q.south east) to ++(0,-.3) node[dot]{} to ++(0,-.6);
      \draw (q.south west) to ++(0,-.6);
      \draw ([xshift=-1mm]f.north west) to ++(0,1.3) node[above]{$\smash{Q}$};
      \draw ([xshift=2.5mm]f.north west) to ++(0,1.3) node[above]{$\smash{U}$};
      \draw (f.south west) to ++(0,-.3) node[below]{$Q$};
      \draw (f.south east) to ++(0,-.3) node[below]{$U$};
    \end{tikzpicture}
    \;=\!
    \begin{pic}
      \draw (-.4,0) node[below]{$Q$} to ++(0,2) node[above]{$\smash{Q}$};
      \draw (0,0) node[below]{$U$} to ++(0,2) node[above]{$U$};
    \end{pic}
  \]
  These two properties are equivalent, by choosing:
  \[
    \begin{pic}
      \node[morphism,width=10mm] (f) {$f$};
      \draw (f.south west) to ++(0,-.6) node[below]{$Q$};
      \draw (f.south east) to ++(0,-.6) node[below]{$U$};
      \draw ([xshift=-1mm]f.north west) to ++(0,.8) node[above]{$\smash{Q}$};
      \draw ([xshift=2.5mm]f.north west) to ++(0,.8) node[above]{$U$};
      \draw ([xshift=-2.5mm]f.north east) to ++(0,.8) node[above]{$\smash{Q}$};
      \draw ([xshift=1mm]f.north east) to ++(0,.8) node[above]{$U$};
    \end{pic}
    =
    \begin{pic}
      \node[morphism,width=5mm] (p) at (0,0) {$p$};
      \node[dot] (l) at (.45,-.5) {};
      \node[dot] (r) at (.8,-.8) {};
      \draw (p.north west) to ++(0,.5) node[above]{$\smash{Q}$};
      \draw (p.south west) to ++(0,-.9) node[below]{$Q$};
      \draw (p.south east) to[out=-90,in=180] (l);
      \draw (l) to[out=-90,in=180] (r);
      \draw (r) to ++(0,-.3) node[below]{$U$};
      \draw (p.north east) to[out=90,in=-90] ++(.5,.5) node[above]{$\smash{Q}$};
      \draw (r) to[out=0,in=-90] ++(.3,.3) to ++(0,1.2) node[above]{$U$};
      \draw[halo] (l) to[out=0,in=-90] ++(.25,.3) to ++(0,.3) to[out=90,in=-90] ++(-.5,.6) node[above]{$U$};
    \end{pic}
    \qquad\qquad
    \begin{pic}
      \node[morphism,width=5mm] (p) {$p$};
      \draw (p.south west) to ++(0,-.7) node[below]{$Q$};
      \draw (p.south east) to ++(0,-.7) node[below]{$U$};
      \draw (p.north west) to ++(0,.7) node[above]{$\smash{Q}$};
      \draw (p.north east) to ++(0,.7) node[above]{$\smash{Q}$};
    \end{pic}
    =
    \begin{pic}
      \node[morphism,width=10mm] (f) {$f$};
      \draw (f.south west) to ++(0,-.7) node[below]{$Q$};
      \draw (f.south east) to ++(0,-.7) node[below]{$U$};
      \draw ([xshift=-1mm]f.north west) to ++(0,.7) node[above]{$\smash{Q}$};
      \draw ([xshift=2.5mm]f.north west) to ++(0,.5) node[dot]{};
      \draw ([xshift=-2.5mm]f.north east) to ++(0,.7) node[above]{$\smash{Q}$};
      \draw ([xshift=1mm]f.north east) to ++(0,.5) node[dot]{};
    \end{pic}
  \]
  Just like in Lemma~\ref{lem:zi:Cu}, a half-braiding on $Q$ in $\cat{C}\restrict{x}$ induces a half-braiding on $Q \otimes U$ in $\cat{C}$.
  Thus every central idempotent in $\cat{C}\restrict{x}$ is induced by a central idempotent in $\cat{C}$. But by definition two central idempotents $u$ and $v$ in $\cat{C}$ induce the same central idempotent in $\cat{C}\restrict{x}$ exactly when $u \sim_x v$.
\end{proof}

\begin{lemma}\label{lem:Cxstiff}
  Let $\cat{C}$ be a monoidal monoidal category and $x \subseteq \ZI(\cat{C})$ a filter.
  \begin{itemize}
    \item If $\cat{C}$ is stiff, then so is $\cat{C}\restrict{x}$.
    \item If $\cat{C}$ has universal finite joins of central idempotents, then so does $\cat{C}\restrict{x}$.
    \item If $\cat{C}$ has universal joins of central idempotents, then so does $\cat{C}\restrict{x}$.
  \end{itemize}
  The functor $\cat{C} \to \cat{C}\restrict{x}$ preserves joins of central idempotents.
\end{lemma}
\begin{proof}
  We start with stiffness.
  Let $v,w$ be central idempotents in $\cat{C}$. 
  It is clear that the inner square below commutes in $\cat{C}\restrict{x}$.
  \[\begin{pic}[xscale=2.5,yscale=2]
    \node (br) at (2,0) {$A$};
    \node (tr) at (2,1) {$A \otimes W$};
    \node (bl) at (0,0) {$A \otimes V$};
    \node (tl) at (0,1) {$A \otimes V \otimes W$};
    \node (t) at (-.75,1.6) {$B$}; 
    \draw[>->] (tl) to node[below]{$[1,A \otimes v \otimes W]$} (tr);
    \draw[>->] (tr) to node[right]{$[1,A \otimes w]$} (br);
    \draw[>->] (tl) to node[right]{$[1,A \otimes V \otimes w]$} (bl);
    \draw[>->] (bl) to node[below]{$[1,A \otimes v]$} (br);
    \draw[->] (t) to[out=0,in=120,looseness=.5] node[above]{$[u',g]$} (tr);
    \draw[->] (t) to[out=-90,in=150] node[left]{$[u,f]$} (bl);
    \draw[->,dashed] (t) to node[right=4mm]{$[u \wedge u' \wedge q,m]$} (tl);
  \end{pic}\]
  Suppose the outer square commutes too.
  This means $(A \otimes v) \circ (f \otimes u' \otimes q) = (A \otimes w) \circ g \circ (B \otimes u \otimes U' \otimes q)$ for some central idempotent $q$ in $\cat{C}$.
  Because $\cat{C}$ is stiff, there is a morphism $m \colon B \otimes U \otimes U' \otimes Q \to A \otimes V \otimes W$ satisfying: 
  \[
    \begin{pic}
      \node[morphism,width=6mm] (f) at (0,0) {$f$};
      \draw (f.north west) to ++(0,.3) node[above]{$A$};
      \draw (f.north east) to ++(0,.3) node[above]{$V$};
      \draw (f.south west) to ++(0,-.4) node[below]{$B$};
      \draw (f.south east) to ++(0,-.4) node[below]{$U$};
      \draw (.75,0) node[dot]{} to ++(0,-.6) node[below]{$U\smash{'}$};
      \draw (1.2,0) node[dot]{} to ++(0,-.6) node[below]{$Q$};
    \end{pic}
    =
    \begin{pic}
      \node[morphism,width=14mm] (f) at (0,0) {$m$};
      \draw (f.north west) to ++(0,.3) node[above]{$A$};
      \draw (f.north) to ++(0,.3) node[above]{$V$};
      \draw (f.north east) to ++(0,.3) node[dot]{} node[above]{$W$};
      \draw (f.south west) to ++(0,-.4) node[below]{$B$};
      \draw ([xshift=-2mm]f.south) to ++(0,-.4) node[below]{$U$};
      \draw ([xshift=2mm]f.south) to ++(0,-.4) node[below]{$U\smash{'}$};
      \draw (f.south east) to ++(0,-.4) node[below]{$Q$};
    \end{pic}
    \qquad\qquad
    \begin{pic}
    \node[morphism,width=10mm] (f) at (0,0) {$g$};
    \draw (f.north west) to ++(0,.3) node[above]{$A$};
    \draw (f.north east) to ++(0,.3) node[above]{$V$};
    \draw (f.south west) to ++(0,-.4) node[below]{$B$};
    \draw (f.south east) to ++(0,-.4) node[below]{$U\smash{'}$};
    \draw (0,-.4) node[dot]{} to ++(0,-.2) node[below]{$U$};
    \draw (.9,0) node[dot]{} to ++(0,-.6) node[below]{$Q$};
    \end{pic}
    =
    \begin{pic}
    \node[morphism,width=14mm] (f) at (0,0) {$m$};
    \draw (f.north west) to ++(0,.3) node[above]{$A$};
    \draw (f.north) to ++(0,.3) node[dot]{} node[above]{$V$};
    \draw (f.north east) to ++(0,.3) node[above]{$W$};
    \draw (f.south west) to ++(0,-.4) node[below]{$B$};
    \draw ([xshift=-2mm]f.south) to ++(0,-.4) node[below]{$U$};
    \draw ([xshift=2mm]f.south) to ++(0,-.4) node[below]{$U\smash{'}$};
    \draw (f.south east) to ++(0,-.4) node[below]{$Q$};
    \end{pic}
  \]
  It follows that the dashed morphism $[u \wedge u' \wedge q,m]$ makes the two triangles commute. The uniqueness of $m$ in $\cat{C}$ also guarantees that the dashed morphism is the unique such morphism in $\cat{C}\restrict{x}$.

  Next we turn to joins of central idempotents.
  The initial object $0$ satisfying $A \otimes 0 \simeq 0$ in $\cat{C}$ is still initial and still satisfies $A \otimes 0 \simeq 0$ in $\cat{C}\restrict{x}$. The proof that~\eqref{eq:pullback-pushout} is a pullback in $\cat{C}\restrict{x}$ is virtually the same as in the stiff case. We focus on the pushout property. 
  Suppose that $[u,f] \circ [1,A \otimes V \otimes w] = [u',g] \circ [1,A \otimes v \otimes W]$ in $\cat{C}\restrict{x}$. This means:
  \[
    f \circ (q \otimes U \otimes u' \otimes A \otimes V \otimes w) 
    = 
    g \circ (q \otimes u \otimes U' \otimes A \otimes v \otimes W)
  \]
  for some central idempotent $q$ in $\cat{C}$.
  If $\cat{C}$ has finite joins of central idempotents, there is a morphism $m \colon ((Q \otimes U \otimes V) \vee (Q \otimes U' \otimes W)) \otimes A \to B$ satisfying:
  \[
    \begin{pic}
      \node[morphism,width=9mm] (m) at (0,0) {$f$};
      \draw (m.north) to ++(0,.3) node[above]{$B$};
      \draw ([xshift=1mm]m.south east) to ++(0,-1) node[below,xshift=2mm]{$Q \otimes U \otimes V$};
      \draw ([xshift=-1mm]m.south west) to ++(0,-1) node[below]{$A$};
    \end{pic}    
    \hspace*{-3mm}
    =
    \begin{pic}
  	  \node[morphism,width=11mm] (m) at (0,0) {$m$};
  	  \node[dot] (d) at ([xshift=1mm,yshift=-7mm]m.south east) {};
  	  \draw (m.north) to ++(0,.3) node[above]{$B$};
  	  \draw ([xshift=1mm]m.south east) to node[right=-1mm,font=\tiny,text width=17mm]{$(Q\!\otimes\!U\!\otimes\!V) \vee (Q\!\otimes\!U'\!\otimes\!W)$} (d);
  	  \draw (d) to ++(0,-.3) node[below,xshift=2mm]{$Q \otimes U \otimes V$};
  	  \draw ([xshift=-1mm]m.south west) to ++(0,-1) node[below]{$A$};
    \end{pic}
    \;\;
    \begin{pic}
      \node[morphism,width=9mm] (m) at (0,0) {$g$};
      \draw (m.north) to ++(0,.3) node[above]{$B$};
      \draw ([xshift=1mm]m.south east) to ++(0,-1) node[below,xshift=3mm]{$Q \otimes \smash{U'} \otimes W$};
      \draw ([xshift=-1mm]m.south west) to ++(0,-1) node[below]{$A$};
    \end{pic}    
    \hspace*{-3mm}
    =
    \begin{pic}
      \node[morphism,width=11mm] (m) at (0,0) {$m$};
      \node[dot] (d) at ([xshift=1mm,yshift=-7mm]m.south east) {};
      \draw (m.north) to ++(0,.3) node[above]{$B$};
      \draw ([xshift=1mm]m.south east) to node[right=-1mm,font=\tiny,text width=17mm]{$(Q\!\otimes\!U\!\otimes\!V) \vee (Q\!\otimes\!U'\!\otimes\!W)$} (d);
      \draw (d) to ++(0,-.3) node[below,xshift=2mm]{$Q \otimes \smash{U'} \otimes W$};
      \draw ([xshift=-1mm]m.south west) to ++(0,-1) node[below]{$A$};
    \end{pic}
  \]
  Observe that $(q \wedge u \wedge v) \vee (q \wedge u' \wedge w) = p \wedge (v \vee w)$ for $p=q \wedge (u \vee u') \wedge (u \vee w) \wedge (v \vee u')$.
  Now $[p,m]$ is the unique mediating morphism in $\cat{C}\restrict{x}$ satisfying $[p,m] \circ [1,A \otimes (v \leq v \vee w)] = [u,f]$ and $[p,m] \circ [1,A \otimes (w \leq v \vee w)] = [u',g]$.
  The above holds equally well for wide pushouts.
\end{proof}

\section{\changed{$\vee$-Locality}}\label{sec:local}

This section proves that the stalks are particularly easy, in the sense that the central idempotents behave well because there are few of them. 

\begin{definition}\label{def:sublocal}
\changed{Call a partially ordered set \emph{$\vee$-local} if it has at least two elements, and $u \vee w=1$ can only happen when $u=1$ or $w=1$.
  Call a monoidal category $\cat{C}$ \emph{$\vee$-local} when $\ZI(\cat{C})$ is $\vee$-local.}
\end{definition}

A partially ordered set is \changed{$\vee$-local} if and only if it has a unique maximal ideal.
\changed{We now connect this property to primality of filters.}



\begin{lemma}\label{lem:sublocal}
  If $\cat{C}$ is a monoidal category with universal finite joins of central idempotents, and $x \subseteq \ZI(\cat{C})$ is a prime filter, then $\cat{C}\restrict{x}$ is \changed{$\vee$-local}.
\end{lemma}
\begin{proof}
  Consider two elements of $\ZI(\cat{C}\restrict{x})$. 
  They are represented by $u,w \in \ZI(\cat{C})$ by Lemma~\ref{lem:zi:stalks}.
  Now, $u \vee w \sim_x 1$ if and only if there is $v$ in the prime filter $x$ and $(u \wedge v) \vee (v \wedge w) = v$.
  Since $(u \wedge v) \vee (v \wedge w)$ is in the prime filter $x$, either $u \wedge v \in x$ or $v \wedge w \in x$.
  But this implies $u \sim_x 1$ or $w \sim_x 1$.

  \changed{Finally, we prove that $0 \neq 1$ in $\ZI(\cat{C}\restrict{x})$.
  For a contradiction suppose that $0 \simeq I$ in $\cat{C}\restrict{x}$. Then $0 \simeq I$ in $\ZI(\cat{C}\restrict{u})$ for some $u \in \ZI(\cat{C})$ contained in $x$. 
  Because $A \simeq B$ in $\cat{C}\restrict{u}$ if and only if $A \otimes U \simeq B \otimes U$ in $\cat{C}$, it follows that $0 \simeq 0 \otimes U \simeq I \otimes U \simeq U$ in $\cat{C}$. But that means that $x=\{0\}$, contradicting the fact that $x$ is a proper filter.}
\end{proof}

\changed{
\begin{theorem}\label{thm:main}
  Any small monoidal category with universal $\vee$-joins of central idempotents is monoidally equivalent to the category of global sections of a sheaf of $\vee$-local monoidal categories.
\end{theorem}
\begin{proof}
  Let $\cat{C}$ be the monoidal category with universal finite joins of central idempotents.
  Take $X=\Spec(\ZI(\cat{C}))$ as the base space, as in Definition~\ref{def:primespectrum}. 
  Then $B_u \mapsto \cat{C}\restrict{u}$ is a presheaf of monoidal categories by Lemma~\ref{lem:restrictionfunctor}.
  It extends to a sheaf as in Proposition~\ref{prop:sheaf}.
  The stalks $\cat{C}\restrict{x}$ for $x \in X$ are \changed{$\vee$-local} by Lemma~\ref{lem:sublocal}, so this is a sheaf of \changed{$\vee$-}local categories.
  Finally, the global sections of this sheaf form $\cat{C}\restrict{1}$, which is monoidally equivalent to $\cat{C}$ itself.
\end{proof}

\begin{corollary}\label{cor:subdirect}
  Any small monoidal category with universal $\vee$-joins of central idempotents embeds monoidally into a product of $\vee$-local monoidal categories.
\end{corollary}
\begin{proof}
  Consider the tuple $\cat{C} \to \prod_{x \in \Spec(\ZI(\cat{C}))} \cat{C}\restrict{x}$ of the quotient functors $\cat{C} \to \cat{C}\restrict{x}$.
  It acts as the identity on objects.
  Faithfulness means that if $[f]_x = [g]_x$ in $\cat{C}\restrict{x}$ for all $x \in \Spec(\ZI(\cat{C}))$ then $f=g \colon A \to B$ in $\cat{C}$.
  Because a colimiting cocone is jointly monic, by Lemma~\ref{lem:basiscompact} it suffices to prove that $f=g$ as soon as $f \otimes u = g \otimes u$ in $\cat{C}\restrict{u}$ for all $u \in \ZI(\cat{C})$.
  The result thus follows from (a jointly monic version of) Lemma~\ref{lem:epimonoisolocal}.
\end{proof}
}

\section{\changed{$\bigvee$-locality}}\label{sec:inflocal}

\changed{
Theorem~\ref{thm:main} is not entirely satisfactory when applying it to categories $\cat{C}$ whose central idempotents already form a frame, because the representation still treats $\ZI(\cat{C})$ as a mere lattice. 
For example, if $X$ is a topological space, then Theorem~\ref{thm:main} represents $\cat{C}=\mathrm{Sh}(X)$ as a sheaf of categories over a different space than $X$; the same holds for several examples in Section~\ref{sec:examples} below.
To remedy this, this section extends the representation from finite to arbitrary joins.
However, in that case there is no longer an analogue of Lemma~\ref{lem:basiscompact} to guarantee that the spectrum of $\ZI(\cat{C})$ has enough points.
We will take the frame $\ZI(\cat{C})$ itself as a base, which always yields a well-defined sheaf of categories. But we only obtain the full force of the representation theorem when we assume that the frame $\ZI(\cat{C})$ is spatial.
}

A \emph{completely prime filter} of a complete lattice $L$ is a nonempty upward-closed and downward-directed \changed{proper} subset $P$ where $\bigvee s_i \in P$ implies $s_i \in P$ for some $i$; equivalently, $P$ is the inverse image of 1 under a complete lattice morphism $L \to \{0,1\}$.

\begin{definition}
  Let $L$ be a frame.	
  Its \emph{completely prime spectrum} is a topological space $X$, whose points are completely prime \changed{filters} $P \subseteq L$, and whose topology is generated by a basis consisting of the sets
  \[
    B_u = \{P \in X \mid u \in P \}
  \]
  where $u$ ranges over $L$.
  \changed{A frame is \emph{spatial} when it is isomorphic to the frame of opens of its completely prime spectrum.}
\end{definition}

\begin{definition}\label{def:local}
\changed{Call a partially ordered set \emph{$\bigvee$-local} if it has at least two elements, and $\bigvee u_i=1$ can only happen when there is an $i$ with $u_i=1$.
  Call a monoidal category $\cat{C}$ \emph{$\bigvee$-local} when $\ZI(\cat{C})$ is $\bigvee$-local.}
\end{definition}

\begin{lemma}\label{lem:local}
  If $\cat{C}$ is a monoidal category with universal joins of central idempotents, and $x \subseteq \ZI(\cat{C})$	is a completely prime filter, then $\cat{C}\restrict{x}$ is \changed{$\bigvee$-local}.
\end{lemma}
\begin{proof}
  Completely analogous to Lemma~\ref{lem:sublocal}.
\end{proof}

By a \emph{sheaf of \changed{($\vee$- or $\bigvee$-)local} categories}, we mean a sheaf $\mathcal{O}(X)\op \to \cat{MonCat}$ on a topological space $X$ whose stalks are all \changed{($\vee$- or $\bigvee$)local}. Pulling everything together, we now arrive at our main result. A \emph{global section} of a sheaf $F \colon \mathcal{O}(X)\op \to \cat{MonCat}$ is an object of $F(X)$. For $\cat{Set}$-valued sheaves, this corresponds to the more usual definition of a global section being a natural transformation $1 \Rightarrow F$. Enriching this as usual, we will call $F(X)$ the category of global sections of $F$.

\changed{
\begin{theorem}\label{thm:main2}
  Let $\cat{C}$ be a small monoidal category with universal $\bigvee$-joins of central idempotents.
  Then $u \mapsto \cat{C}\restrict{u}$ defines a sheaf $F \colon \ZI(\cat{C})\op \to \cat{MonCat}$ of $\bigvee$-local monoidal categories.
  If $\ZI(\cat{C})$ is spatial, then $\cat{C}$ is monoidally equivalent to the category of global sections of $F$.
\end{theorem}
\begin{proof}
  Lemma~\ref{lem:restrictionfunctor} still shows that $u \mapsto \cat{C}\restrict{u}$ is a presheaf.
  For the sheaf condition it now no longer suffices to verify binary equaliser of Lemma~\ref{lem:equaliser}, and we have to consider a wide equaliser instead. But the proof of Proposition~\ref{prop:stalkfunctor} extends easily to this case. The exceptional nullary case is taken care of by Lemma~\ref{lem:C0} as in Proposition~\ref{prop:sheaf} as before. The stalks are now \changed{$\bigvee$-}local by Lemma~\ref{lem:local}.
  Because $\ZI(\cat{C})$ is assumed spatial, it is isomorphic to the frame of opens of its completely prime spectrum $X$. 
  This makes $B_u \mapsto \cat{C}\restrict{u}$ a well-defined sheaf of $\bigvee$-local monoidal categories on $X$.
  Finally, the global sections of this sheaf form $\cat{C}\restrict{1}$, which is monoidally equivalent to $\cat{C}$ itself.
\end{proof}

The dependence on spatiality is related to projectivity of the tensor unit in the topos case, and removing it is left to future work.

\begin{corollary}
  Any small monoidal category with universal $\bigvee$-joins of central idempotents whose frame of central idempotents is spatial embeds monoidally into a product of $\bigvee$-local monoidal categories.
\end{corollary}
\begin{proof}
  Completely analogous to Corollary~\ref{cor:subdirect}, where spatiality is needed to replace Lemma~\ref{lem:basiscompact}.
\end{proof}
}

\section{Preservation}\label{sec:preservation}

Theorem~\ref{thm:main} showed that any small monoidal category with universal joins of central idempotents is a category of global sections of a sheaf of \changed{($\vee-$ or $\bigvee-$)}local monoidal categories. In this section we extend that main result by showing that it preserves various properties: if the original category has a certain property, then so do the stalks. We consider two kinds of properties: properties associated with monoidal categories and linear logic, such as being compact, and having a trace; and properties associated with toposes and intuitionistic logic, such as having a Boolean algebra of central idempotents, and having limits. We also investigate the property of being closed, which resides in both camps.
So for example, we show that any small monoidal category with (finite) universal joins of central idempotents that is closed, is equivalent to the category of global sections of a sheaf of \changed{($\vee-$ or $\bigvee$-)}local closed categories.

We start with compactness. Recall that a symmetric monoidal category is compact when every object has a dual~\cite[Chapter~3]{heunenvicary:cqm}.

\begin{corollary}
  If $\cat{C}$ is a compact category, $u$ is a central idempotent, and $x \subseteq \ZI(\cat{C})$ is a prime filter, then $\cat{C}\restrict{u}$ and $\cat{C}\restrict{x}$ are compact categories, too.
\end{corollary}
\begin{proof}
  Strong monoidal functors preserve dual objects~\cite[Theorem~3.14]{heunenvicary:cqm}, so this follows directly from Lemma~\ref{lem:restrictionfunctor} and Proposition~\ref{prop:stalkfunctor}.
\end{proof}

Notice, however, that a central idempotent $U$ has itself as a dual only if it is (represented by) a split monomorphism.

\begin{lemma}
  A central idempotent $u$ is split monic if and only if $U \dashv U$ with counit $\rho \circ (u \otimes u)$.
\end{lemma}
\begin{proof}
  If $e \circ u = U$, set $\eta = (U \otimes u)^{-1} \circ (e \otimes I) \circ \lambda^{-1}$ and $\varepsilon = \rho \circ (u \otimes u)$, so:
  \[
    \begin{pic}
	  \node[morphism,width=6mm] (eta) at (0,-.5) {$\eta$};
	  \node[morphism,width=6mm,anchor=south west] (eps) at ([yshift=3mm]eta.north east) {$\varepsilon$};
	  \draw (eta.north east) to (eps.south west);
	  \draw (eta.north west) to ++(0,.8) node[above]{$U$};
	  \draw (eps.south east) to ++(0,-.8) node[below]{$U$};
    \end{pic}
    =
    \begin{pic}
	  \node[morphism] (e) at (0,0) {$e$};
	  \node[dot] (d) at (0,.5) {};
	  \draw (e.north) to (d);
	  \draw (d) to[out=180,in=-90] ++(-.3,.3) to ++(0,.2) node[above]{$U$};
	  \draw (d) to[out=0,in=-90] ++(.3,.3) node[dot]{};
	  \draw (.65,-.3) node[below]{$U$} to ++(0,1.1) node[dot]{};
    \end{pic}
    = \;
    \begin{pic}
	  \draw (0,1) node[above]{$U$} to (0,.5) node[morphism]{$e$};
	  \draw (0,-.3) node[below]{$U$} to ++(0,.35) node[dot]{};
    \end{pic}
  \]
  Conversely, if $U \dashv U$ with unit $\eta$ and counit $u \otimes u$, then $e \circ u = U$ for $e = (U \otimes u) \circ \eta$.
\end{proof}

Another linear property that is preserved is having a trace~\cite{joyalstreetverity:traced}.

\begin{lemma}
  If a monoidal category $\cat{C}$ with universal finite joins of central idempotents is traced, $u$ is a central idempotent, and $x \subseteq \ZI(\cat{C})$ a prime filter, then $\cat{C}\restrict{u}$ and $\cat{C}\restrict{x}$ are traced too, and the functors $\cat{C}\restrict{u \leq v} \colon \cat{C}\restrict{v} \to \cat{C}\restrict{u}$ and $\cat{C} \to \cat{C}\restrict{x}$ preserve the trace.
\end{lemma}
\begin{proof}
  If $\cat{C}$ is braided, then so are $\cat{C}\restrict{u}$ and $\cat{C}\restrict{x}$.
  The trace of $f \in \cat{C}\restrict{u}(A \otimes Z, B \otimes Z) = \cat{C}(A \otimes Z \otimes U, B \otimes Z)$ is defined to be the trace of $f \circ (A \otimes \sigma) \in \cat{C}(A \otimes U \otimes Z, B \otimes Z)$. It is easy to see that this satisfies the axioms for a trace.
  For example, with the usual graphical calculus~\cite{selinger:graphicallanguages} of the traced monoidal category $\cat{C}$, the superposing axiom becomes:
  \[
    g \otimes \Tr(f)
    =
    \begin{pic}
	  \node[morphism,width=6mm] (g) at (0,0) {$g$};
	  \node[morphism,width=8mm] (f) at (1.2,0) {$f$};
	  \draw (g.north) to ++(0,.5) node[above]{$B$};
	  \draw (f.north west) to ++(0,.5) node[above]{$D$};
	  \draw (g.south west) to ++(0,-1) node[below]{$A$};
	  \draw (f.south west) to[out=-90,in=90] ++(-.7,-.5) to ++(0,-.5) node[below]{$C$};
	  \node[dot] (d) at (.9,-.7) {};
	  \draw[halo] (g.south east) to[out=-90,in=180] (d);
	  \draw (d) to ++(0,-.5) node[below]{$U$};
	  \draw (f.south) to[out=-90,in=90] ++(.5,-.5) to[out=-90,in=-90,looseness=2] ++(.7,0) to ++(0,.8) node[right=-1mm]{$Z$} to[out=90,in=90,looseness=2] (f.north east);
    \draw[halo] (d) to[out=0,in=-90] (f.south east);
	  \node[dot] (d) at (.9,-.7) {};
    \node[morphism,width=6mm] (g) at (0,0) {$g$};
    \node[morphism,width=8mm] (f) at (1.2,0) {$f$};
    \end{pic}
    =\;
    \begin{pic}
	  \node[morphism,width=6mm] (g) at (0,0) {$g$};
	  \node[morphism,width=8mm] (f) at (1.2,0) {$f$};
	  \draw (g.north) to ++(0,.5) node[above]{$B$};
	  \draw (f.north west) to ++(0,.5) node[above]{$D$};
	  \draw (g.south west) to ++(0,-1) node[below]{$A$};
	  \draw (f.south west) to[out=-90,in=90] ++(-.7,-.5) to ++(0,-.5) node[below]{$C$};
	  \node[dot] (d) at (1.2,-.7) {};
    \draw (f.south) to[out=-90,in=90] ++(-.3,-.5) to[out=-90,in=-90,looseness=.7] ++(1.2,0) to ++(0,.9) node[right=-1mm]{$Z$} to[out=90,in=90,looseness=2] (f.north east);
	  \draw[halo] (g.south east) to[out=-90,in=180] (d);
	  \draw[halo] (d) to ++(0,-.5) node[below]{$U$};
	  \draw[halo] (d) to[out=0,in=-90] (f.south east);
	  \node[dot] (d) at (1.2,-.7) {};
    \node[morphism,width=6mm] (g) at (0,0) {$g$};
    \node[morphism,width=8mm] (f) at (1.2,0) {$f$};
    \end{pic}
    =
    \Tr(g \otimes f)
  \]

  Stalks are entirely similar: define the trace of $[v,f] \in \cat{C}\restrict{x}(A \otimes Z, B \otimes Z)$ to be $[v,g]$ where $g$ is the trace of $f$ in $\cat{C}\restrict{v}$. To verify that this is well-defined, if $[v,f]=[v',f']$, say because $f \otimes u \otimes v' = f' \otimes u \otimes v$ for a central idempotent $u$ in $\cat{C}$, then:
  \[
    \begin{pic}
	  \node[morphism,width=6mm] (f) at (0,0) {$f$};
	  \node[dot] (d) at (.9,0) {};
	  \node[dot] (d2) at (1.2,0) {};
	  \draw (f.north west) to ++(0,.5) node[above]{$B$};
	  \draw (f.south west) to ++(0,-.5) node[below]{$A$};
	  \draw (f.south) to[out=-90,in=-90,looseness=1.5] ++(.7,0) to ++(0,.4) to[out=90,in=90,looseness=2] node[right=1mm]{$Z$} (f.north east);
	  \draw[halo] ([xshift=1mm]f.south east) to[out=-90,in=90] ++(-.25,-.5) node[below]{$V$};
	  \draw (d) to ++(0,-.7) node[below]{$V\smash{'}$};
	  \draw (d2) to ++(0,-.7) node[below]{$U$};
    \end{pic}
    \;=\;\;
    \begin{pic}
	  \node[morphism,width=10mm] (f) at (0,0) {$f$};
	  \node[dot] (d2) at (1.1,0) {};
	  \draw (f.north west) to ++(0,.4) node[above]{$B$};
	  \draw ([xshift=-1mm]f.south west) to ++(0,-.6) node[below]{$A$};  
	  \draw ([xshift=-2mm]f.south) to[out=-90,in=-90,looseness=1.5] ++(1,0) to ++(0,.4) to[out=90,in=90,looseness=2] node[right=2mm]{$Z$} (f.north east);
	  \draw[halo] ([xshift=1mm]f.south east) to ++(0,-.6) node[below]{$V\smash{'}$};
	  \draw[halo] (.1,-.8) node[below]{$V$} to ++(0,.4) node[dot]{};
	  \draw (d2) to ++(0,-.8) node[below]{$U$};	
	\end{pic}
  \]
  The functors $\cat{C}\restrict{u \leq v} \colon \cat{C}\restrict{v} \to \cat{C}\restrict{u}$ and $\cat{C} \to \cat{C}\restrict{x}$ preserve trace by construction.
\end{proof}

Next we turn to closedness.

\begin{lemma}\label{lem:localsectionsclosed}
  If a monoidal category $\cat{C}$ is closed and $u$ is a central idempotent, then $\cat{C}\restrict{u}$ is closed.
  If $u \leq v$, then the functor $\cat{C}\restrict{u \leq v} \colon \cat{C}\restrict{v} \to \cat{C}\restrict{u}$ is closed.
\end{lemma}
\begin{proof}
  Suppose that $(-) \otimes B \colon \cat{C} \to \cat{C}$ has a right adjoint $B \multimap (-) \colon \cat{C} \to \cat{C}$. 
  Write $\varepsilon_{B,C} \colon (B \multimap C) \otimes B \to C$ for the counit and $\eta_{A,B} \colon A \to (B \multimap (A \otimes B))$ for the unit.
  Define a functor $B \multimap^u (-) \colon \cat{C}\restrict{u} \to \cat{C}\restrict{u}$ by $B \multimap^u C = B \multimap C$ on objects, and by sending a morphism $f \in \cat{C}\restrict{u}(C,D) = \cat{C}(C \otimes U,D)$ to $B \multimap^u f$ defined as:
  \[\begin{pic}[xscale=5,yscale=1.5]
	\node (tl) at (0,1) {$(B \multimap C) \otimes U$};
	\node (tr) at (2,1) {$B \multimap D$};
	\node[font=\tiny] (bl) at (0,0) {$B \multimap \big( (B \multimap C) \otimes U \otimes B \big)$};
	\node[font=\tiny] (b) at (1,0) {$B \multimap \big( (B \multimap C) \otimes B \otimes U \big)$};
	\node[font=\tiny] (br) at (2,0) {$B \multimap (C \otimes U)$};
	\draw[->,dashed] (tl) to node[above]{$B \multimap^u f$} (tr);
	\draw[->,font=\tiny] (tl) to node[right]{$\eta_{(B \multimap C) \otimes U, B}$} (bl);
	\draw[->,font=\tiny] (bl) to node[below]{$B \multimap (1 \otimes \sigma_{U,B})$} (b);
	\draw[->,font=\tiny] (b) to node[below]{$B \multimap (\varepsilon_{B,C} \otimes 1)$} (br);
	\draw[->,font=\tiny] (br) to node[left]{$B \multimap f$} (tr);
  \end{pic}\]
  It is straightforward to see that $B \multimap^u (-)$ is functorial. 
  There is a bijection
  \[ 
    \cat{C}\restrict{u}(A, B \multimap^u C)
    \simeq 
    \cat{C}\restrict{u}(A \otimes B, C) 
  \]
  that sends $g \in \cat{C}\restrict{u}(A, B \multimap^u C) = \cat{C}(A \otimes U, B \multimap C)$ to the morphism
  \[\begin{pic}
	\node[morphism,width=6mm] (g) at (0,0) {$g$};
	\node[morphism,width=8mm,anchor=south west] (e) at ([yshift=5mm]g.north) {$\varepsilon_{B,C}$};
	\draw (g.north) to node[left=-1mm]{$B \multimap C$} (e.south west);
	\draw (e.north) to ++(0,.3) node[above]{$C$};
	\draw (g.south west) to ++(0,-.4) node[below]{$A$};
	\draw ([xshift=1mm]e.south east) to ++(0,-.8) to[out=-90,in=90] ++(-.5,-.5) node[below]{$B$};
  \draw[halo] (g.south east) to[out=-90,in=90] ++(.5,-.4) node[below]{$U$};
  \end{pic}\]
  in $\cat{C}\restrict{u}(A \otimes B, C)$.
  It is easy to see that this bijection is natural in $A$ and $C$.
  Thus $B \multimap^u (-)$ is right adjoint to $(-) \otimes B$, and $\cat{C}\restrict{u}$ is closed.

  It is clear that the functor $\cat{C}\restrict{u \leq v}$ preserves $\multimap$ strictly.
\end{proof}

\begin{lemma}\label{lem:stalksclosed}
  If $\cat{C}$ is a closed monoidal category and $x \subseteq \ZI(\cat{C})$ a prime filter, then $\cat{C}\restrict{x}$ is closed, and the functor $\cat{C} \to \cat{C}\restrict{x}$ is closed.
\end{lemma}
\begin{proof}
  The proof is very close to that of Lemma~\ref{lem:localsectionsclosed}. 
  The functor $B \multimap^x (-) \colon \cat{C}\restrict{x} \to \cat{C}\restrict{x}$ is defined as $C \mapsto B \multimap C$ on objects and as $[u,f] \mapsto [u, B\multimap^u f]$ on morphisms. This is well-defined and functorial as before.
  There is a natural bijection $\cat{C}\restrict{x}(A \otimes B,C) \simeq \cat{C}\restrict{x}(A, B \multimap C)$ that sends $[u,f]$ to $[u,g]$ where $g = (B \multimap f) \circ (B \multimap (1 \otimes \sigma)) \circ \eta \colon A \otimes U \to B \multimap C$ corresponds to $f \colon A \otimes B \otimes U \to C$ under the bijection $\cat{C}\restrict{u}(A \otimes B,C) \simeq \cat{C}\restrict{u}(A, B \multimap C)$. 
  To see that it is well-defined: if $[u,f] = [u',f'] \in \cat{C}\restrict{x}(A \otimes B,C)$, then for some central idempotent $v$ of $\cat{C}$,
  \begin{align*}
	& (B \multimap f) \circ (B \multimap (1 \otimes \sigma)) \circ \eta \circ (1 \otimes u' \otimes v) \\
	& = (B \multimap (f \otimes u' \otimes v)) \circ (B \multimap (1 \otimes \sigma)) \circ \eta \\
	& = (B \multimap (f' \otimes u \otimes v)) \circ (B \multimap (1 \otimes \sigma)) \circ \eta \\
	& = (B \multimap f') \circ (B \multimap (1 \otimes \sigma)) \circ \eta \circ (1 \otimes u \otimes v \otimes 1) 
  \end{align*}
  by naturality of $\eta$, so $[u,g] = [u',g']$. 
  Thus $\cat{C}\restrict{x}$ is closed, and the functor $\cat{C} \to \cat{C}\restrict{x}$ preserves $\multimap$ strictly.
\end{proof}

Call a monoidal category \emph{Boolean} when its central idempotent semilattice is a Boolean algebra.
In this case, the stalks are also Boolean, and in fact two-valued, as follows.

\begin{lemma}\label{lem:booleancongruence}
  If $L$ is a Boolean algebra and $x$ a prime filter, then $L \slash \smash{\sim}_x$ is again a Boolean algebra.
\end{lemma}
\begin{proof}
  It suffices to prove that $\sim_x$ is a congruence of lattices, because complements are uniquely defined in terms of joins and meets. Suppose $a \sim_x a'$ and $b \sim_x b'$. Say $a \wedge s = a' \wedge s$ and $b \wedge t = b' \wedge t$ for $s,t \in x$. For $r=s \wedge t$, then clearly $a \wedge b \wedge r = a' \wedge b' \wedge r$ so $a \wedge b \sim_x a' \wedge b'$. Similarly, $(a \vee b) \wedge r = (a \wedge r) \vee (b \wedge r) = (a' \wedge r) \vee (b' \wedge r) = (a' \vee b') \wedge r$, so $a \vee b \sim_x a' \vee b'$.
\end{proof}

\begin{corollary}\label{cor:Boolean}
  If a monoidal category $\cat{C}$ with universal finite joins of central idempotents is Boolean, and $x \subseteq \ZI(\cat{C})$ is a prime filter, then $\ZI(\cat{C}\restrict{x})=\{0,1\}$.
\end{corollary}
\begin{proof}
  By Lemmas~\ref{lem:zi:stalks} and~\ref{lem:booleancongruence}, $\ZI(\cat{C}\restrict{x})$ is again a Boolean algebra. 
  Hence if $u \in \ZI(\cat{C}\restrict{x})$, then $u \vee \neg u = 1$.
  By Lemma~\ref{lem:sublocal} $\ZI(\cat{C}\restrict{x})$ is also \changed{$\vee$-}local.
  Hence $u=1$ or $\neg u=1$, that is, $u \in \{0,1\}$.
\end{proof}

\changed{
For toposes $\cat{C}$, it is known that various logical properties are preserved by passing to stalks $\cat{C}\restrict{x}$, such the internal axiom of choice. Next we show that the external axiom of choice is preserved for monoidal categories $\cat{C}$.

\begin{lemma}
  If all epimorphism in a small monoidal category $\cat{C}$ split, $u$ is a central idempotent, and $x \subseteq \ZI(\cat{C})$ is a prime filter, then all epimorphisms in $\cat{C}\restrict{u}$ and $\cat{C}\restrict{x}$ also split.
\end{lemma}
\begin{proof}
  Let $f \colon A \to B$ be an epimorphism in $\cat{C}\restrict{u}$. That means that $f \otimes U \colon A \otimes U \otimes U \to B \otimes U$ is an epimorphism in $\cat{C}$. It is split by some $g \colon B \otimes U \to A \otimes U \otimes U$ in $\cat{C}$. Taking $h=(A \otimes u \otimes u) \circ g \colon B \otimes U \to A$, it follows that $f \circ (g \otimes U) = (B \otimes u \otimes u) \circ (f \otimes U \otimes U) \circ (g \otimes U) = B \otimes u \otimes u$ in $\cat{C}$, that is, $f \circ g = B$ in $\cat{C}\restrict{u}$, so $f$ is a split epimorphism in $\cat{C}\restrict{u}$. The same reasoning holds for $\cat{C}\restrict{x}$.
\end{proof}
}

Next, we turn to limits. In the cartesian case, as in Lemma~\ref{lem:slice}, it is well known that limits in the slice category are computed as in the base category. In the monoidal setting we need to be more careful.

\begin{lemma}
  Let $\cat{J}$ be a small category, and $\cat{C}$ a monoidal category with $\cat{J}$-shaped limits.
  \begin{itemize}
  	\item 
  	If $\cat{C}$ has universal finite joins of central idempotents, and $u$ is a central idempotent, then $\cat{C}\restrict{u}$ has $\cat{J}$-shaped limits, and the functors $\cat{C}\restrict{u \leq v} \colon \cat{C}\restrict{v}\to\cat{C}\restrict{u}$ preserve them.
  	\item 
  	If $\cat{C}$ has universal finite joins of central idempotents, $x \subseteq \ZI(\cat{C})$ is a prime filter, and $\cat{J}$ is finite, then $\cat{C}\restrict{x}$ has $\cat{J}$-shaped limits, and the functor $\cat{C}\to\cat{C}\restrict{x}$ preserves them.
  	\item
  	If $\cat{C}$ has universal joins of central idempotents, and $x \subseteq \ZI(\cat{C})$ is a prime filter, then $\cat{C}\restrict{x}$ has $\cat{J}$-shaped limits, and the functor $\cat{C}\to\cat{C}\restrict{x}$ preserves them.
  \end{itemize}
\end{lemma}
\begin{proof}
  For the sake of clarity we will consider equalisers, but the proof works for arbitrary (finite) shapes $\cat{J}$. 
  Let $f,g \in \cat{C}\restrict{u}(X,Y)$. Let $l \colon E \to X \otimes U$ be their equaliser in $\cat{C}$. 
  Then $e = (X \otimes u \otimes u) \circ (l \otimes U)$ defines a morphism $E \to X$ in $\cat{C}\restrict{u}$.
  Now $f \circ e = f \circ g$ in $\cat{C}\restrict{u}$, by the central equation~\ref{eq:central}:
  \[
    \begin{pic}
	  \node[morphism,width=8mm] (f) at (0,1) {$f$};
	  \node[morphism,width=4mm] (l) at (-.6,0) {$l$};
	  \node[dot] (d) at (.15,-.5) {};
	  \draw (f.north) to ++(0,.3) node[above]{$Y$};
	  \draw ([xshift=-1mm]f.south west) to[out=-90,in=90] node[left]{$X$} (l.north west);
	  \draw (l.south) to ++(0,-.6) node[below]{$E$};
	  \draw (d) to ++(0,-.3) node[below]{$U$};
	  \draw (d) to[out=180,in=-90,looseness=.5] ++(-.2,.5) node[dot]{};
	  \draw (d) to[out=0,in=-90,looseness=.5] ([xshift=1.5mm]f.south east);
	  \draw (l.north east) to ++(0,.2) node[right]{$U$} node[dot]{};
    \end{pic}
    =
    \begin{pic}
	  \node[morphism,width=8mm] (f) at (0,1) {$f$};
	  \node[morphism,width=8mm] (l) at (0,0) {$l$};
	  \node[dot] (d) at (.4,-.5) {};
	  \draw (f.north) to ++(0,.3) node[above]{$Y$};
	  \draw (l.north west) to node[left]{$X$} (f.south west);
	  \draw (l.north east) to node[right]{$U$} (f.south east);
	  \draw (l.south) to ++(0,-.7) node[below]{$E$};
	  \draw (d) to ++(0,-.4) node[below]{$U$};
    \end{pic}
    =
    \begin{pic}
	  \node[morphism,width=8mm] (f) at (0,1) {$g$};
	  \node[morphism,width=8mm] (l) at (0,0) {$l$};
	  \node[dot] (d) at (.4,-.5) {};
	  \draw (f.north) to ++(0,.3) node[above]{$Y$};
	  \draw (l.north west) to node[left]{$X$} (f.south west);
	  \draw (l.north east) to node[right]{$U$} (f.south east);
	  \draw (l.south) to ++(0,-.7) node[below]{$E$};
	  \draw (d) to ++(0,-.4) node[below]{$U$};
    \end{pic}
    =
    \begin{pic}
	  \node[morphism,width=8mm] (f) at (0,1) {$g$};
	  \node[morphism,width=4mm] (l) at (-.6,0) {$l$};
	  \node[dot] (d) at (.15,-.5) {};
	  \draw (f.north) to ++(0,.3) node[above]{$Y$};
	  \draw ([xshift=-1mm]f.south west) to[out=-90,in=90] node[left]{$X$} (l.north west);
	  \draw (l.south) to ++(0,-.6) node[below]{$E$};
	  \draw (d) to ++(0,-.3) node[below]{$U$};
	  \draw (d) to[out=180,in=-90,looseness=.5] ++(-.2,.5) node[dot]{};
	  \draw (d) to[out=0,in=-90,looseness=.5] ([xshift=1.5mm]f.south east);
	  \draw (l.north east) to ++(0,.2) node[right]{$U$} node[dot]{};
    \end{pic}  
  \]
  Suppose that $f \circ e' = g \circ e'$ for some $e' \colon E' \to X$ in $\cat{C}\restrict{u}$, so $e' \colon E' \otimes U \to X$ in $\cat{C}$. 
  Setting $l' = (e' \otimes U) \circ (E' \otimes U \otimes u)^{-1} \colon E' \otimes U \to X \otimes U$ in $\cat{C}$, then $f \circ l' = g \circ l'$ in $\cat{C}$.
  Hence there is a mediating map $m \colon E' \otimes U \to E$ with $l' = l \circ m$ in $\cat{C}$. 
  \[
    \begin{pic}
	  \node[morphism,width=8mm] (f) at (0,1) {$f$};
	  \node[morphism,width=4mm,anchor=north] (e) at ([yshift=-5mm]f.south west) {$e'$};
	  \node[dot] (d) at ([xshift=2.5mm,yshift=-3mm]e.south east) {};
	  \draw (e.north) to node[left]{$X$} (f.south west);
	  \draw (f.north) to ++(0,.3) node[above]{$Y$};
	  \draw (e.south west) to ++(0,-.5) node[below]{$E'$};
	  \draw (d) to ++(0,-.2) node[below]{$U$};
	  \draw (d) to[out=180,in=-90] (e.south east);
	  \draw (d) to[out=0,in=-90,looseness=.5] ([xshift=1mm]f.south east);
    \end{pic}
    =
    \begin{pic}
	  \node[morphism,width=8mm] (f) at (0,1) {$g$};
	  \node[morphism,width=4mm,anchor=north] (e) at ([yshift=-5mm]f.south west) {$e'$};
	  \node[dot] (d) at ([xshift=2.5mm,yshift=-3mm]e.south east) {};
	  \draw (e.north) to node[left]{$X$} (f.south west);
	  \draw (f.north) to ++(0,.3) node[above]{$Y$};
	  \draw (e.south west) to ++(0,-.5) node[below]{$E'$};
	  \draw (d) to ++(0,-.2) node[below]{$U$};
	  \draw (d) to[out=180,in=-90] (e.south east);
	  \draw (d) to[out=0,in=-90,looseness=.5] ([xshift=1mm]f.south east);    	
    \end{pic}
    \qquad\qquad
    l' 
    = 
    \begin{pic}
	  \node[morphism,width=5mm] (e) at (0,0) {$e'$};
	  \node[dot] (d)at (.5,-.5) {};
	  \draw (e.north) to ++(0,.6) node[above]{$X$};
	  \draw (e.south west) to ++(0,-.8) node[below]{$E'$};
	  \draw (d) to ++(0,-.5) node[below]{$U$};
	  \draw (d) to[out=180,in=-90] (e.south east);
	  \draw (d) to[out=0,in=-90,looseness=.5] ++(.3,1.3) node[above]{$U$};
    \end{pic}
    =
    \begin{pic}
	  \node[morphism,width=6mm] (l) at (0,.4) {$l$};
	  \node[morphism,width=6mm] (m) at (0,-.4) {$m$};
	  \draw (m.north) to node[left]{$E$} (l.south);
	  \draw (l.north west) to ++(0,.3) node[above]{$X$};
	  \draw (l.north east) to ++(0,.3) node[above]{$U$};
	  \draw (m.south west) to ++(0,-.3) node[below]{$E'$};
	  \draw (m.south east) to ++(0,-.3) node[below]{$U$};
    \end{pic}
  \]
  But then $e' = e \circ m$ in $\cat{C}\restrict{u}$:
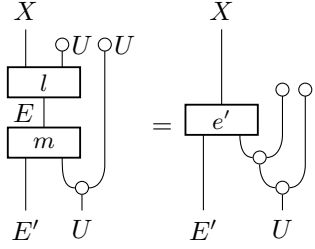
  \[
    \begin{pic}
	  \node[morphism,width=6mm] (l) at (0,.4) {$l$};
	  \node[morphism,width=6mm] (m) at (0,-.4) {$m$};
	  \draw (m.north) to node[left]{$E$} (l.south);
	  \draw (l.north west) to ++(0,.5) node[above]{$X$};
	  \draw (l.north east) to ++(0,.3) node[right]{$U$} node[dot]{};
	  \draw (m.south west) to ++(0,-.7) node[below]{$E'$};
	  \node[dot] (d) at (.5,-1) {};
	  \draw (d) to[out=180,in=-90] (m.south east);
	  \draw (d) to ++(0,-.3) node[below]{$U$};
	  \draw (d) to[out=0,in=-90] ++(.3,.3) to ++(0,1.6) node[right]{$U$} node[dot]{};
    \end{pic}
    = 
    \begin{pic}
	  \node[morphism,width=6mm] (e) at (0,0) {$e'$};
	  \draw (e.north) to ++(0,1) node[above]{$X$};
	  \draw (e.south west) to ++(0,-1) node[below]{$E'$};
	  \node[dot] (l) at (.5,-.5) {};
	  \node[dot] (r) at (.8,-.9) {};
	  \draw (r) to[out=180,in=-90] (l);
	  \draw (l) to[out=180,in=-90] (e.south east);
	  \draw (r) to ++(0,-.3) node[below]{$U$};
	  \draw (l) to[out=0,in=-90] ++(.3,.3) to ++(0,.6) node[dot]{};
	  \draw (r) to[out=0,in=-90] ++(.3,.3) to ++(0,1) node[dot]{};
    \end{pic}
  \]
  Conversely, if $e'=e \circ m$ in $\cat{C}\restrict{u}$, then $(e' \otimes U) \circ (E' \otimes U \otimes u)^{-1} = l \circ m$ in $\cat{C}$, so $m$ is the unique such morphism.
  The functors $\cat{C}\restrict{u \leq v} \colon \cat{C}\restrict{v}\to\cat{C}\restrict{u}$ preserve these limits by construction.

  The same idea works for $\cat{C}\restrict{x}$, except that we may need the diagram $\cat{J}$ to be finite, depending on the amount of universal joins of central idempotents available in $\cat{C}$. This is best illustrated for products. Suppose that $\cat{C}$ has universal finite joins of central idempotents and has finite products, then $\prod_{i=1}^n X_i$ is also a product in $\cat{C}\restrict{x}$. Define the projections to be $[1,\pi_i] \colon \prod_i X_i \to X_i$, and given finitely many $[u_i,f_i] \colon A \otimes U_i \to X_i$, define their tuple to be $[u_1 \wedge \ldots \wedge u_n, \langle f_i \rangle] \colon A \to \prod_i X_i$ in $\cat{C}\restrict{x}$. It is easily checked that this is well-defined and satisfies the universal property. 
  If $\cat{C}$ had universal joins of central idempotents, we could have defined the tuple $[\bigwedge u_i, \langle f_i \rangle] \colon A \to \prod X_i$ for an arbitrary number of morphisms $A \to X_i$.
  In either case, the functor $\cat{C}\to\cat{C}\restrict{x}$ preserves these limits by construction.  
\end{proof}

Finally, let us mention two properties that are preserved only in certain cases.

If $\cat{C}$ has colimits of a certain shape $\cat{J}$, and $U \otimes (-) \colon \cat{C} \to \cat{C}$ preserves these colimits for a central idempotent $u$, then $\cat{C}\restrict{u}$ inherits these $\cat{J}$-shaped colimits, and the functor $\cat{C}\to\cat{C}\restrict{u}$ preserves them. If $\cat{J}$ is finite or $\cat{C}$ has universal joins of central idempotents, then the same holds for $\cat{C}\restrict{x}$.

If $\cat{C}$ has a dagger~\cite[Section~2.3]{heunenvicary:cqm} then $\cat{C}\restrict{u}$ inherits it if and only if the central idempotent $u$ is (represented by) an isometry, and $\cat{C}\restrict{x}$ inherits it if and only if for every central idempotent $v$ there is a central idempotent $u \leq v$ that is (represented by) an isometry.

\section{Examples}\label{sec:examples}

This brief section illustrates the sheaf representation theorem in several example cases that are genuinely monoidal and could not have been handled using sheaf representation theorems for toposes.
\changed{We start by considering three posetal examples.}

\begin{example}\label{ex:frame}
  A topological space $X$ has a frame of open sets $L=\mathcal{O}(X)$ that may be regarded as a stiff monoidal category with universal joins of central idempotents (but that is not a topos unless $X$ is empty).
  Because $\ZI(L)\simeq L$, the completely prime spectrum $\Spec(\ZI(L))$ is (homeomorphic to) $X$ itself.
  Theorem~\ref{thm:main} says that $L$ is isomorphic to the global sections of a sheaf of \changed{$\bigvee$-}local monoidal categories on $X$. If $x$ is a point of $X$, then the stalk $L\restrict{x}$ is the quotient of $L$ by the filter $x$, that is, the quotient $L\slash\mathop{\sim}$ by the equivalence relation $\sim$ where $U \sim V$ when $U \cap W = V \cap W$ for some $W \in L$. 
  This is sometimes called the frame of germs of open subsets of $X$ with respect to $x$~\cite[I.6.10]{bourbaki:generaltopology}.
  Concretely, the stalk $L\restrict{x}$ has the same objects as $L$, that is, open subsets $U$ of $X$. There is a unique morphism $U \to V$ in $L\restrict{x}$ when there is an open neighbourhood $W$ of $x$ with $U \cap W \subseteq V$. 
  \begin{itemize}
	\item If $x \in V$, then there is a morphism $U \to V$: take $W=V$;
	\item If $x \in U$ and $x \not\in V$, then there is no morphism $U \to V$: for any open neighbourhood $W$ of $x$, there is a point $x$ in $U \cap W$ that is not in $V$;
	\item If $x \not\in\overline{U}$, then there is a morphism $U \to V$: there is a open neighbourhood $W$ of $x$ disjoint from $U$;
  \end{itemize}
  Hence all open neighbourhoods of $x$ are isomorphic as objects in $L\restrict{x}$, and they are all terminal.
  Similarly, all open sets whose closure does not contain $x$ are isomorphic as objects in $L\restrict{x}$, and they are all initial.
  Therefore, if $x$ is an isolated point, $L\restrict{x}$ is equivalent to the partially ordered set $0 \leq 1$ regarded as a 2-object category.
  But in general more subtle behaviour can occur:
  \begin{itemize}
	\item If $x \in\overline{U}$ and $x\not\in\overline{V}$, then there is no morphism $U \to V$: because $x \not\in\overline{V}$ there is an open neighbourhood $W$ of $x$ disjoint from $V$, but because $x \in \overline{U}$ this neighbourhood is not disjoint from $U$.
	\item If $x \in \partial{U} = \overline{U} \setminus U$ and $x \in \partial{V} = \overline{V} \setminus V$, there may be a morphism $U \to V$ or there may not be. For example, take $X=\mathbb{R}$, $x=0$, and $U=(-1,0)$.
	If $V=(0,1)$, then any open neighbourhood $W$ of $x$ contains points of $U$ that are not in $V$, so there is no morphism $U \to V$. 
	But if $V=(-2,0)$, then $U \subseteq V$, so $W=X$ shows that there is a morphism $U \to V$.
  \end{itemize}  
  Nevertheless, the global sections of the sheaf $U \mapsto L\restrict{U}$
  correspond to continuous functions $X \to \{0,1\}$.
\end{example}

\begin{example}
  A Boolean algebra $B$ may be regarded as a stiff monoidal category with finite universal joins of central idempotents.
  Because $\ZI(B) \simeq B$, in this case $\Spec(\ZI(B))$ is the Stone space of $B$.
  Theorem~\ref{thm:main} hence says that any Boolean algebra is isomorphic to the global sections of a sheaf of \changed{$\vee$-}local rings on a Stone space.
  By Corollary~\ref{cor:Boolean} the stalks are two-valued, so this becomes Stone's representation theorem~\cite{johnstone:stonespaces}: any Boolean algebra is isomorphic \changed{to} the algebra of clopen subsets of its Stone space.
  Going from \changed{$\vee$-}local to \changed{$\bigvee$-}local, this extends to complete Boolean algebras and Stonean spaces.

  In fact, this holds more generally for a distributive lattice $L$ and its prime spectrum $X$. 
  As the quotient of $L$ by a prime filter $x$ is always $\{0,1\}$,
  \changed{ we obtain a sheaf whose stalks are all isomorphic to the two-element lattice. Global sections are functions $X \to \{0,1\}$ that are continuous with respect to the discrete topology on $\{0,1\}$, that is, the clopen subsets of $X$, and so correspond to elements of $L$.}

  \changed{Similarly, starting with a locale $L$ regarded as a monoidal category with universal joins of central idempotents, Theorem~\ref{thm:main2} gives the Stone representation on a sober space, and recovers the locale $L$ if it is spatial.}
\end{example}

\begin{example}
  A quantale $Q$ may be regarded as a monoidal category with universal joins of central idempotents.
  Theorem~\ref{thm:main} shows that any quantale is equivalent to the category of global sections of a sheaf of monoidal categories.
  This gives a different view on works combining sheaves and quantales~\cite{borceuxvdbossche:quantalesheaves,resende:quantalesheaves,heymans:quantalesheaves}.
  The central idempotents in $Q$ are the central elements $u$ satisfying $u \cdot u = u \leq e$. The category $Q\restrict{u}$:
  \begin{itemize}
    \item has objects $q \in Q$;
    \item there is a unique morphism $p \to q$ when $p \cdot u \leq q$; 
    \item composition is uniquely determined: $p \cdot u \leq q$ and $q \cdot u \leq v$ imply $p \cdot u \leq v$;
    \item the identity morphism is the equality $q \cdot e = q$.
  \end{itemize} 
  If $x \subseteq \ZI(Q)$ is a completely prime filter, $Q\restrict{x}$ has:
  \begin{itemize}
  \item objects $q \in Q$;
  \item a unique morphism $p \to q$ if there exists $u \in x$ satisfying $p \cdot u \leq q$.
  \end{itemize}
  The partially ordered sets $Q\restrict{x}$ still have finite meets and joins, bottom and top elements, and a multiplication that distributes over joins. To see that $Q\restrict{x}$ inherits meets: $p_i \cdot 1 = p_i \leq p_1 \wedge p_2$ and $1 \in x$ so $p_i \leq_x p_1 \wedge p_2$, and if $q \leq_x p_i$, say because $q \cdot u_i \leq p_i$, then $q \cdot (u_1 \wedge u_2) \leq q \cdot u_1 \wedge q \cdot u_2 \leq p_1 \wedge p_2$ with $u_1 \wedge u_2 \in x$, and so $q \leq_x p_1 \wedge p_2$; a similar reasoning holds for joins.
\end{example}

\changed{Next we discuss three examples that are not posetal.}

\begin{example}\label{ex:hilbagain}
  If $X$ is a locally compact Hausdorff space, the category $\cat{Hilb}_{C_0(X)}$ of Hilbert $C_0(X)$-modules is a stiff monoidal category with universal joins of central idempotents. Here $\ZI(\cat{Hilb}_{C_0(X)}) \simeq \mathcal{O}(X)$~\cite[3.16]{enriquemolinerheunentull:tensortopology}, so as in Example~\ref{ex:frame}, the representation theorem says that $\cat{Hilb}_{C_0(X)}$ is isomorphic to the global sections of a sheaf of \changed{$\bigvee$-}local monoidal categories over $X$. 
  In this case, every stalk $\cat{Hilb}_{C_0(X)}\restrict{x}$ is the category $\cat{Hilb}$ of Hilbert spaces (see also~\cite[2.5,2.8]{heunenreyes:frobenius}). Thus we recover Takahashi's representation theorem of Hilbert modules as continuous fields of Hilbert spaces~\cite{takahashi:hilbertmodules}.
\end{example}

\begin{example}
  Any commutative ring $R$ has a category $\cat{Mod}_R$ of modules that is a stiff monoidal category with finite universal joins of central idempotents.
  As in Example~\ref{ex:modules}, the central idempotents of $\cat{Mod}_R$ are linear maps $A \to R$ whose image is an idempotent ideal of $R$.
  If $R$ is semisimple, any idempotent ideal is generated by an idempotent element, and hence $\ZI(\cat{Mod}_R)$ includes the set $E(R)=\{e \in R \mid e^2=e \}$.
  Therefore $\Spec(\ZI(\cat{Mod}_R))$ is a quotient of the Pierce spectrum of $R$~\cite[V.2]{johnstone:stonespaces}.
  Theorem~\ref{thm:main} now says that an $R$-module corresponds to a global section of a sheaf of \changed{$\vee$-}local modules over a quotient of the Pierce spectrum of $R$. Indeed, any semisimple commutative ring is a finite product of fields.
\end{example}


\begin{example}\label{example:linearnonlinear}
  Consider a linear-non-linear model of linear logic~\cite{benton:linearnonlinear}, that is, a symmetric monoidal adjunction between a symmetric monoidal closed category $\cat{L}$ and a cartesian closed category $\cat{C}$. 
  \[\begin{pic}
    \node (C) at (0,0) {$\cat{C}$};
    \node (L) at (3,0) {$\cat{L}$};
    \draw[->] ([yshift=2mm]C.east) to node[above]{$F$} ([yshift=2mm]L.west);
    \draw[<-] ([yshift=-2mm]C.east) to node[below]{$G$} ([yshift=-2mm]L.west);
    \draw[draw=none] (C) to node{$\perp$} (L);
  \end{pic}\]
  Because $F$ is automatically strong monoidal~\cite[Proposition~1]{benton:linearnonlinear}, there is a semilattice morphism $\ZI(\cat{C}) \to \ZI(\cat{L})$, and a functor $\cat{C}\restrict{s} \to \cat{L}\restrict{F(s)}$ for each $s \in \ZI(\cat{C})$.
  If $\cat{C}$ and $\cat{L}$ both have universal joins of (finite) central idempotents, then $\ZI(\cat{C}) \to \ZI(\cat{L})$ preserves (finite) joins, because $F$ is a left adjoint and so preserves colimits.
  If $G$ preserves central idempotents, we obtain a Galois connection between $\ZI(\cat{C})$ and $\ZI(\cat{L})$, and a functor $\cat{L}\restrict{u} \to \cat{C}\restrict{G(u)}$ for each $u \in \ZI(\cat{L})$.
\end{example}


\section{Functoriality}\label{sec:functoriality}

This section shows that the main construction of Theorem~\ref{thm:main} is functorial. We start by defining the appropriate category of sheaves of monoidal categories, and prove that it is dual to the category of monoidal categories. Lemma~\ref{lem:lari} shows that the sheaves of monoidal categories that arise from the representation theorem are (a categorification of) flabby sheaves~\cite{godement:flabby}.

\begin{definition}
  A sheaf $P \colon \mathcal{O}(X)\op \to \cat{MonCat}$ of monoidal categories is \emph{flabby} when the functors $P_{u \leq v} \colon P_v \to P_u$ all act the same on objects, and each has a left adjoint $P^{u \leq v} \colon P_v \to P_u$ and the unit of the adjunction is invertible.
\end{definition}

Recall that a topological space is \emph{spectral} \changed{(or \emph{coherent})} when it is $T_0$, compact, sober, and its compact open subsets form a base and are closed under finite intersections. A continuous function between spectral spaces is a \emph{morphism of spectral spaces} when the preimage of a compact open subset is again compact. The category of spectral spaces is dually equivalent to the category of distributive lattices~\cite{hochster:primeideal,johnstone:stonespaces}.

\begin{definition}
  Write $\cat{MonScheme}$ for the following category. 
  Objects are pairs of a topological space $X$ and a flabby sheaf $P \colon \mathcal{O}(X)\op \to \cat{MonCat}$ of monoidal categories where all $P_u$ have the same objects and $X=\ZI(P_1)$.
  A morphism $(X,P) \to (Y,Q)$ consists of a continuous function $\varphi \colon X \to Y$ and a natural transformation $F_v \colon Q(v) \to P(\varphi^{-1}(v))$ such that
  \begin{equation}\label{eq:localfunctor}
	 F_1(u) = \theta_{I,U}
  \end{equation} 
  or in other words $\varphi^{-1}(u) = \theta_I \circ F_1(u)$.
  Write $\cat{MonScheme}_{\cat{l}}$ for the full subcategory of sheaves of \changed{$\bigvee$-}local categories.
  Write $\cat{MonScheme}_\cat{sl}$ for the subcategory of sheaves of \changed{$\vee$-}local categories where $X$ is a spectral space, and morphisms where $\varphi$ is a morphism of spectral spaces.
\end{definition}

\changed{Recall from Definition~\ref{def:moncat} that $\cat{MonCat_{fj}}$ has as objects small monoidal categories with universal finite joins of central idempotents, and as morphisms lax monoidal functors $F$ whose coherence morphism $I \to F(I)$ is invertible that preserve finite joins of central idempotents. Similarly, $\cat{MonCat_j}$ is the subcategory of small monoidal categories whose central idempotents have universal joins and form a spatial frame, and morphisms $F$ that preserve joins of central idempotents.}

\begin{theorem}\label{thm:functoriality}
  The representation of Proposition~\ref{prop:sheaf} extends to equivalences of categories $\cat{MonCat}_\cat{fj} \to \cat{MonScheme}_{\cat{sl}}\op$ and $\cat{MonCat}_\cat{j} \to \cat{MonScheme}_{\cat{l}}\op$.
\end{theorem}
\begin{proof}
  If $F \colon \cat{C} \to \cat{D}$ is a morphism in $\cat{MonCat}_\cat{fj}$ then it induces a semilattice morphism $\ZI(\cat{C}) \to \ZI(\cat{D})$, and hence a continuous function $\varphi \colon \Spec(\ZI(\cat{D})) \to \Spec(\ZI(\cat{C}))$, by Lemma~\ref{lem:functorspreservingZI}. 
  Furthermore, if $u \in \ZI(\cat{C})$, then $F$ induces a morphism $\cat{C}\restrict{u} \to \cat{D}\restrict{F(u)}$ in $\cat{MonCat}$ that acts as $F$ on objects and that sends a morphism $f \colon A \otimes U \to B$ to $F(f) \circ \theta_{A,U} \colon F(A) \otimes F(U) \to F(B)$. 
	This is natural with respect $\cat{C}\restrict{u \leq v}$, and satisfies the condition $F_1(u)=\theta_{I,U}$, and so gives a well-defined morphism $\cat{D}\restrict{(-)} \to \cat{C}\restrict{(-)}$ in $\cat{MonScheme}_\cat{sl}$.
  This assignment and its counterpart $\cat{MonCat}_{\cat{j}} \to \cat{MonScheme}_{\cat{l}}\op$ are clearly functorial, faithful, and injective on objects.

	To see that these functors are full, suppose $(\varphi,\{F_u\}_u)$ is a map in $\cat{MonScheme}$ from $P_-=\cat{D}\restrict{-}$ to $Q_-=\cat{C}\restrict{-}$. 
  Then $F=F_1$ is a morphism $\cat{C} \to \cat{D}$ in $\cat{MonCat}$. 
  It follows from naturality of that all $F_u$ act the same on objects.
  Regarding a morphism $f \colon A \otimes U \to B$ in $\cat{C}$ as a morphism in $\cat{C}\restrict{u}$, we find:
  \begin{align*}
    & F_u\big( A \otimes U \stackrel{f}{\to} B \big) \\
    & =  F_u\big( A \otimes U \otimes U \stackrel{f \otimes u}{\longrightarrow} B \big) \:\circ_{Fu}\: F_u \big( A \otimes U \stackrel{(A \otimes U \otimes u)^{-1}}{\longrightarrow} A \otimes U \otimes U \big) \\
    & = F_u\big( \cat{C}\restrict{u \leq 1}\big( A \otimes U \stackrel{f}{\to} B\big)\big) \:\circ _{Fu}\: F_u\big( A \otimes U \stackrel{(A \otimes U \otimes u)^{-1}}{\longrightarrow} A \otimes U \otimes U \big) \\
    & = \cat{D}\restrict{F(u) \leq 1} \big(F_1 \big( A \otimes U \stackrel{f}{\to} B \big) \big) \:\circ _{Fu}\: F_u\big( A \otimes U \stackrel{(A \otimes U \otimes u)^{-1}}{\longrightarrow} A \otimes U \otimes U \big) \\
    & = F_1(f) \circ F_u \big( A \otimes U \stackrel{A \otimes U}{\longrightarrow} A \otimes U \big)
  \end{align*}
	Now, observe that $A \otimes U \colon A \otimes U \to A \otimes U$, regarded as a morphism $A \to A \otimes U$ in $\cat{C}\restrict{u}$, is an isomorphism in $\cat{C}\restrict{u}$, with inverse $A \otimes u \otimes u \colon A \otimes U \otimes U \to A$, regarded as a morphism $A \to_u A \otimes U$ in $\cat{C}\restrict{u}$. 
  Hence	$F_u(A \otimes U)$ is inverse to $F_u(A \otimes u \otimes u)$ in $\cat{D}\restrict{F(u)}$.
	But now it follows from $F_1(A \otimes u) \circ \theta_{A,U} = F(A) \otimes F(U)$ that $F_u(A \otimes u \otimes u) = F_1(A \otimes u) \otimes F(U)$ is inverted by $\theta_{A,U} \colon F(A) \otimes F(U) \to F(A \otimes U)$, regarded as a morphism $F(A) \to F(A \otimes U)$ in $\cat{D}\restrict{F(u)}$.
	Thus $F_u(A \otimes U) = \theta_{A,U}$, and so
  $F_u(f) = F_1\big(f \colon A \otimes U B\big) \circ \theta_{A,U}$.
  Therefore the morphism $(\varphi,\{F_u\})$ in $\cat{MonScheme}$ is indeed induced by the morphism $F$ in $\cat{MonCat}$.

  Similarly, any object $P$ in $\cat{MonScheme}$ is (monoidally naturally isomorphic to) $\cat{C}\restrict{-}$ for $\cat{C}=P_1$, because there is a bijection between morphism $f \colon A \to B$ in $P_u$ and morphisms $g \colon A \otimes U \to B$ in $P(1)$ by flabbiness. Thus the functors $\cat{MonCat}_\cat{fj} \to \cat{MonScheme}_\cat{sl}\op$ and $\cat{MonCat}_\cat{j} \to \cat{MonScheme}_\cat{l}\op$ are essentially surjective, and so equivalences.
\end{proof}

\changed{Apart from the issue surrounding projectivity of the tensor unit discussed in footnote~\ref{footnote:hyperlocal}, the previous theorem generalises from the topos case the functoriality result that there is a dual equivalence between the following two categories~\cite{awodey:sheafrepresentationsinlogic}: Boolean pretoposes and pretopos morphisms; and so-called affine logical schemes, consisting of a stack $\widetilde{\cat{A}}$ on a topological groupoid $\cat{A}$, with as morphisms a continuous functor $F \colon \cat{A} \to \cat{B}$ together with a pretopos morphism $\widetilde{F} \colon \widetilde{\cat{B}} \to F_*(\widetilde{\cat{A}})$.}

Next we adapt the equivalence between \'etale bundles and sheaves (see~\cite[II.6.5]{maclanemoerdijk:sheaves} and~\cite{cardechi:etale}) to the current setting.

\begin{lemma}
  Morphisms $(X,P) \to (Y,Q)$ in $\cat{MonScheme}$ that do not necessarily satisfy~\eqref{eq:localfunctor} are equivalently given by continuous functions $\varphi \colon X \to Y$ together with a family $F_x \colon Q_{\varphi(x)} \to P_x$ of morphisms in $\cat{MonCat}$ between stalks indexed by $x \in X$ satisfying:
  \begin{itemize} 
    \item all $F_x$ act the same on objects;
    \item for all $f \colon A \to B$ in $Q(v)$, the functions
      $X \supseteq \varphi^{-1}(v) \to \Lambda P$ given by
      \[
        x \mapsto F_x\big( [f]_{\varphi(x)} \big)
      \]
      are continuous, where $[-]_x \colon P_u \to P_x = \colim_{x \in u} P_u$ are the stalk maps, and $\Lambda P = \coprod_{A,B \in P_x, x \in X} P_x(A,B)$ has a base of open sets
      $
        \{ [g]_x \mid x \in u \}
      $
      for $g \colon A \to B$ in $P(u)$;
    \item the coherence morphisms $\theta_I^x \colon I \to F_x(I)$ in
      $Q_x$ form a continuous function $X \to \coprod_{x \in X} Q_x\big(I,F_x(I)\big)$, where the latter has a base of open sets
      $
        \{ [\theta^u_I]_x \mid x \in u\}
      $
      for $\theta^u_I \colon I \to F_x(I)$ in $Q_u$;
    \item the coherence morphisms 
      $\theta_{A,B}^x \colon F_x(A) \otimes F_x(B) \to F_x(A \otimes B)$ in $Q_x$ form a continuous function $X \to \coprod_{A,B \in P(x), x \in X} Q_x\big( F_x(A) \otimes F_x(B), F_x(A \otimes B) \big)$, where the latter has a base of open sets
      $
        \{ [ \theta_{A,B}^u]_x \mid x \in u\}
      $
      for $\theta_{A,B}^u \colon F_x(A) \otimes F_x(B) \to F_x(A \otimes B)$ in $Q_u$.
  \end{itemize}
\end{lemma}
\begin{proof}
  Fix a continuous function $\varphi \colon X \to Y$.
  Given $\{F_v\}_{v \in \mathcal{O}(Y)}$, take $F_x$ to be the morphism $Q_{\varphi(x)} \to P_x$ in $\cat{MonCat}$ induced by the fact that the stalk $P_x$ is the colimit of $P_{\varphi^{-1}(v)}$ over $x \in \varphi^{-1}(x)$, and the stalk $Q_{\varphi(x)}$ is the colimit of $Q_v$ over $\varphi(x) \in v$. This satisfies the conditions by construction.

  Conversely, suppose given $\{F_x\}_{x \in X}$, and fix $v \in \mathcal{O}(Y)$. On objects, set $F_v(A)=F_x(A)$ for any $x \in X$.
  Given $f \colon A \to B$ in $Q_v$, the continuity condition lets us pick opens $v_x$ around each point $x \in u=\varphi^{-1}(v)$ and $g_x \colon A \to B$ in $P(\varphi^{-1}(v_x))$ with $F_x\big([f]_{\varphi(x)}\big) = [g_x]_x$. By universality of joins in $P_u$, we can paste these into $F_v(f) := \bigvee_{x \in u} g_x \colon A \to B$ in $P_u$. In the \changed{$\vee$-}local case, $X$ is a spectral space and hence compact, so finitely many $v_x$ already cover $X$, and universal finite joins of central idempotents in $P_u$ suffice to define $F(f)$. Lemma~\ref{lem:epimonoisolocal} makes $F$ functorial.

  Similarly, there are opens $v_x$ around each point $x \in u$ and $\theta_I^{v_x} \colon I \to F(I)$ and $\theta_{A,B}^{v_x} \colon F(A) \otimes F(B) \to F(A \otimes B)$ in $P_{\varphi^{-1}(v_x)}$ with $[\theta_I^{v_x}]_x=\theta_I^x$ and $[\theta_{A,B}^{v_x}]_x=\theta_{A,B}^x$. These paste into $\theta_I := \bigvee_x \theta_I^{v_x} \colon I \to F(I)$ and $\theta_{A,B} := \bigvee_x \theta_{A,B}^{v_x} \colon F(A) \otimes F(B) \to F(A \otimes B)$ in $\cat{D}$. These make the functor $F_v$ into a morphism in $\cat{MonCat}$ by Lemma~\ref{lem:epimonoisolocal} because $\theta_I^x$ and $\theta^x_{A,B}$ make $F_x$ into a morphism in $\cat{MonCat}$ for each point $x$. Thus we have defined a natural transformation $F_v \colon Q_v \to P_{\varphi^{-1}(v)}$. 

  It is straightforward to verify that these two constructions are inverses.
\end{proof}

We can now recognise~\eqref{eq:localfunctor} as a condition on the morphisms $F_x$ between \changed{($\vee$- or $\bigvee$-)}local categories, to make the previous lemma into an equivalence.

\begin{definition}\label{def:localfunctor}
  A \emph{morphism of \changed{($\vee$- or $\bigvee$-)}local monoidal categories} $F \colon \cat{C} \to \cat{D}$ is a morphism in $\cat{MonCat_{(f)j}}$ that is conservative on central idempotents: if $F(u)=1$, then the central idempotent $u$ itself must be $1$.
\end{definition}

If $\cat{C}$ and $\cat{D}$ are toposes, the previous definition means that $F$ reflects all isomorphisms~\cite[Definition~3.3.2]{breiner:logicalschemes}. 
If $\cat{C}$ and $\cat{D}$ are the frames of opens of topological spaces, it is equivalent to the continuous function preserving focal points. 

\begin{proposition}
  Let $P \colon \mathcal{O}(X)\op \to \cat{MonCat}$ and $Q \colon \mathcal{O}(Y)\op \to \cat{MonCat}$ be sheaves of \changed{($\vee$- or $\bigvee$-)}local categories, let $\varphi \colon X \to Y$ be a continuous function, and let $F \colon Q \Rightarrow \varphi_* P$ be a natural transformation.
  Condition~\eqref{eq:localfunctor} holds if and only if the induced functors $F_x$ on stalks are morphisms of \changed{($\vee$- or $\bigvee$-)}local monoidal categories.
\end{proposition}
\begin{proof}
  It suffices to see that if $F \colon \cat{C} \to \cat{D}$ is a morphism in $\cat{MonCat_{(f)j}}$ such that each $F_x$ is a morphism of \changed{($\vee$- or $\bigvee$-)}local monoidal categories, $F_1(u) \in x \iff u \in \varphi(x)$.

  For a central idempotent $u$ in $\cat{C}$, observe that $F_1(u) \in x$ if and only if $F_1(u) \sim_x 1$, if and only if $F_x\big( [u]_{\varphi(x)} \big)=1$. Definition~\ref{def:localfunctor} makes this equivalent to $[u]_{\varphi(x)} = 1$. But that holds exactly when $u \in \varphi(x)$.
\end{proof}

\section{Embedding}\label{sec:embedding}

The goal of this setion is to prove that any small stiff monoidal category can be freely completed with universal (finite) joins of central idempotents.
The guiding idea will be that objects of the completion of $\cat{C}$ with universal (finite) joins can be thought of as formal colimits of the appropriate kind. 

\begin{definition}
  A morphism $f \colon A \to B$ in a monoidal category \emph{restricts} to a central idempotent $u$ if it factors through $B \otimes u$ via some $g \colon A \to U$.
  \[
    \begin{tikzcd}[row sep=5mm]
      A \ar[rr,"f"] \ar[ddrr,dashed,swap,"g"] & &  B \\ ~\\
      && B \otimes U \ar[uu,swap,"\rho_B \circ (B \otimes u)"]\\ 
    \end{tikzcd}
    \qquad\qquad
    \begin{pic}
      \node[morphism,width=8mm] (f) at (0,0) {$f$};
      \draw (f.north) to ++(0,.5) node[above]{$B$};
      \draw (f.south) to ++(0,-.5) node[below]{$A$};
    \end{pic}
    \quad=\quad
    \begin{pic}
      \node[morphism,width=8mm] (f) at (0,0) {$g$};
      \node[dot] (d) at ([yshift=5mm]f.north east) {};
      \draw (f.north west) to ++(0,.5) node[above]{$B$};
      \draw (f.north east) to node[right=0mm]{$U$} (d);
      \draw (f.south) to ++(0,-.5) node[below]{$A$};
    \end{pic}
  \]
\end{definition}

It is clear that if $u \leq v$ are central idempotents, then any morphisms $f \colon A \otimes U \to B$ restricts to $v$.

\begin{proposition}
  For a small stiff monoidal category $\cat{C}$, there is a category $D[\cat{C}]$:
  \begin{itemize}
    \item \emph{objects} are pairs $\tuple{D,A}$ of a down-closed $D \subseteq \ZI(\cat{C})$ and an object $A \in \cat{C}$;
    \item \emph{morphisms} $\tuple{D,A} \to \tuple{E,B}$ in $D[\cat{C}]$ are families $\{\eta_u \colon A \otimes U \to B\}_{u \in D}$ of morphisms in $\cat{C}$ such that each $\eta_u$ restricts to some $v \in E$, and if $u \leq u'$:~\footnote{That is, $\eta$ is a compatible family for the presheaf $\cat{C}\restrict{(-)}\colon D\op \to \cat{MonCat}$.}
    \[\begin{pic}
      \node[morphism,width=8mm] (f) at (0,0) {$\eta_{u}$};
      \draw (f.north) to ++(0,.3) node[above]{$B$};
      \draw (f.south west) to ++(0,-.8) node[below]{$A$};
      \draw (f.south east) to ++(0,-.8) node[below]{$U$};
    \end{pic}
    \quad=\quad
    \begin{pic}
      \node[morphism,width=8mm] (f) at (0,0) {$\eta_{u'}$};
      \node[dot] (d) at ([yshift=-5mm]f.south east) {};
      \draw (f.north) to ++(0,.3) node[above]{$B$};
      \draw (f.south east) to node[right=-.7mm]{$U'$} (d);
      \draw (d) to ++(0,-.3) node[below]{$U$};
      \draw (f.south west) to ++(0,-.8) node[below]{$A$};
    \end{pic}\]
    \item the \emph{identity} on $\tuple{D,A}$ has components $A \otimes u \colon A \otimes U \to D$;
    \item Composition of $\eta \colon \tuple{D,A} \to \tuple{E,B}$ and $\zeta \colon \tuple{E,B} \to \tuple{F,C}$ is given by $(\zeta \widehat{\circ} \eta)_u = \zeta_v \circ \tilde{\eta}_{u,v}$ for (any) $v \in E$ and $\tilde{\eta}_{u,v} \colon U\otimes A \to V\otimes B$ satisfying $\eta_u = (B \otimes v) \circ \tilde{\eta}_{u,v}$.
    \[\begin{tikzcd}[column sep=20mm, row sep=1cm]
      {A \otimes U} & B \\
      C & {B \otimes V}
      \arrow["{\eta_u}", from=1-1, to=1-2]
      \arrow["{\zeta_v}", from=2-2, to=2-1]
      \arrow["{B \otimes v}"', from=2-2, to=1-2]
      \arrow["{\tilde{\eta}_{u,v}}"', from=1-1, to=2-2]
      \arrow["{(\zeta\widehat{\circ}\eta)_u}"', from=1-1, to=2-1]
    \end{tikzcd}\]
  \end{itemize}
\end{proposition}
\begin{proof}
  We prove that the composition is independent of the choice of $v$ and $\tilde{\eta}_{u,v}$. Suppose that $v, v' \in E$ satisfy $v \leq v'$, and that $\eta_u$ restricts to $v$ via $\tilde{\eta} \colon A \otimes U \to B \otimes V$, and hence also to $v'$ via a morphism $\tilde{\eta}' = (B \otimes m_{v,v'}) \circ \tilde{\eta}$. Then:
  \begin{equation}\label{eq:welldefinedcomposition}\tag{$*$}
  \begin{tikzcd}[row sep=1cm, column sep=12mm]
    & {A \otimes U} && {B \otimes V'} \\
    { B} & {B \otimes V} && C
    \arrow["{\tilde{\eta}'}", from=1-2, to=1-4]
    \arrow[from=2-2, to=1-4]
    \arrow["{\zeta_v}"', from=2-2, to=2-4]
    \arrow["{\eta_u}"', from=1-2, to=2-1]
    \arrow["{\zeta_{v'}}", from=1-4, to=2-4]
    \arrow["{\tilde{\eta}}", from=1-2, to=2-2]
    \arrow["{B \otimes v'}", from=2-2, to=2-1]
  \end{tikzcd}\end{equation}
  Now suppose $\eta_u$ restricts to $v_1$ and $v_2$ in $E$ via $\tilde{\eta}_i \colon A \otimes U \to B \otimes V_i$. By stiffness, it then  restricts to $v_1 \wedge v_2$ via $\tilde{\eta} \colon A \otimes U \to B \otimes V_1 \otimes V_2$:
  \[\begin{tikzcd}[row sep=5mm]
    {B \otimes V_1 \otimes V_2} && {B \otimes V_1} \\
    & {A \otimes U} \\
    {B \otimes V_2} && Y
    \arrow["{\tilde{\eta}}", from=2-2, to=1-1]
    \arrow["{B \otimes V_1 \otimes v_2}", from=1-1, to=1-3]
    \arrow["{B \otimes v_2}"', from=3-1, to=3-3]
    \arrow["{B \otimes v_1}", from=1-3, to=3-3]
    \arrow["{B \otimes v_1 \otimes V_2}"', from=1-1, to=3-1]
    \arrow["{\tilde{\eta}_2}", from=2-2, to=3-1]
    \arrow["{\tilde{\eta}_1}"', from=2-2, to=1-3]
  \end{tikzcd}\]
  Applying~\eqref{eq:welldefinedcomposition} twice, we conclude that
  $\zeta_{v_1} \circ \tilde{\eta}_1
  =
  \zeta_{v_1\wedge v_2} \circ \tilde{\eta}
  =
  \zeta_{v_2} \circ \tilde{\eta}_1
  $, so $\zeta \widehat{\circ} \eta$ is well defined.
  It is routine to verify that the composition is associative and satisfies the identity laws.
\end{proof}

\begin{proposition}
  If $\cat{C}$ is a small stiff monoidal category, then $D[\cat{C}]$ is monoidal:
  \begin{itemize}
    \item the tensor product of objects is $\tuple{D_1,A_1} \widehat{\otimes} \tuple{D_2,A_2} = \tuple{D_1 \cap D_2, A_1 \otimes A_2}$;
    \item the tensor unit is $\widehat{I} = \tuple{\ZI(\cat{C},I}$;
    \item the tensor product of morphisms is:
    \[
      (\eta_1 \widehat{\otimes} \eta_2)_u
      = 
      \begin{pic}
      \node[morphism,width=7mm] (f) at (0,0) {$\eta_{1,u}$};
      \node[morphism,width=7mm] (g) at (1.5,0) {$\eta_{2,u}$};
      \node[dot] (d) at (1.5,-.7) {};
      \draw (f.north) to ++(0,.4);
      \draw (g.north) to ++(0,.4);
      \draw (g.south east) to[out=-90,in=0] (d);
      \draw (d) to (1.5,-1) node[below]{$U$};
      \draw (f.south west) to ++(0,-.8);
      \draw (g.south west) to[out=-90,in=90] ++(-.6,-.8);
      \draw[halo] (f.south east) to[out=-90,in=180] (d);
      \node[dot] (d) at (1.5,-.7) {};
      \end{pic}
    \]
  \end{itemize}
\end{proposition}
\begin{proof}
  To see that the tensor product of morphisms is well-defined, two conditions must be verified. The first holds because $\cat{C}\restrict{u\leq u'}$ is strict monoidal by Lemma~\ref{lem:restrictionfunctor}.
  For the second, if $\eta_i$ restricts to $v_i \in E_i$, then $(\eta_1 \widehat{\otimes} \eta_2)_u$ restricts to $v_1 \wedge v_2$, which is in $E_1 \cap E_2$ because these sets are down-closed.
  The associator and unitors, and the coherence conditions these satisfy, now follow routinely from those in $\cat{C}$.
\end{proof}

Notice that if $A \simeq B$ in $\cat{C}$, then $\tuple{D,A} \simeq \tuple{D,B}$ in $D[\cat{C}]$ for any $D$.

\begin{lemma}\label{lem:embedding}
  There is a strictly monoidal full embedding $\cat{C} \to D[\cat{C}]$ that sends an object $A$ to $\widehat{A}=\tuple{\ZI(\cat{C},A}$ and a morphism $f$ to $\{f \otimes u\}_{u \in \ZI(\cat{C})}$.
\end{lemma}
\begin{proof}
  This is clearly a functor that is injective on objects. 
  If $E \subseteq \ZI(\cat{C})$ is down-closed and $u \in E$, then $D[\cat{C}]\big( \tuple{\downset u,A}, \tuple(E,B) \big) \simeq \cat{C}(A \otimes U, Y) \simeq \cat{C}\restrict{u}(A,B)$. It follows that $D[\cat{C}](\widehat{A},\widehat{B}) \simeq \cat{C}(A,B)$, so the functor is full.
  Finally, observe that $\widehat{I}$ is the monoidal unit in $D[\cat{C}]$, and that $\widehat{A \otimes B}=\widehat{A} \otimes \widehat{B}$.
\end{proof}

\begin{proposition}\label{prop:ZIDfreeframe}
  If $\cat{C}$ is a small stiff monoidal category, there is an isomorphism of partially ordered sets $\ZI(D[\cat{C}]) \simeq \{ D \subseteq \ZI(\cat{C}) \mid D = \downset D \}$ that sends a down-closed $D \subseteq \ZI(\cat{C})$ to the morphism $\tuple{D,I} \to \widehat{I}$ with component $u \circ \lambda_I$ at $u \in D$.
\end{proposition}
\begin{proof}
  For down-closed $D \subseteq \ZI(\cat{C})$, define $\widehat{D} = \tuple{D,I}$, and write $d \colon \widehat{D} \to \widehat{I}$ for the morphism in $D[\cat{C}]$ given by $d_u = u \circ \lambda_I$ for each $u \in D$.

  We first prove that $d$ is a central idempotent in $D[\cat{C}]$. 
  Note that $\widehat{D} \widehat{\otimes} \widehat{D} = \tuple{D,I \otimes I}$. Now $d_u = u \circ \lambda_I = \lambda_I \circ (I \otimes u) = \rho_I \circ (I \otimes u) = \widehat{D}_u$, so the components of $d$ and the identity on $\widehat{D}$ are the same morphisms in $\cat{C}$. Hence $d \widehat{\otimes} \widehat{D} = \widehat{D} \widehat{\otimes} d$, and this has an inverse with components $I \otimes u \colon I \otimes U \to I \otimes I$.

  The assignment $D \mapsto d$ is a semilattice morphism, because $\widehat{D} \widehat{\otimes} \widehat{E} \simeq \widehat{D \cap E}$ via an isomorphism with components $\rho_U \circ (u \otimes \lambda_I) = \rho_u \circ (u \otimes \rho_I) \colon U \otimes I \otimes I \to I$.

  Moreover, if $d \leq e$ in $D[\cat{C}]$ then $d$ factors through $e$ via a map $\widehat{D} \to \widehat{E}$. Its component at $u \in D$ must be $u \colon U \to I$ factoring via $v \colon V \to I$ for some $v \in E$. So any $u \in D$ is below some $v \in E$, so $D \subseteq E$. Thus the map $D \mapsto d$ is injective. It remains to show that it is surjective.

  Let $\eta \colon \tuple{D,A} \to \widehat{I}$ be a central idempotent in $D[\cat{C}]$. It suffices to prove that each $\eta_u \colon A \otimes U \to I$ is a central idempotent in $\cat{C}$, for then the maps $\eta_u$ themselves evidently restrict to $\eta_u$ and form an isomorphism $\tuple{D,A} \to \widehat{\{\eta_u\}_{u \in D}}$.
    
  From $(\tuple{D,A} \widehat{\otimes} \eta)_u = (\eta \widehat{\otimes} \tuple{D,A})_u$, it follows that
  \[
    \begin{pic}
      \node[morphism,width=7mm] (eta) at (0,0) {$\eta_u$};
      \node[dot] (d) at (0,-.7) {};
      \draw (d) to ++(0,-.3) node[below]{$U$};
      \draw (d) to[out=0,in=-90] ([xshift=1mm]eta.south east);
      \draw (eta.south west) to[out=-90,in=90] ++(-.6,-.8) node[below]{$A$};
      \draw[halo] (d) to[out=180,in=-90] ([xshift=-10mm]eta) node[dot]{};
      \draw (-1.6,-1) node[below]{$A$} to (-1.6,.5);
    \end{pic}
    \quad=\quad
    \begin{pic}
      \node[morphism,width=7mm] (eta) at (0,0) {$\eta_u$};
      \node[dot] (d) at (1,-.7) {};
      \draw (.3,-1) node[below]{$A$} to[out=90,in=-90] (1,.5);
      \draw (d) to ++(0,-.3) node[below]{$U$};
      \draw[halo] (d) to[out=180,in=-90] ([xshift=1mm]eta.south east);
      \draw (eta.south west) to ++(0,-.8) node[below]{$A$};
      \draw (d) to[out=0,in=-90,looseness=.8] ++(.3,.7) node[dot]{};
    \end{pic}
  \]
  and upon simplification and tensoring with $U$, we find
  $\eta \otimes A \otimes U = A \otimes U \otimes \eta$.

  Invertibility of $\tuple{D,A} \widehat{\otimes} \eta$ gives a map $\zeta \colon \tuple{D,A} \to \tuple{D, A \otimes A}$ 
  such that:
  \[
    \begin{pic}
      \node[width=7mm,morphism] (f) at (0,0) {$\zeta_u$};
      \draw (f.north west) to ++(0,.4) node[above]{$A$};
      \draw (f.north east) to ++(0,.4) node[above]{$A$};
      \node[morphism,width=7mm] (g) at (1.5,0) {$\eta_{u}$};
      \node[dot] (d) at (1.5,-.7) {};
      \draw (g.south east) to[out=-90,in=0] (d);
      \draw (d) to (1.5,-1) node[below]{$U$};
      \draw (f.south west) to ++(0,-.8) node[below]{$A$};
      \draw (g.south west) to[out=-90,in=90] ++(-.6,-.8) node[below]{$A$};
      \draw[halo] (f.south east) to[out=-90,in=180] (d);
    \end{pic}
    \;=
    \begin{pic}
      \draw (0,0) node[below]{$A$} to ++(0,1.5) node[above]{$A$};
      \draw (.4,0) node[below]{$A$} to ++(0,1.5) node[above]{$A$};
      \draw (.8,0) node[below]{$U$} to ++(0,1) node[dot]{};
    \end{pic}
    \qquad\qquad\qquad
    \begin{pic}
      \node[morphism,width=7mm] (f) at (0,0) {$\zeta_u$};
      \node[morphism,width=7mm,anchor=south west] (g) at ([yshift=2mm]f.north east) {$\eta_{u}$};
      \node[dot] (d) at (.6,-.5) {};
      \draw (f.north west) to ++(0,.6) node[above]{$A$};
      \draw (g.south west)to[out=-90,in=90] (f.north east); 
      \draw ([xshift=1mm]g.south east) to[out=-90,in=0,looseness=.8] (d);
      \draw (d) to ++(0,-.3) node[below]{$U$};
      \draw (f.south west) to ++(0,-.6) node[below]{$A$};
      \draw[halo] (f.south east) to[out=-90,in=180] (d);
    \end{pic}
    \;=\;
    \begin{pic}
      \draw (0,0) node[below]{$A$} to ++(0,1.5) node[above]{$A$};
      \draw (.4,0) node[below]{$U$} to ++(0,1) node[dot]{};
    \end{pic}
  \]
  Precomposing with $(U \otimes u)^{-1}$ (twice on the left, once on the right) and simplifying shows that
  \[\begin{pic}
    \node[width=7mm,morphism] (f) at (0,0) {$\zeta_u$};
    \node[dot] (d) at (.6,-.5) {};
    \node[dot] (d2) at (1.1,0) {};
    \draw (f.north west) to ++(0,.4) node[above]{$A$};
    \draw (f.north east) to ++(0,.4) node[above]{$A$};
    \draw (f.south west) to ++(0,-.5) node[below]{$A$};
    \draw (d) to ++(0,-.2) node[below]{$U$};
    \draw (f.south east) to[out=-90,in=180] (d);
    \draw (d) to[out=0,in=-90] (d2);
    \draw (d2) to[out=180,in=-90,looseness=.8] ++(-0.3,.6) node[above]{$U$};
    \draw (d2) to[out=0,in=-90,looseness=.8] ++(+0.3,.6) node[above]{$U$};
  \end{pic}\]
  inverts $A \otimes U \otimes \eta_u$, and so $\eta_u$ is a central idempotent in $\cat{C}$.
\end{proof}

\begin{lemma}\label{lem:DChasjoins}
  If $\cat{C}$ is a small stiff monoidal category, then $D[\cat{C}]$ has universal joins of central idempotents.
\end{lemma}
\begin{proof}
  It follows from Proposition~\ref{prop:ZIDfreeframe} that $\ZI(D[\cat{C}])$ is the free frame on the semilattice $\ZI(\cat{C})$, so it certainly has arbitrary suprema.

  Let $\{D_i\}$ be a family of down-closed subsets of $\ZI(\cat{C})$ closed under intersection, so $\{\widehat{D_i}\}$ is a family of central idempotents of $D[\cat{C}]$ closed under meets. Let $\tuple{E,B}$ be an object in $D[\cat{C}]$. 
  Consider the diagram in $D[\cat{C}]$ consisting of the objects 
  $\tuple{E,B} \widehat{\otimes} \widehat{D_i} \simeq \tuple{D_i \cap E, B}$ and morphisms
  $\tuple{E,B} \widehat{\otimes} \mu_{D_i,D_j} \colon \tuple{E,B} \widehat{\otimes} \widehat{D_i}  \to \tuple{E,B} \widehat{\otimes} \widehat{D_j}$ whenever $D_i \subset D_j$. The components of these maps are just $B \otimes u \colon B \otimes U \to B$ for each $u \in D_i \cap E$, and evidently restrict to $u \in D_i \cap E \subset D_j \cap E$ itself.

  We show that
  $
    \tuple{E,B} \widehat{\otimes} \bigvee \widehat{D_i}
    =
    \tuple{E,B} \widehat{\otimes} \widehat{\bigcup D_i} 
    \simeq
    \tuple{D_i \cap E, B}
  $
  with cocone maps
  \[
    \tuple{E,B} \widehat{\otimes} \mu_{D_i,\bigcup D_i}
    \colon
    \tuple{E,B} \widehat{\otimes} \widehat{D_i}
    \to
    \tuple{E,B} \widehat{\otimes} \widehat{\bigcup D_i}
  \]
  is the colimit of this diagram. 
  Suppose that there were another cocone with vertex $\tuple{F,C}$ and morphisms
  $\zeta_i \colon \tuple{E \cap D_i,B} \to \tuple{F,C}$.
  A mediating morphism $\eta \colon \tuple{E\cap\bigcup D_i, B} \to \tuple{F,C}$ must satisfy
  $\zeta_i = \eta \,\widehat{\circ}\, (\tuple{E,B} \widehat{\otimes} \mu_{D_i,\bigcup D_i})$.
  Componentwise, by the definition of composition in $D[\cat{C}]$, this yields for any $u \in E \cap D_i$:
  \[\begin{tikzcd}[column sep=10mm, row sep=10mm]
    {Y \otimes U} & U \\
    Z & {Y \otimes U}
    \arrow["{\mu_{D_i,\bigcup D_i}}", from=1-1, to=1-2]
    \arrow["{\eta_u}", from=2-2, to=2-1]
    \arrow["{Y \otimes u}"', from=2-2, to=1-2]
    \arrow["{\zeta_{i,u}}"', from=1-1, to=2-1]
    \arrow[Rightarrow, no head, from=1-1, to=2-2]
  \end{tikzcd}\]
  Hence, $u \in E \cap \bigcup D_i$ forces $\eta_u = \zeta_{i,u}$ for any $i$ with $u \in E \cap D_i$. Hence mediating morphisms is unique. It remains to show that this mediating morphism is well defined.

  First, we show that the definition of $\nu_u$ does not depend on $i$.
  Suppose that $u \in E \cap D_i$ and $u \in E \cap D_j$. Then $u \in E \cap (D_i\cap D_j)$, and we write $i \wedge j$ for the index such that $D_{i \wedge j} = D_i \cap D_j$.
  By the fact that $\{\zeta_i\}$ is a cocone, $\zeta_{i \wedge j} = \zeta_i \widehat{\circ} (\tuple{E,B} \widehat{\otimes} \mu_{i \wedge j,i})$.
  Since the component of $\mu_{i\wedge j,i}$ at $u$ is $Y \otimes u = (B \otimes u) \circ (B \otimes U)$,
  we get $\zeta_{i \wedge j,u} = \zeta_{i,u} \circ (B \otimes U) = \zeta_{i,u}$. Consequently, $\zeta_{i,u} = \zeta_{i \wedge j,u} = \zeta_{j,u}$ as required.

  Next, we check that $\eta$ is indeed a morphism in $D[\cat{C}]$.
  Given $u \leq u' \in E \cap \bigcup D_i$, the restriction condition
  $\eta_u = \eta_{u'} \circ m_{u,u'}$ follows from the corresponding condition for $\zeta_i$ for any $i$ such that $u' \in E \cap D_i$ (and thus also $u \in E \cap D_i$).
  Finally, each $\eta_u$ evidently restricts to a central idempotent in $F$ since all $\zeta_{i,u}$ do.  
\end{proof}

An analogous construction can be carried out for finite joins of central idempotents, by considering only finitely generated downsets, which form free distributive lattice on a meet-semilattice. This proves the Corollary from the introduction.

\begin{corollary}\label{cor:embedding}
  Any small stiff monoidal category allows a central idempotent-preserving monoidal embedding into a category of global sections of a sheaf of \changed{($\vee$- or $\bigvee$-)}local monoidal categories.
\end{corollary}
\begin{proof}
  Combine Theorem~\ref{thm:main} with Lemmas~\ref{lem:embedding} and~\ref{lem:DChasjoins}.
\end{proof}

\changed{The previous corollary generalises the same results for (pre)toposes~\cite{breiner:logicalschemes,awodey:sheafrepresentations,awodey:sheafrepresentationsinlogic}.}

Finally, we show that $D[\cat{C}]$ is in fact the free such completion of $\cat{C}$.

\begin{theorem}
  If $\cat{C}$ is a small stiff monoidal category, $\cat{D}$ is a monoidal category with universal joins of central idempotents, and $F \colon \cat{C} \to \cat{D}$ a morphism in $\cat{MonCat_s}$, then there is a unique morphism $\widehat{F} \colon D[\cat{C}] \to \cat{D}$ in $\cat{MonCat_j}$ with: 
  \[\begin{tikzcd}[row sep=2mm]
    \cat{C} \ar[rr,"\widehat{(-)}"] \ar[ddrr,swap,"F"]
    & &  D[\cat{C}] \ar[dd,dashed,"\widehat{F}"] \\ ~\\
    & & \cat{D} \\ 
  \end{tikzcd}\]
\end{theorem}
\begin{proof}
  Let $\tuple{D,A}$ be an object in $D[\cat{C}]$.
  Consider the diagram in $\cat{D}$ with objects $F(U) \otimes F(A)$ for $u \in D$ and morphisms whenever $u \leq u' \in D$.
  Note that $F(U) \to F(U')$ is the morphism in $\cat{D}$ witnessing $F(u) \leq F(u')$ as central idempotents in $\cat{D}$.
  Because $D$ is closed under meets, and $F$ induces a semilattice homomorphism $\ZI(\cat{C}) \to \ZI(\cat{D})$, also $\{F(u) \mid u \in D\}$ is closed under meets.
  Because $\cat{D}$ has universal joins of central idempotents, the diagram under consideration has a colimit, given by the image of $\tuple{D,A}$ under $\widehat{F}$:
  \[
    \widehat{F}(\tuple{D,A}) = \colim_{u \in D} F(A) \otimes F(U)
  \]
  Now, let $\eta \colon \tuple{D,A} \to \tuple{E,B}$ be a morphism in $D[\cat{C}]$.
  Writing $\theta$ for the coherence morphism witnessing that $F$ is monoidal, the morphisms
  \[
    F(A) \otimes F(U)
    \stackrel{\theta_{A,U}}{\longrightarrow}
    F(A \otimes U)
    \stackrel{F(\eta_u)}{\longrightarrow}
    Y
  \]
  form a cocone for the diagram in $\cat{D}$ with objects $\{F(A) \otimes F(U)\}_{u \in D}$ and morphisms induced by $u \leq u'$, since the following diagram commutes:
  \[
  \begin{tikzcd}
    F(A) \otimes F(U') \ar[rr,"\theta_{A,U'}"]
    & & F(A \otimes U') \ar[rr,"F(\eta_{u'})"]
    & & F(B)
    \\~\\
    F(A) \otimes F(U)  \ar[rr,swap,"\theta_{A,U}"] \ar[uu]
    & & F(A \otimes U) \ar[uu] \ar[uurr,swap,"F(\eta_{u})"]
    & & \\ 
  \end{tikzcd}
  \]
  This gives a mediating morphism $h \colon \widehat{F}(\tuple{D,A}) \to F(B)$ in $\cat{D}$.

  Finally, each $\eta_u$ restricts to some $v \in E$:
  \[\begin{tikzcd}[row sep=5mm]
    A \otimes U \ar[rr,"\eta_u"] \ar[dr,dashed,swap,"\tilde{\eta}"] &&  B \\
    & B \otimes V \ar[ur,swap,"B \otimes v"] 
  \end{tikzcd}\]
  Hence the morphisms $F(\eta_u) \circ \theta_{A,U}$ in the cocone restrict to $F(v)$:
  \[\begin{tikzcd}[row sep=6mm, column sep=6mm]
    F(A) \otimes F(U) \ar[rr,"\theta_{A,U}"] && F(A \otimes U) \ar[rrrr, bend left,"F(\eta_u)"] \ar[rr,"F(\tilde{\eta})"] \ar[ddrr,dashed] & &  F(B \otimes V) \ar[rr,"F(B \otimes v)"] & & F(Y)
    \\ ~\\
    & & & & F(B) \otimes F(V) \ar[uu,"\theta_{B,V}"] \ar[rr,swap,"F(B) \otimes  F(v)"] && F(B) \otimes F(I) \ar[uu,swap,"\simeq"]\ar[uu,"F(B) \otimes \theta_{I}^{-1}"]
  \end{tikzcd}\]
  Consequently, the mediating morphism $h$ factors through the colimit of the diagram in $\cat{D}$ of objects $\{F(B) \otimes F(V)\}_{v \in E}$ and morphisms induced by $v \leq v' \in E$, that is, through $\widehat{F}(\tuple{E,B})$. We define $\widehat{F}(\eta)$ to be this factor $\widehat{F}(\tuple{D,A}) \to \tilde{F}(\tuple{E,B})$.

  It is routine to verify that this is functorial. Because
  \[
    \widehat{F}\left(\bigvee_i \overline{D_i}\right)
    =
    \widehat{F}\left(\overline{\bigcup_i D_i}\right)
    =
    \colim_{u\in \bigcup_i D_i}F(U)
    =
    \bigvee_i \colim_{u \in D_i} F(U)
    =
    \bigvee \widehat{F}(\overline{D_i}) \, .
  \]
  it preserves joins of central idempotents.
\end{proof}

In other words, $\cat{C} \mapsto D[\cat{C}]$ is functorial and left adjoint to the forgetful functor $\cat{MonCat_j} \to \cat{MonCat_s}$. Similarly, the forgetful functor $\cat{MonCat_{fj}} \to \cat{MonCat_s}$ has a left adjoint.

This generalises and improves~\cite[Theorem~10.4]{enriquemolinerheunentull:tensortopology} in two ways.
First, that result is based on subunits (see Appendix~\ref{sec:subunits}).
Second, it encodes objects of $D[\cat{C}]$ as presheaves on $\cat{C}$. 
Both lead to somewhat ad-hoc steps in the proof that do not generalise to central idempotents easily (namely Lemma 8.3 and Proposition 9.6 of ~\cite{enriquemolinerheunentull:tensortopology}). The results of this section are more general and conceptually cleaner.

\section{Further work}\label{sec:conclusion}

Theorem~\ref{thm:main} and Lemma~\ref{lem:stalksclosed} do for multiplicative linear logic what earlier results~\cite{awodey:sheafrepresentations} do for higher-order intuitionistic logic. This raises several questions.
\begin{itemize}
  \item Does our construction preserve additive connectives too, so that our main result can be extended from the multiplicative fragment to full linear logic? We expect that weakly distributive categories~\cite{cockettseely:weaklydistributive} are global sections of sheaves of \changed{($\vee$- or $\bigvee$-)}local weakly distributive categories. 

  \item The representation result for toposes is often used to obtain logical completeness theorems~\cite{lambekscott:categoricallogic,awodey:sheafrepresentations}. Can we derive a similar completeness result for (multiplicative) linear logic? Concrete questions towards this aim include: if a morphism is stalkwise epi/mono/iso, must it be epi/mono/iso? Similarly, one could study converses to the results in Section~\ref{sec:preservation}, for example: if all stalks $\cat{C}\restrict{x}$ are closed, must $\cat{C}$ itself be closed?

  \item Completeness theorems for linear logic give rise to coherence proofs for symmetric monoidal closed categories~\cite{soloviev:coherence}. Does a coherence theorem follow from our construction?

  \item What does a central idempotent signify in a model of linear logic, and when is the model \changed{($\vee$- or $\bigvee$-)}local? Example~\ref{example:linearnonlinear} gives one first step. For another, consider the model of linear logic given by a \emph{Petri net}~\cite{engbergwinskel:linearlogicpetrinets}. Every Petri net generates a monoidal category~\cite{baezgenovesemastershulman:nets} whose objects are markings and whose morphisms are firing sequences. Central idempotents then correspond to full subnets with one place and two transitions, where both transitions have no other incoming or outgoing edges, and the place is marked by either none or a nonzero number of tokens.
  \tikzstyle{place}=[circle,draw,minimum size=5mm,outer sep=1mm]
  \tikzstyle{transition}=[rectangle,draw,minimum size=5mm,outer sep=1mm]    
  \[\begin{pic}[xscale=2]
    \node[transition] (l) at (0,0) {};
    \node[place] (m) at (1,0) {};
    \node[transition] (r) at (2,0) {};
    \draw[->] (m) to (r);
    \draw[->] ([yshift=2mm]m.west) to ([yshift=2mm]l.east);
    \draw[<-] (m.west) to (l.east);
    \draw[<-] ([yshift=-2mm]m.west) to ([yshift=-2mm]l.east);
  \end{pic}\]
\end{itemize}
Additionally, the structure of the sheaf representation raises several questions.
\begin{itemize}
  \changed{\item In the case of toposes, the sheaf representation has the additional property that the tensor unit becomes projective in the stalks~\cite{awodey:sheafrepresentations,awodey:sheafrepresentationsinlogic}. That requires taking Henkin models (coherent functors to $\cat{Set}$) instead of the completely prime spectrum as the base space. Is there an analogue for monoidal categories that removes the dependence on spatiality in Theorem~\ref{thm:main2}?}

  \item The stalks $\cat{C}\restrict{x}$ are closely related to certain restriction categories called tensor-restriction categories~\cite{heunenlemay:tensorrestriction}. Is there a general representation theorem for restriction categories as categories of global sections?

  \item Corollary~\ref{cor:embedding} first embedded a semilattice of central idempotents into a distributive lattice (or frame) to obtain a sheaf representation. But a spectrum can be defined directly for a semilattice~\cite{bezhanishvilijansana:priestleydualitysemilattices}. Does the representation Theorem~\ref{thm:main} extend to a sheaf on the spectrum of the semilattice of central idempotents of a stiff monoidal category without universal joins directly?

  \item The prime ideal (Zariski) spectrum of the central idempotents of a monoidal category is analogous to the Pierce spectrum, which is the prime ideal spectrum of the central idempotents, of a ring~\cite{johnstone:stonespaces}. Is there an analogon for monoidal categories of the prime spectrum of a ring?

  \item \changed{The sheaf representation theorem trivialises when applied to the topos of actions of a monoid: there are only two central idempotents. However, that topos carries other interesting topologies involving idempotent ideals of the monoid~\cite{hemelaerrogers:msets,pirashvili:msets}. This seems related to Example~\ref{ex:modules}. Is there a refinement of the sheaf representation theorem that takes this into account?}
\end{itemize}
Finally, there are many applications of localisation of rings in algebra, algebraic geometry, and tensor-triangular geometry~\cite{balmer:spectrum,balmerfavi:telescope,brandenburg:localizations,boyarchenkodrinfeld:idempotent}. 
We have not yet explored such applications of the representation theorem.
For example, Theorem~\ref{thm:main} may carry generalisations of simplicity assumptions in linear categories~\cite{kuperberg:linearcategories}, and Corollary~\ref{cor:embedding} may let us drop simplicity assumptions from embedding and characterisation theorems~\cite{heunen:embedding,heunenkornell:axioms}.

\bibliographystyle{plain}	
\bibliography{bibliography}

\begin{thebibliography}{10}

\bibitem{aneljoyal:topologie}
M.~Anel and A.~Joyal.
\newblock Topo-logie.
\newblock In M.~Anel and G.~Catren, editors, {\em New spaces in mathematics,
  formal and conceptual reflections}, pages 155--257. Cambridge University
  Press, 2021.

\bibitem{awodey:sheafrepresentations}
S.~Awodey.
\newblock Sheaf representation for topoi.
\newblock {\em Journal of Pure and Applied Algebra}, 145:107--121, 2000.

\bibitem{awodey:sheafrepresentationsinlogic}
S.~Awodey.
\newblock Sheaf representations in duality and logic.
\newblock In {\em Joachim Lambek: the interplay of mathematics, logic, and
  linguistics}, pages 39--57. Springer, 2021.

\bibitem{baez:2hilbertspaces}
J.~C. Baez.
\newblock Higher-dimensional algebra {II}: 2-{H}ilbert spaces.
\newblock {\em Advances in Mathematics}, 127:125--189, 1997.

\bibitem{baezgenovesemastershulman:nets}
J.~C. Baez, F.~Genovese, J.~Master, and M.~Shulman.
\newblock Categories of nets.
\newblock arXiv:2101.04238 [math.CT], 2021.

\bibitem{balmer:spectrum}
P.~Balmer.
\newblock The spectrum of prime ideals in tensor triangulated categories.
\newblock {\em Journal f{\"u}r die {R}eine und {A}ngewandte {M}athematik},
  588:149--168, 2005.

\bibitem{balmerfavi:telescope}
P.~Balmer and G.~Favi.
\newblock Generalized tensor idempotents and the telescope conjecture.
\newblock {\em Proceedings of the London Mathematical Society},
  102(6):1161--1185, 2011.

\bibitem{benton:linearnonlinear}
N.~Benton.
\newblock A mixed linear and non-linear logic: Proofs, terms and models.
\newblock In {\em Computer Science Logic}, volume 933 of {\em Lecture Notes in
  Computer Science}, pages 121--135. Springer, 1994.

\bibitem{bezhanishvilijansana:priestleydualitysemilattices}
G.~Bezhanishvili and R.~Jansana.
\newblock Priestley style duality for distributive meet-semilattices.
\newblock {\em Studia Logica}, 98:83--122, 2011.

\bibitem{borceux:2}
F.~Borceux.
\newblock {\em Handbook of Categorical Algebra 2: Categories and Structures}.
\newblock Cambridge University Press, 1994.

\bibitem{borceux:3}
F.~Borceux.
\newblock {\em Handbook of Categorical Algebra 3: Categories of Sheaves}.
\newblock Cambridge University Press, 1994.

\bibitem{borceuxvdbossche:quantalesheaves}
F.~Borceux and G.~{Van den Bossche}.
\newblock Quantales and their sheaves.
\newblock {\em Order}, 3:61--87, 1986.

\bibitem{bourbaki:generaltopology}
N.~Bourbaki.
\newblock {\em Elements of Mathematics: General topology, Chapters 1--4}.
\newblock Springer, 1987.

\bibitem{boyarchenkodrinfeld:idempotent}
M.~Boyarchenko and V.~Drinfeld.
\newblock Idempotents in monoidal categories.

\bibitem{brandenburg:localizations}
M.~Brandenburg.
\newblock Localizations of tensor categories and fiber products of schemes.
\newblock arXiv:2002.00383 [math.AG], 2020.

\bibitem{breiner:logicalschemes}
S.~Breiner.
\newblock {\em Scheme representation for first-order logic}.
\newblock PhD thesis, Carnegie Mellon University, 2014.

\bibitem{cardechi:etale}
D.~Cardechi.
\newblock An {\'e}tal{\'e} space construction for stacks.
\newblock {\em Algebraic \& Geometric Topology}, 13:831--903, 2013.

\bibitem{clarkwisbauer:idempotent}
J.~Clark and R.~Wisbauer.
\newblock Idempotent monads and $\ast$-functors.
\newblock {\em Journal of Pure and Applied Algebra}, 215(2):145--153, 2011.

\bibitem{cockettseely:weaklydistributive}
J.~R.~B. Cockett and R.~A.~G. Seely.
\newblock Weakly distributive categories.
\newblock {\em Journal of Pure and Applied Algebra}, 114:133--173, 1997.

\bibitem{constantindicaireheunen:localisablemonads}
C.~Constantin, N.~Dicaire, and C.~Heunen.
\newblock Localisable monads.
\newblock In {\em Computer Science Logic}, volume 216, pages 15:1--17. Leibniz
  International Proceedings in Informatics, 2022.

\bibitem{daunshofmann:sections}
J.~Dauns and K.~H. Hofmann.
\newblock {\em Representations of rings by sections}.
\newblock Number~83 in Memoirs. American Mathematical Society, 1968.

\bibitem{dickmannschwartztressl:spectralspaces}
M.~Dickmann, N.~Schwartz, and M.~Tressl.
\newblock {\em Spectral spaces}.
\newblock Cambridge University Press, 2019.

\bibitem{engbergwinskel:linearlogicpetrinets}
U.~H. Engberg and G.~Winskel.
\newblock Linear logic on {P}etri nets.
\newblock Technical Report RS-94-3, Basic Research in Computer Science (BRICS)
  report, February 1994.

\bibitem{enriquemolinerheunentull:space}
P.~{Enrique Moliner}, C.~Heunen, and S.~Tull.
\newblock Space in monoidal categories.
\newblock In {\em Quantum Physics and Logic}, volume 266 of {\em Electronic
  Proceedings in Theoretical Computer Science}, pages 399--410, 2017.

\bibitem{enriquemolinerheunentull:tensortopology}
P.~{Enrique Moliner}, C.~Heunen, and S.~Tull.
\newblock Tensor topology.
\newblock {\em Journal of Pure and Applied Algebra}, 224(10):106378, 2020.

\bibitem{godement:flabby}
R.~Godement.
\newblock {\em Topologie Alg{\'e}brique et Th{\'e}orie des Faisceux}.
\newblock Number 1252 in Actualit{\'e}s scientifique et industrielles. Hermann,
  1958.

\bibitem{hemelaerrogers:msets}
J.~Hemelaer and M.~Rogers.
\newblock Monoid properties as invariants of toposes of monoid actions.
\newblock {\em Applied Categorical Structures}, 29:379--413, 2021.

\bibitem{heunen:embedding}
C.~Heunen.
\newblock An embedding theorem for {H}ilbert categories.
\newblock {\em Theory and Applications of Categories}, 22(13):321--344, 2009.

\bibitem{heunenkornell:axioms}
C.~Heunen and A.~Kornell.
\newblock Axioms for the category of {H}ilbert spaces.
\newblock {\em Proceedings of the National Academy of Sciences},
  119(9):e2117024119, 2022.

\bibitem{heunenlemay:tensorrestriction}
C.~Heunen and J.~S. {Pacaud Lemay}.
\newblock Tensor-restriction categories.
\newblock {\em Theory and Application of Categories}, 37(21):635--670, 2021.

\bibitem{heunenreyes:frobenius}
C.~Heunen and M.~L. Reyes.
\newblock Frobenius structures over {H}ilbert {C}*-modules.
\newblock {\em Communications in Mathematical Physics}, 361(2):787--824, 2018.

\bibitem{heunenvicary:cqm}
C.~Heunen and J.~Vicary.
\newblock {\em Categories for quantum theory: an introduction}.
\newblock Oxford University Press, 2019.

\bibitem{heymans:quantalesheaves}
H.~Heymans.
\newblock {\em Sheaves on quantales as generalized metric spaces}.
\newblock PhD thesis, Universiteit Antwerpen, 2010.

\bibitem{hochster:primeideal}
M.~Hochster.
\newblock Prime ideal structure in commutative rings.
\newblock {\em Transactions of the American Mathematical Society}, pages
  43--60, 1968.

\bibitem{johnstone:stonespaces}
P.~T. Johnstone.
\newblock {\em Stone spaces}.
\newblock Cambridge University Press, 1982.

\bibitem{joyalstreet:yangbaxter}
A.~Joyal and R.~Street.
\newblock Tortile {Y}ang-{B}axter operators in tensor categories.
\newblock {\em Journal of Pure and Applied Algebra}, 71:43--51, 1991.

\bibitem{joyalstreetverity:traced}
A.~Joyal, R.~Street, and D.~Verity.
\newblock Traced monoidal categories.
\newblock {\em Mathematical Proceedings of the Cambridge Philosophical
  Society}, 119:447--468, 1996.

\bibitem{kapranovvoevodsky:2vectorspaces}
M.~Kapranov and V.~Voevodsky.
\newblock 2-categories and {Z}amolodchikov tetrahedra equations in algebraic
  groups and their generalization.
\newblock In {\em Proceedings of Symposia in Pure Mathematics}, volume~56,
  pages 177--259. American Mathematical Society, 1994.

\bibitem{kassel:quantumgroups}
C.~Kassel.
\newblock {\em Quantum groups}.
\newblock Springer, 1995.

\bibitem{kelly:transfinite}
G.~M. Kelly.
\newblock A unified treatment of transfinite constructions for free algebras,
  free monoids, colimits, associated shaves, and so on.
\newblock {\em Bulletin of the Australian Mathematical Society}, 22:1--83,
  1980.

\bibitem{kuperberg:linearcategories}
G.~Kuperberg.
\newblock Finite, connected, semisimple, rigid tensor categories are linear.
\newblock {\em Mathematical Research Letters}, 10:411--421, 2003.

\bibitem{lam:modules}
T.~Y. Lam.
\newblock {\em Lectures on Modules and Rings}.
\newblock Springer, 1999.

\bibitem{lambek:possibleworlds}
J.~Lambek.
\newblock On the sheaf of possible worlds.
\newblock In J.~Adamek and S.~{Mac Lane}, editors, {\em Categorical topology
  and its relation to analysis, algebra, and combinatorics}, pages 36--54.
  World Scientific, 1989.

\bibitem{lambek:world}
J.~Lambek.
\newblock What is the world of mathematics?
\newblock {\em Annals of Pure and Applied Logic}, 126(1--3):149--158, 2004.

\bibitem{lambekmoerdijk:sheafrepresentations}
J.~Lambek and I.~Moerdijk.
\newblock Two sheaf representations of elementary toposes.
\newblock {\em The {L. E. J.} {B}rouwer centenary symposium}, pages 275--295,
  1982.

\bibitem{lambekrattray:localization}
J.~Lambek and B.~A. Rattray.
\newblock Localization and duality in additive categories.
\newblock {\em Houston Journal of Mathematics}, 1(1):87--100, 1975.

\bibitem{lambekscott:categoricallogic}
J.~Lambek and P.~Scott.
\newblock {\em Introduction to higher order categorical logic}.
\newblock Cambridge University Press, 1986.

\bibitem{lance1995hilbert}
E.~C. Lance.
\newblock {\em Hilbert C*-modules: a toolkit for operator algebraists}.
\newblock Cambridge University Press, 1995.

\bibitem{lurie:highertopos}
J.~Lurie.
\newblock {\em Higher topos theory}.
\newblock Princeton University Press, 2009.

\bibitem{maclanemoerdijk:sheaves}
S.~{Mac Lane} and I.~Moerdijk.
\newblock {\em Sheaves in Geometry and Logic}.
\newblock Springer, 1992.

\bibitem{pirashvili:msets}
I.~Pirashvili.
\newblock Idempotents and the points of the topos of {$M$}-sets.
\newblock {\em arXiv:2011.11747}, 2020.

\bibitem{resende:quantalesheaves}
P.~Resende.
\newblock Groupoid sheaves as quantale sheaves.
\newblock {\em Journal of Pure and Applied Algebra}, 216:41--70, 2012.

\bibitem{rosenthal:quantales}
K.~I. Rosenthal.
\newblock {\em Quantales and their applicatoins}.
\newblock Pitman Research Notes in Mathematics. Longman Scientific \&
  Technical, 1990.

\bibitem{selinger:graphicallanguages}
P.~Selinger.
\newblock A survey of graphical languages for monoidal categories.
\newblock Number 813 in Lecture Notes in Physics, pages 289--356. Springer,
  2009.

\bibitem{soloviev:coherence}
S.~Soloviev.
\newblock Proof of a conjecture of {S}. {M}ac {L}ane.
\newblock {\em Annals of Pure and Appplied Logic}, 90:101--162, 1997.

\bibitem{takahashi:hilbertmodules}
A.~Takahashi.
\newblock Hilbert modules and their representation.
\newblock {\em Revista {C}olumbiana de matematicas}, 13:1--38, 1979.

\end{thebibliography}

\appendix
\section{Subunits}\label{sec:subunits}

In this appendix we compare central idempotents with \emph{subunits}, which are central idempotents $u \colon U \to I$ that are monic. 
There is an inclusion $\ISub(\cat{C}) \subseteq \ZI(\cat{C})$; see~\cite[Lemma~2.2]{heunenlemay:tensorrestriction}.
In a so-called \emph{firm} braided monoidal category, the subunits form a semilattice, and this inclusion is an embedding of semilattices.
However, $\ZI(\cat{C})$ does not need $\cat{C}$ to be braided, allowing more examples such as Example~\ref{ex:quantale} and Example~\ref{ex:idempotentmonad} above. Additionally, several conditions simplify; for example, central idempotents do not require firmness. 
The next definition exhibits a condition that holds very often, under which the subunits and central idempotents coincide.

\begin{definition}
  Call a monoidal category \emph{bilinear} when morphisms $f,g \colon A \to B \otimes U$ for a central idempotent $u$ are equal if (and only if) $f \otimes U = g \otimes U$.
\end{definition}

\begin{lemma}
  If a braided monoidal category $\cat{C}$ is bilinear, then $\ISub(\cat{C}) = \ZI(\cat{C})$.
\end{lemma}
\begin{proof}
  Let $u \colon U \to I$ be a central idempotent. 
  Suppose that $u \circ f = u \circ g$ for $f,g \colon A \to U$.
  Then $(u \otimes U) \circ (f \otimes U) = (u \otimes U) \circ (g \otimes U)$, and hence $f \otimes U = g \otimes U$. 
  But bilinearity now implies $f=g$.
\end{proof}

Any posetal category is bilinear, including semilattices, frames, and quantales.
Any sheaf category is bilinear, because there is only one function into the empty set (namely the empty function, which has the empty set as domain). More generally, the following lemma shows that any cartesian category is bilinear.

\begin{lemma}
  Any cartesian category is bilinear. 
\end{lemma}
\begin{proof}
  Observe that a cartesian category satisfies the following property:
  \[\begin{tikzpicture}[xscale=3,yscale=1.5]
  \node (A) at (0,0) {$A$};
  \node (BS) at (2,0) {$B \otimes U$};
  \node (AS) at (1,1) {$A \otimes U$};
  \draw[->] (A) to node[below]{$f$} (BS);
  \draw[->] (AS) to node[above right]{$f \otimes u$} (BS);
  \draw[>->] ([xshift=-.5mm]AS.-135) to node[above left]{$A \otimes u$} ([xshift=-.5mm]A.45);
  \draw[<<-,dashed] ([xshift=.5mm]AS.-135) to node[below right]{$e$} ([xshift=.5mm]A.45);
  \draw[->] (AS) to[out=110,in=70,looseness=10] node[right=1mm,font=\tiny]{$A \otimes U$} (AS);
  \end{tikzpicture}\]   
  Namely, given $f$, take $e=\langle A, f \circ \pi_2 \rangle$.
  In fact, notice that $e$ is independent of $f$, because $f \circ \pi_2$ is the unique morphism from $A$ into the subterminal object $U$ by Lemma~\ref{lem:zi:Cu}.
  Now if $f,g \colon A \to B \otimes U$ satisfy $f \otimes U = g \otimes U$, then $f = (f \otimes u) \circ e = (g \otimes u) \circ e = g$.
\end{proof}

The next two lemmas show that categories of (Hilbert) modules (over nonunital rings) are bilinear too.

\begin{lemma}
  If $R$ is a nondegenerate firm commutative ring, then the category $\cat{Mod}_R$ of nondegenerate firm $R$-modules and linear maps is bilinear.
\end{lemma}
\begin{proof}
  By Example~\ref{ex:modules}, a central idempotent in $\cat{Mod}_R$ is a morphism $u \colon U \to S$ for an ideal $S\subseteq R$ that is nondegenerate, firm, and idempotent~\cite[Proposition~3.11]{enriquemolinerheunentull:tensortopology}.
  If $B$ is an $R$-module, then $B \otimes U$ is the submodule $\{\varphi \cdot b \mid \varphi \in S, b \in B \}$.
  Suppose that $f,g \colon A \to B \otimes U$ are linear maps, such that $f \otimes U = g \otimes U$.
  The latter means that $f(u(\varphi) \cdot a)=g(u(\varphi) \cdot a)$ for any $a \in A$ and $\varphi \in U$.
  Let $a \in A$. Because $U \otimes U \simeq U$, then $f(a) \in B \otimes U$ corresponds to $u(\varphi) \cdot f(a) \in B \otimes U \otimes U$ for some $\varphi \in U$. But $u(\varphi) \cdot f(a) = f(u(\varphi) \cdot a) = g(u(\varphi) \cdot a) = u(\varphi) \cdot g(a)$, so $f(a)=g(a)$. Thus $f=g$.
\end{proof}

\begin{lemma}
  If $X$ is a locally compact Hausdorff space, $\cat{Hilb}_{C_0(X)}$ is bilinear.
\end{lemma}
\begin{proof}
  As in Example~\ref{ex:hilbert}, a central idempotent $U$ in $\cat{C}=\cat{Hilb}_{C_0(X)}$ is of the form $U=\{\varphi \in C_0(X) \mid \varphi(X \setminus U)=0\}$ for an open set $U \subseteq X$~\cite[Proposition~3.16]{enriquemolinerheunentull:tensortopology}.
  If $B$ is a Hilbert $C_0(X)$-module, then $B \otimes U$ is the submodule $\{b \in B \mid \langle b \mid b \rangle (X \setminus U)=0\}$.
  Suppose that $f,g \colon A \to B \otimes U$ are bounded $C_0(X)$-linear maps, and that $f \otimes U = g \otimes U$.
  The latter means that $f(a)=g(a)$ for $a \in A$ with $\langle a \mid a \rangle (X \setminus U)=0$.
  Now pick a net $\varphi_n \in C(X)$ such that $\varphi_n(X \setminus U)=0$ but $\lim_n \varphi_n(x)=1$ for every $x \in U$.
  Then $\varphi_n \cdot f(a) = f(\varphi_n \cdot a) = g(\varphi_n \cdot a) = \varphi_n \cdot g(a)$ for every $a \in A$.
  Moreover, $\lim_n \langle (\varphi_n-1)f(a) \mid (\varphi_n-1)g(a)\rangle (x)$ vanishes for every $x \in X$: for if $x \not\in U$ then the limit equals $\lim \langle f(a) \mid f(a) \rangle(x)=0$; but if $x \in U$ then $\varphi_n-1$ tends to zero and so the limit vanishes too.
  Thus $f(a) = \lim_n \varphi_n \cdot f(a) = \lim_n \varphi_n \cdot g(a) = g(a)$ for every $a$, that is, $f=g$.
\end{proof}

\end{document}